\documentclass[smallextended]{svjour3}       

\smartqed  
\usepackage{graphicx}
\usepackage{amsmath}
\usepackage{amssymb}
\usepackage{url}
\usepackage{hyperref}
\usepackage{algorithm}
\usepackage{algpseudocode}

\usepackage{ifpdf}
\usepackage{url}
\usepackage{hyperref}
\usepackage{algorithm}
\usepackage[usenames,dvipsnames,table]{xcolor}
\usepackage{xfrac}
\usepackage{comment}

\usepackage[ruled,vlined, linesnumbered, algo2e]{algorithm2e}
\usepackage{algorithmicx}
\usepackage{algpseudocode}
\usepackage{algcompatible}
\usepackage{amssymb,makeidx,mathrsfs}
\usepackage{fixltx2e}
\usepackage{booktabs} 
\usepackage[letterpaper, margin=1.3in]{geometry}

\newcommand{\ud}{\mathrm{d}} 
\newcommand{\Id}{\mathbb{I}}
\newcommand{\Pa}{\mathcal{P}}
\newcommand{\R}{\mathbb{R}}
\newcommand{\Rext}{\R\cup\{+\infty\}}

\newcommand{\abs}[1]{\left\vert#1\right\vert}
\newcommand{\set}[1]{\left\{#1\right\}}
\newcommand{\norm}[1]{\left\Vert#1\right\Vert}
\newcommand{\norms}[1]{\Vert#1\Vert}

\newcommand{\Eproof}{\hfill $\square$}
\newcommand{\prox}{\mathrm{prox}}

\newcommand{\argmin}{\mathrm{arg}\!\min}
\newcommand{\dom}[1]{\mathrm{dom}(#1)}
\newcommand{\zero}[1]{\boldsymbol{0}}

\newcommand{\xb}{x}
\newcommand{\xopt}{x^{\star}}

\newcommand{\yb}{y}
\newcommand{\eb}{e}
\newcommand{\zb}{z}
\newcommand{\wb}{w}
\newcommand{\ub}{u}             
\newcommand{\vb}{v}

\newcommand{\sbb}{s}

\newcommand{\ab}{a}
\newcommand{\bb}{b}
\newcommand{\cb}{c}
\newcommand{\db}{d}

\newcommand{\pb}{p}
\newcommand{\Ab}{A}
\newcommand{\Bb}{B}

\newcommand{\Gb}{G}
\newcommand{\Hb}{H}
\newcommand{\Kb}{K}
\newcommand{\Qb}{Q}

\newcommand{\Wb}{W}
\newcommand{\Ac}{\mathcal{A}}

\newcommand{\Xc}{\mathcal{X}}

\newcommand{\Sc}{\mathcal{S}}

\newcommand{\Fc}{\mathcal{F}}
\newcommand{\Uc}{\mathcal{U}}

\newcommand{\Nc}{\mathcal{N}}

\newcommand{\pntdir}{n_{\mathrm{pnt}}}
\newcommand{\ntdir}{n_{\mathrm{nt}}}
\newcommand{\iprod}[1]{\left\langle #1\right\rangle}
\newcommand{\iprods}[1]{\langle #1\rangle}

\renewcommand{\vec}[1]{\mathrm{vec}(#1)}

\newcommand{\ri}[1]{\mathrm{ri}\left(#1\right)}

\newcommand{\gsc}{\textit{generalized self-concordant} }

\newcommand{\SFunc}[2]{\mathcal{C}^{#1}\left(#2\right)}
\newcommand{\needcheck}[1]{#1}

\newcommand{\beforesec}{\vspace{-3.25ex}}
\newcommand{\aftersec}{\vspace{-2.25ex}}
\newcommand{\beforesubsec}{\vspace{-3ex}}
\newcommand{\aftersubsec}{\vspace{-2ex}}

\newcommand{\beforepar}{\vspace{-2.25ex}}

\begin{document}

\title{Generalized Self-Concordant Functions: A Recipe for Newton-Type Methods}

\titlerunning{Generalized Self-Concordant Functions: A Recipe for Newton-Type Methods}        

\author{Tianxiao Sun \and Quoc Tran-Dinh$^{\ast}$}

\authorrunning{T. Sun and Q. Tran-Dinh} 

\institute{
$^{\ast}$Corresponding author (\url{quoctd@email.unc.edu})\vspace{1ex}\newline
Tianxiao Sun \and Quoc Tran-Dinh \at
		Department of Statistics and Operations Research, University of North Carolina at Chapel Hill (UNC)\\
		 318 Hanes Hall, CB\# 3260, UNC Chapel Hill, NC 27599-3260\\ 
		Email: \texttt{\{tianxias, quoctd\}@email.unc.edu}
}

\date{Received: date / Accepted: date}

\maketitle

\vspace{-2ex}
\begin{abstract}
We study the smooth structure of convex functions by generalizing a powerful concept so-called \textit{self-concordance} introduced by Nesterov and Nemirovskii in the early 1990s to a broader class of convex functions which we call \textit{generalized self-concordant functions}.
This notion allows us to develop a unified framework for designing Newton-type methods to solve convex optimization problems.
The proposed theory provides a mathematical tool to analyze both local and global convergence  of Newton-type methods without imposing unverifiable assumptions as long as the underlying functionals fall into our class of generalized self-concordant functions.
First, we introduce the class of generalized self-concordant functions which covers the class of standard self-concordant functions as a special case. 
Next, we establish several properties and key estimates of this function class which can be used to design numerical methods.
Then, we apply this theory to develop  several Newton-type  methods for solving a class of smooth convex optimization problems involving generalized self-concordant functions.
We provide an explicit step-size for  a damped-step Newton-type scheme which can guarantee a global convergence  without performing any globalization strategy. 
We also prove a local quadratic convergence of this method and its full-step variant without requiring the Lipschitz continuity of the objective Hessian mapping. 
Then, we extend our result to develop proximal Newton-type methods for a class of composite convex minimization problems involving generalized self-concordant functions.
We also achieve both global and local convergence without additional assumptions.
Finally, we verify our theoretical results via several numerical examples, and compare them with existing methods.
\end{abstract}



\beforesec
\section{Introduction}\label{sec:intro}
\aftersec
The Newton method is a classical numerical scheme for solving systems of nonlinear equations and smooth optimization \cite{Nocedal2006,Ortega2000}. 
However, there are at least two reasons that prevent the use of such methods from solving large-scale problems.
Firstly, while these methods often have a fast local convergence rate which can be up to a quadratic rate, their global convergence has not been well-understood \cite{Nesterov2006}.
In practice, one can use a damped-step scheme utilizing the Lipschitz constant of the objective derivatives to compute a suitable step-size as often seen in gradient-type methods, or incorporate the algorithm with a globalization strategy such as line-search, trust-region, or filter to guarantee a descent property \cite{Nocedal2006}.
Both strategies allow us to prove a global convergence of the underlying Newton-type method in some sense.
Unfortunately, in practice, there exist several problems whose objective function does not have global Lipschitz gradient or Hessian such as logarithmic or reciprocal functions. 
This class of problems does not provide us some uniform bounds to obtain a  constant step-size in optimization algorithms. 
On the other hand, using a globalization strategy for determining step-sizes often requires centralized computation such as function evaluations which prevent us from using distributed computation and stochastic descent methods.
Secondly, Newton algorithms are second-order methods which often require a high per-iteration complexity due to the operations on the Hessian mapping of the objective function or its approximations.
In addition, these methods require the underlying functionals to be smooth up to a given  smoothness levels which does not often hold in many practical models.

\beforepar
\paragraph{\textbf{Motivation:}}
In recent years, there has been a great interest in Newton-type methods for solving convex optimization problems and monotone equations due to the development of new techniques and mathematical tools in optimization, machine learning, and randomized algorithms \cite{Becker2012a,byrd2016stochastic,Deuflhard2006,erdogdu2015convergence,Lee2014,Nesterov2006b,Nesterov2008b,pilanci2015newton,polyak2009regularized,Roosta-Khorasani2016,roosta2016sub,Tran-Dinh2013b}.
Several combinations of Newton-type methods and other techniques such as proximal operators \cite{Bonnans1994a}, cubic regularization \cite{Nesterov2006b}, gradient regularization \cite{polyak2009regularized}, randomized algorithms such as sketching \cite{pilanci2015newton}, subsampling \cite{erdogdu2015convergence}, and fast eigen-decomposition \cite{halko2009finding} have opened  up a new research direction and  attracted a great attention in solving nonsmooth and large-scale problems.
Hitherto, research in this direction remains focusing on specific classes of problems where standard assumptions such as nonsingularity and Hessian Lipschitz  continuity are preserved. 
However, such assumptions do not hold for many other examples as shown in \cite{Tran-Dinh2013a}. 
Moreover, if they are satisfied, then we often get a lower bound of possible step-sizes for our algorithm which may lead to a poor performance, especially in large-scale problems.

In the seminal work \cite{Nesterov1994}, Nesterov and Nemirovskii showed that the class of log-barriers does not satisfy the standard assumptions of the Newton method if the solution of the underlying problem is closed to the boundary of the barrier function domain.
They introduced a powerful concept called ``self-concordance'' to overcome this drawback and developed new Newton schemes to achieve global and local convergence without requiring any additional assumption, or a globalization strategy.
While the self-concordance notion was initially invented to study interior-point methods, it is less well-known in other communities.
Recent works \cite{Bach2009,cohen2017matrix,monteiro2015hybrid,Tran-Dinh2013a,TranDinh2016c,zhang2015disco} have popularized this concept to solve other problems arising from machine learning, statistics, image processing, scientific computing, and variational inequalities.

\beforepar
\paragraph{\textbf{Our goals:}}
In this paper, motivated by \cite{Bach2009,TranDinh2014d,zhang2015disco}, we aim at generalizing the self-concordance concept in \cite{Nesterov1994} to a broader class of smooth and convex functions.
To illustrate our idea, we consider a univariate smooth and convex function $\varphi : \R\to\R$. 
If $\varphi$ satisfies the inequality $\vert\varphi'''(t)\vert \leq M_{\varphi}\varphi''(t)^{3/2}$ for all $t$ in the domain of $\varphi$ and for a given constant $M_{\varphi}\geq 0$, then we say that $\varphi$ is self-concordant (in Nesterov and Nemirovskii's sense \cite{Nesterov1994}).
We instead generalize this inequality to 
\begin{equation}\label{eq:one_var_gsc_def}
\vert\varphi'''(t)\vert \leq M_{\varphi}\varphi''(t)^{\frac{\nu}{2}},
\end{equation}
 for all $t$ in the domain of $\varphi$, and for given constants $\nu > 0$ and $M_{\varphi}\geq 0$.

We emphasize that generalizing from univariate to multivariate functions in the standard self-concordant case (i.e., $\nu = 3$) \cite{Nesterov1994} preserves several important properties including the multilinear symmetry \cite[Lemma 4.1.2]{Nesterov2004}, while, unfortunately, they do not hold for the case $\nu \neq 3$. 
Therefore, we modify the definition in \cite{Nesterov1994} to overcome this drawback.
Note that a similar idea has been also studied in \cite{Bach2009,TranDinh2014d} for a class of logistic-type functions. 
Nevertheless, the definition using in these papers is limited, and still creates certain difficulty for developing further theory in general cases.

Our second goal is to develop a unified mechanism to analyze convergence (including global and local convergence) of the following Newton-type scheme:
\begin{equation}\label{eq:newton_scheme}
\xb^{k+1} := \xb^k - s_kF'(\xb^k)^{-1}F(\xb^k),
\end{equation}
where $F$ can  be represented as the right-hand side of a smooth monotone equation $F(x) = 0$, or the optimality condition of a convex optimization or a convex-concave saddle-point problem, $F'$ is the Jacobian map of $F$, and $s_k\in (0, 1]$ is a given step-size.
Despite the Newton scheme \eqref{eq:newton_scheme} is invariant  to a change of variables \cite{Deuflhard2006}, its convergence property relies on the growth of the Hessian mapping along the Newton iterative process. 
In classical settings, the Lipschitz continuity and the non-degeneracy of the Hessian mapping in a neighborhood of a given solution are key assumptions to achieve local quadratic convergence rate \cite{Deuflhard2006}.
These assumptions have been considered to be standard, but they are often very difficult to check in practice, especially the second requirement.
A natural idea is to classify the functionals of the underlying problem into a known class of functions to choose a suitable method for minimizing it. 
While first-order methods for convex optimization essentially rely on the Lipschitz gradient continuity, Newton schemes usually use the Lipschitz continuity of the Hessian mapping and its non-degeneracy to obtain a well-defined Newton direction as we have mentioned.
For self-concordant functions, the second condition automatically holds, but the first assumption fails to satisfy. 
However, both full-step and damped-step Newton methods still work in this case by appropriately choosing a suitable metric. 
This situation has been observed and standard assumptions have been modified in different directions to still guarantee convergence of Newton-type methods, see \cite{Deuflhard2006} for an intensive study of generic Newton-type methods, and \cite{Nesterov1994,Nesterov2004} for the self-concordant function class.

\beforepar
\paragraph{\textbf{Our approach:}}
We attempt to develop some background theory for a broad class of smooth and convex functions under the structure \eqref{eq:one_var_gsc_def}. 
By adopting the local norm defined via the Hessian mapping of such a convex function from \cite{Nesterov1994}, we can prove some lower and upper bound estimates for the local norm distance between two points in the domain as well as for the growth of the Hessian mapping. 
Together with this background theory, we also identify a class of functions using in generalized linear models \cite{mccullagh1989generalized,nelder1972generalized} as well as in empirical risk minimization \cite{vapnik1998statistical} that falls into our generalized self-concordance class for many well-known loss-type functions as listed in Table \ref{tbl:examples}.

Applying our generalized self-concordant theory, we  develop a class of Newton-type methods to solve the following composite convex minimization problem:
\begin{equation}\label{eq:composite_cvx0}
F^{\star}:=\min_{\xb\in\R^p}\Big\{ F(\xb):=f(\xb)+g(\xb) \Big\},
\end{equation}
where $f$ is a generalized self-concordant function in our context, and $g$ is a proper, closed, and convex function that can be referred to as a regularization term.
We consider two cases. The first case is a non-composite convex problem in which $g$ is vanished (i.e., $g = 0$). In the second case, we assume that $g$ is equipped with a ``tractably'' proximal operator (see \eqref{eq:def_gprox} for the definition).

\beforepar
\paragraph{\textbf{Our contribution:}}
To this end, our main contribution can be summarized as follows.
\begin{itemize}
\item[$(\mathrm{a})$] We generalize the self-concordant notion in \cite{Nesterov2004} to a more broader class of smooth convex functions which we call generalized self-concordance.
We identify several loss-type functions that can be cast into our generalized self-concordant class. 
We also prove several fundamental properties and show that the sum and linear transformation of \gsc functions are \gsc for a given range of $\nu$ or under suitable assumptions.

\vspace{0.5ex}
\item[$(\mathrm{b})$]
We develop lower and upper bounds on the Hessian mapping, the gradient mapping, and the function values for generalized self-concordant functions.
These estimates are key to analyze several numerical optimization methods including Newton-type methods.

\vspace{0.5ex}
\item[$(\mathrm{c})$]
We propose a class of Newton methods including full-step and damped-step schemes to minimize a generalized self-concordant function. 
We explicitly show how to choose a suitable step-size to guarantee a descent direction in the damped-step scheme, and prove a local quadratic convergence for both the damped-step and the full-step schemes using a suitable metric.

\vspace{0.5ex}
\item[$(\mathrm{d})$]
We also extend our Newton schemes to handle the composite setting \eqref{eq:composite_cvx0}. 
We develop both full-step and damped-step proximal Newton methods to solve this problem and provide a rigorous theoretical convergence guarantee in both local and global sense.

\vspace{0.5ex}
\item[$(\mathrm{e})$] We also study a quasi-Newton variant of our Newton scheme to minimize a generalized self-concordant function. Under a modification of the well-known Dennis-Mor\'{e} condition \cite{Dennis1974} or a BFGS update, we show that our quasi-Newton method locally converges at a superlinear rate to the solution of the underlying problem.
\end{itemize}

Let us emphasize the following aspects of our contribution.
Firstly, we observe that the self-concordance notion is a powerful concept and has widely been  used in interior-point methods as well as in other optimization schemes \cite{He2016,Lu2016a,Tran-Dinh2013a,zhang2015disco}, generalizing it to a broader class of smooth convex functions can substantially cover a number of new applications or can develop new methods for solving old problems including logistic and multimonomial logistic regression, optimization involving exponential objectives, and distance-weighted discrimination problems in support vector machine (see Table \ref{tbl:examples} below).
Secondly, verifying theoretical assumptions for convergence guarantees of a Newton method is not trivial, our theory allows one to classify the underlying functions into different subclasses by using different parameters $\nu$ and $M_{\varphi}$ in order to choose suitable algorithms to solve the corresponding optimization problem.
Thirdly, the theory developed in this paper can potentially apply to other optimization methods such as gradient-type, sketching and sub-sampling Newton, and Frank-Wolfe's algorithms as done in the literature \cite{odor2016frank,pilanci2015newton,Roosta-Khorasani2016,roosta2016sub,Tran-Dinh2013a}. 
Finally, our generalization also shows that it is possible to impose additional structure such as self-concordant barrier to develop path-following scheme or interior-point-type methods for solving a subclass of composite convex minimization problems of the form \eqref{eq:composite_cvx0}.
We believe that our theory is not limited to convex optimization, but can be extended to solve convex-concave saddle-point problems, and monotone equations/inclusions involving generalized self-concordant functions \cite{TranDinh2016c}.

\beforepar
\paragraph{\textbf{Summary of generalized self-concordant properties:}}
For our reference convenience, we provide a short summary on the main properties of generalized self-concordant (gsc) functions in Table \ref{tbl:summary}.
\begin{table}[hpt!]
\vspace{-5ex}
\begin{center}
\caption{A summary of generalized self-concordant properties}\label{tbl:summary}
\rowcolors{2}{white}{black!15!white}
\vspace{-2ex}
\begin{tabular}{ p{2.9cm} | p{5.8cm} | p{5cm} }\toprule
\multicolumn{1}{c|}{\textbf{Result}} & \multicolumn{1}{c|}{\textbf{Property}} & \multicolumn{1}{c}{\textbf{Range of $\nu$}}  \\ \midrule
Definitions~\ref{de:gsc_def0} and \ref{de:gsc_def} & definitions of gsc functions & $\nu > 0$ \\ \midrule
Proposition~\ref{pro:sum_rule} & sum of gsc functions & $\nu \geq 2$ \\ \midrule
Proposition~\ref{pro:affine_transform} & affine transformation of gsc functions with $\Ac(x) = Ax + b$ &   $\nu \in (0, 3]$ for general $A$\newline $\nu > 3$ for over-completed $A$ \\ \midrule
Proposition~\ref{pro:Hessian_nondegenerate}(a) & non-degenerate property  &  $\nu \geq 2$ \\ \midrule
Proposition~\ref{pro:Hessian_nondegenerate}(b) & unboundedness  &   $\nu > 0$ \\ \midrule
Proposition~\ref{pro:scvx_lips_gsc}(a) & gsc and strong convexity  &  $\nu \in (0, 3]$ \\ \midrule
Proposition~\ref{pro:scvx_lips_gsc}(b) & gsc and Lipschitz gradient continuity &  $\nu \geq 2$ \\ \midrule
Proposition~\ref{pro:conjugate} & if $f^{\ast}$ is the conjugate of a gsc function\newline  $f$, then $\nu + \nu_{\ast} = 6$ &  $\nu_{\ast} \in (0, 6)$ if $p = 1$ (univariate)\newline $\nu_{\ast} \in [3, 6)$ if $p > 1$ (multivariate)\\ \midrule
Propositions \ref{pro:xy_bounds}, \ref{pro:hessian_bounds}, \ref{pro:gradient_bound1}, and \ref{pro:fx_bound1} & local norm, Hessian, gradient, and function value bounds & $\nu \geq 2$ \\ 
\bottomrule
\end{tabular}
\vspace{-5ex}
\end{center}
\end{table}
Although several results hold for a different range of $\nu$, the complete theory only holds for $\nu \in [2, 3]$.
However, this is sufficient to cover two important cases: $\nu = 2$ in \cite{Bach2009,Bach2013a} and $\nu = 3$ in \cite{Nesterov1994}.

\beforepar
\paragraph{\textbf{Related work:}}
Since the self-concordance concept was introduced in 1990s \cite{Nesterov1994}, its first extension is perhaps proposed by \cite{Bach2009} for a class of logistic regression.
In \cite{TranDinh2014d}, the authors extended \cite{Bach2009} to study proximal Newton method for logistic, multinomial logistic, and exponential loss functions.
By augmenting a strongly convex regularizer, Zhang and Lin in \cite{zhang2015disco} showed that the regularized logistic loss function is indeed standard self-concordant.
In \cite{Bach2013a} Bach continued exploiting his result in \cite{Bach2009} to show that the averaging stochastic gradient method can achieve the same best-known convergence rate as in strongly convex case without adding a regularizer.
In \cite{Tran-Dinh2013a}, the authors exploited standard self-concordance theory in \cite{Nesterov1994} to develop several classes of optimization algorithms including proximal Newton, proximal quasi-Newton, and proximal gradient methods to solve composite convex minimization problems. 
In \cite{Lu2016a}, Lu extended \cite{Tran-Dinh2013a} to study randomized block coordinate descent methods. 
In a recent paper \cite{gao2016quasi}, Gao and Goldfarb investigated quasi-Newton methods for self-concordant problems.
As another example, \cite{peng2009self} proposed an alternative to the standard self-concordance, called self-regularity. 
The authors applied this theory to develop a new paradigm for interior-point methods.
The theory developed in this paper, on the one hand, is a generalization of the well-known self-concordance notion developed in \cite{Nesterov1994};
on the other hand, it also covers the work in \cite{Bach2009,Tran-Dinh2013b,zhang2015disco} as specific examples.
Several concrete applications and extensions of self-concordance notion can also be found in the literature including \cite{He2016,Kyrillidis2014,odor2016frank,peng2009self}.
Recently, \cite{cohen2017matrix} exploited smooth structures of exponential functions to design interior-point methods for solving two fundamental problems in scientific computing called matrix scaling and balancing.

\beforepar
\paragraph{\textbf{Paper organization:}} 
The rest of this paper is organized as follows.
Section \ref{sec:gsc_background} develops the foundation theory for our generalized self-concordant functions including definitions, examples, basic properties, Fenchel's conjugate, smoothing technique, and key bounds.
Section~\ref{sec:gsc_min} is devoted to studying full-step and damped-step Newton schemes to minimize a generalized self-concordant function including their global and local convergence guarantees.
Section~\ref{sec:gsc_composite_min} considers to the composite setting \eqref{eq:composite_cvx0} and studies proximal Newton-type methods, and investigates their convergence guarantees.
Section~\ref{sec:quasi_newton} deals with a quasi-Newton scheme for solving the noncomposite problem of \eqref{eq:composite_cvx0}.
Numerical examples are provided in Section~\ref{sec:num_experiments} to illustrate advantages of our theory.
Finally, for clarity of presentation, several technical results and proofs are moved to the appendix.

\vspace{-0.25ex}
\beforesec
\section{Theory of generalized self-concordant functions}\label{sec:gsc_background}
\aftersec
\vspace{-0.25ex}
We generalize the class of self-concordant functions introduced by Nesterov and Nemirovskii in \cite{Nesterov2004} to a broader class of smooth and convex functions. 
We identify several examples of such functions.
Then, we develop several properties of this function class by utilizing our new definitions.

\beforepar
\paragraph{\textbf{Notation:}}
Given a proper, closed, and convex function $f:\R^p\to\Rext$, we denote by $\dom{f} := \set{\xb\in\R^p \mid f(\xb) <+\infty}$ the domain of $f$, and by $\partial{f}(\xb) := \big\{\wb\in\R^p  \mid f(\yb) \geq f(\xb) + \iprods{\wb, \yb - \xb},~\forall\yb\in\dom{f} \big\}$ the subdifferential of $f$ at $\xb\in\dom{f}$.
We use $\mathcal{C}^3(\dom{f})$ to denote the class of three times continuously differentiable functions on its open domain $\dom{f}$.
We denote by $\nabla{f}$ its gradient map, by $\nabla^2{f}$ its Hessian map, and by $\nabla^3{f}$  its third-order derivative.
For a twice continuously differentiable convex function $f$, $\nabla^2{f}$ is symmetric positive semidefinite, and can be written as $\nabla^2{f}(\cdot) \succeq 0$.
If it is positive definite, then we write $\nabla^2{f}(\cdot) \succ 0$. 

Let $\R_+$ and $\R_{++}$ denote the sets of nonnegative and positive real numbers, respectively.
We use $\Sc^p_{+}$ and $\Sc^p_{++}$ to denote the sets of symmetric positive semidefinite and symmetric positive definite matrices of the size $p\times p$, respectively.
Given a $p\times p$ matrix $\Hb\succ 0$, we define a weighted norm with respect to $\Hb$ as $\norm{\ub}_{\Hb} := \iprod{\Hb\ub,\ub}^{1/2}$ for $\ub\in\R^p$. 
The corresponding dual norm is $\norm{\vb}_{\Hb}^{\ast} := \iprod{\Hb^{-1}\vb,\vb}^{1/2}$. 
If $\Hb = \Id$, the identity matrix, then $\norm{\ub}_{\Hb}=\norm{\ub}_{\Hb}^{*} =\norm{\ub}_2$, where $\norm{\cdot}_2$ is the standard Euclidean norm. 
Note that $\Vert\cdot\Vert_2^{\ast} = \Vert\cdot\Vert_2$.

We say that $f$ is strongly convex with the strong convexity parameter $\mu_f\geq 0$ if $f(\cdot) - \frac{\mu_f}{2}\norm{\cdot}^2$ is convex.
We also say that $f$ has Lipschitz gradient if $\nabla{f}$ is Lipschitz continuous with the Lipschitz constant $L_f \in [0, +\infty)$, i.e., $\Vert\nabla{f}(x) - \nabla{f}(y)\Vert^{\ast} \leq L_f\norm{x - y}$ for all $x, y\in\dom{f}$.

For $f\in\mathcal{C}^3(\dom{f})$, if $\nabla^2 f(\xb)\succ 0$ at a given $\xb\in\dom{f}$, then we define a local norm  $\norm{\ub}_{\xb} := \iprods{\nabla^2{f}(\xb)\ub, \ub}^{1/2}$ as a weighted norm of $\ub$ with respect to $\nabla^2{f}(\xb)$. 
The corresponding dual norm $\norm{\vb}^{*}_{\xb}$, is defined as $\norm{\vb}^{*}_{\xb}:=\max\set{ \iprods{\vb, \ub} \mid \norm{\ub}_{\xb}\leq 1}=\iprod{ \nabla^2 f(\xb)^{-1}\vb,\vb}^{1/2}$ for $\vb\in\R^p$.

\beforesubsec
\subsection{\bf Univariate generalized self-concordant functions}\label{def_norm}
\aftersubsec
Let $\varphi : \R\to\R$ be a three times continuously differentiable function on the open domain $\dom{\varphi}$. 
Then, we write $\varphi\in\SFunc{3}{\dom{\varphi}}$.
In this case, $\varphi$ is convex if and only if $\varphi''(t) \geq 0$ for all $t\in\dom{\varphi}$.
We introduce the following definition.

\begin{definition}\label{de:gsc_def0}
\textit{
Let $\varphi : \R\to\R$ be a $\SFunc{3}{\dom{\varphi}}$ and univariate function with open domain $\dom{\varphi}$, and $\nu > 0$ and $M_{\varphi} \geq 0$ be two constants. 
We say that $\varphi$ is $(M_{\varphi},\nu)$-generalized self-concordant if
\begin{equation}\label{eq:gsc_def0}
\vert\varphi{'''}(t)\vert \leq M_{\varphi}\varphi{''}(t)^{\frac{\nu}{2}},~~~\forall t\in \dom{\varphi}.
\end{equation}
}
\end{definition}
The inequality \eqref{eq:gsc_def0} also indicates that $\varphi''(t) \geq 0$ for all $t\in\dom{f}$. Hence, $\varphi$ is convex.
Clearly, if $\varphi(t) = \frac{a}{2}t^2 + bt$ for any constants $a\geq 0$ and $b\in\R$, we have $\varphi''(t) = a$ and $\varphi'''(t) = 0$. 
The inequality \eqref{eq:gsc_def0} is automatically satisfied for any $\nu > 0$ and $M_{\varphi} \geq 0$. The smallest value of $M_{\varphi}$ is zero. 
Hence, any convex quadratic function is $(0, \nu)$-\gsc for any $\nu > 0$.
While \eqref{eq:gsc_def0} holds for any other constant $\hat{M}_{\varphi} \geq M_{\varphi}$, we often require that $M_{\varphi}$ is the smallest constant satisfying \eqref{eq:gsc_def0}.

\begin{example}
Let us now provide some common examples satisfying Definition \ref{de:gsc_def0}.
\begin{itemize}
\item[(a)]\textit{Standard self-concordant functions:} 
If we choose $\nu = 3$, then \eqref{eq:gsc_def0} becomes $\vert\varphi'''(t)\vert \leq M_{\varphi}\varphi''(t)^{3/2}$ which is the standard self-concordant functions in $\R$ introduced in \cite{Nesterov1994}.

\vspace{0.5ex}
\item[(b)]\textit{Logistic functions:} 
In \cite{Bach2009}, Bach modified the standard self-concordant inequality in \cite{Nesterov1994} to obtain $\abs{\varphi'''(t)} \leq M_{\varphi}\varphi''(t)$, and  showed that the well-known logistic loss $\varphi(t) := \log(1 + e^{-t})$ satisfies this definition.
In \cite{TranDinh2014d} the authors also exploited this definition, and developed a class of first-order and second-order methods to solve composite convex minimization problems.
Hence, $\varphi(t) := \log(1 + e^{-t})$ is a generalized self-concordant function with $M_{\varphi} = 1$ and $\nu = 2$.

\vspace{0.5ex}
\item[(c)]\textit{Exponential functions:} 
The exponential function $\varphi(t) := e^{-t}$ also belongs to \eqref{eq:gsc_def0} with $M_{\varphi} = 1$ and $\nu = 2$.
This function is often used, e.g., in Ada-boost \cite{lafferty2002boosting}, or in matrix scaling \cite{cohen2017matrix}.

\vspace{0.5ex}
\item[(d)]\textit{Distance-weighted discrimination} (DWD): 
We consider a more general function $\varphi(t) := \frac{1}{t^q}$ on $\dom{\varphi} = \R_{++}$ and $q \geq 1$ studied in \cite{marron2007distance} for  DWD using in support vector machine.
As shown in Table \ref{tbl:examples}, this function satisfies Definition \ref{de:gsc_def0} with $M_{\varphi} = \frac{q+2}{\sqrt[(q+2)]{q(q+1)}}$ and  $\nu = \tfrac{2(q+3)}{q+2} \in (2, 3)$.

\vspace{0.5ex}
\item[(e)]\textit{Entropy function:} 
We consider the well-known entropy function $\varphi(t) := t\ln(t)$ for $t > 0$. We can easily show that $\vert\varphi'''(t)\vert = \tfrac{1}{t^2} = \varphi''(t)^2$. 
Hence, it is generalized self-concordant with $\nu = 4$ and $M_{\varphi} = 1$ in the sense of Definition \ref{de:gsc_def0}.

\vspace{0.5ex}
\item[(f)]\textit{Arcsine distribution:} 
We consider the function $\varphi(t) := \frac{1}{\sqrt{1 - t^2}}$ for $t \in (-1, 1)$. 
This function is convex and smooth. 
Moreover, we verify that it satisfies Definition \ref{de:gsc_def0} with $\nu = \frac{14}{5} \in (2, 3)$ and $M_{\varphi} = \frac{3\sqrt{495-105\sqrt{21}}}{(7-\sqrt{21})^{7/5}} < 3.25$.
We can generalize this function to $\varphi(t) := \left[(t-a)(b-t)\right]^{-q}$ for $t\in (a, b)$, where $a < b$ and $q > 0$.
Then, we can show that $\nu = \frac{2(q+3)}{q+2} \in (2, 3)$.

\vspace{0.5ex}
\item[(g)]\textit{Robust Regression:} 
Consider a monomial function $\varphi(t):=t^q$ for $q\in (1,2)$ studied in \cite{yang2016rsg} for robust regression using in statistics. 
Then, $M_{\varphi} = \frac{2-q}{\sqrt[(2-q)]{q(q-1)}}$ and $\nu = \frac{2(3-q)}{2-q}\in (4,+\infty)$.
\end{itemize}
\end{example}
As concrete examples, the following table, Table \ref{tbl:examples}, provides a non-exhaustive list of generalized self-concordant functions used in the literature.

\begin{table}[H]
\rowcolors{2}{white}{black!15!white}
\begin{center}
\vspace{-3ex}
\begin{scriptsize}
\caption{Examples of univariate generalized self-concordant functions ($\mathcal{F}^{1,1}_L$ means that $\nabla{\varphi}$ is Lipschitz continuous).}\label{tbl:examples}
\begin{tabular}{ l | l | c | c | l | l | l | l }\toprule
\multicolumn{1}{c|}{Function name} & \multicolumn{1}{c|}{Form of $\varphi(t)$} & $\nu$ & $M_f$ &  $\dom{\varphi}$ & \multicolumn{1}{c|}{Application} &  $\mathcal{F}_L^{1,1}$ & Reference \\ \midrule
Log-barrier & $-\ln(t)$ & $3$ & 2 & $\R_{++}$ & Poisson & no &  \cite{Boyd2004,Nesterov2004,Nesterov1994}\\ \midrule
Entropy-barrier & $t\ln(t) - \ln(t)$ & $3$ & 2 & $\R_{++}$ & Interior-point & no &  \cite{Nesterov2004}\\ \midrule
Logistic & $\ln(1 + e^{-t})$ & $2$ & 1 & $\R$ & Classification & yes &  \cite{Hosmer2005}\\ \midrule
Exponential & $e^{-t}$ & $2$ & 1 & $\R$ & AdaBoost, etc & no &   \cite{cohen2017matrix,lafferty2002boosting}\\ \midrule
Negative power & $t^{-q},~~(q > 0)$ & $\frac{2(q+3)}{q+2}$ & $\frac{q+2}{\sqrt[(q+2)]{q(q+1)}}$ & $\R_{++}$ & DWD & no &   \cite{marron2007distance}\\ \midrule
Arcsine distribution& $\frac{1}{\sqrt{1-t^2}}$ & $\frac{14}{5}$ & $< 3.25$ & $(-1, 1)$ & Random walks & no &   \cite{Goel2006}\\ \midrule
Positive power & $t^q,~~(q\in (1,2))$ & $\frac{2(3-q)}{2-q}$ & $\frac{2-q}{\sqrt[(2-q)]{q(q-1)}}$ & $\R_{+}$ & Regression & no & \cite{yang2016rsg} \\ \midrule
Entropy & $t\ln(t)$ & $4$ & 1 & $\R_{+}$ & KL divergence & no &   \cite{Boyd2004}\\  
\bottomrule
\end{tabular}
\end{scriptsize}
\end{center}
\vspace{-5ex}
\end{table}
\begin{remark}\label{re:link_to_zhang2015}
All examples given in Table~\ref{tbl:examples} fall into the case $\nu \geq 2$.
However, we note that Definition~\ref{de:gsc_def0} also covers \cite[Lemma 1]{zhang2015disco} as a special case when $\nu \in (0, 2)$.
Unfortunately, as we will see in what follows, it is unclear how to generalize several properties of generalized self-concordance from univariate  to multivariable functions for $\nu\in (0, 2)$, except for strongly convex functions. 
\end{remark}
Table~\ref{tbl:examples} only provides common generalized self-concordant functions using in practice.
However, it is possible to combine these functions to obtain mixture functions that preserve the generalized self-concordant inequality given in Definition~\ref{de:gsc_def0}. 
For instance, the barrier entropy $t\ln(t) - \ln(t)$ is a standard self-concordant function, and it is the sum of the entropy $t\ln(t)$ and the negative logarithmic function $-\log(t)$ which are generalized self-concordant with $\nu = 4$ and $\nu = 3$, respectively. 

\beforesubsec
\subsection{\bf Multivariate generalized self-concordant functions}
\aftersubsec
Let $f : \R^p\to\R$ be a $\mathcal{C}^3(\dom{f})$ smooth and convex function with open domain $\dom{f}$.
Given $\nabla^2f$ the Hessian of $f$, $\xb\in \dom{f}$, and $\ub,\vb\in\R^p$, we consider the function $\psi(t) := \iprods{\nabla^2{f}(\xb + t\vb)\ub, \ub}$.
Then, it is obvious to show that 
\begin{equation*} 
\psi'(t) := \iprods{\nabla^3f(\xb + t\vb)[\vb]\ub, \ub}.
\end{equation*}
for $t \in\R$ such that $\xb + t\vb \in \dom{f}$, where $\nabla^3f$ is the third-order derivative of $f$. 
It is clear that $\psi(0) = \iprods{\nabla^2{f}(\xb)\ub, \ub} = \norm{\ub}_{\xb}^2$.
By using the local norm, we generalize Definition \ref{de:gsc_def0} to multivariate functions $f : \R^p\to\R$  as follows.

\begin{definition}\label{de:gsc_def}
A $\mathcal{C}^3$-convex function $f : \R^p\to\R$ is said to be an $(M_f, \nu)$-generalized self-concordant function of the order $\nu > 0$ and the constant $M_f \geq 0$ if, for any $\xb\in\dom{f}$ and $\ub,\vb\in\R^p$, it holds
\begin{equation}\label{eq:gsc_def}
\abs{\iprods{\nabla^3f(\xb)[\vb]\ub, \ub}} \leq M_f\norm{\ub}_{\xb}^2\norm{\vb}_{\xb}^{\nu-2}\norm{\vb}_2^{3-\nu}.
\end{equation}
Here, we use a convention that $\frac{0}{0} = 0$ for the case $\nu < 2$ or $\nu > 3$.
We denote this class of functions by $\widetilde{\Fc}_{M_f,\nu}(\dom{f})$ (shortly, $\widetilde{\Fc}_{M_f,\nu}$ when $\dom{f}$ is explicitly defined).
\end{definition}
Let us consider the following two  extreme cases:
\begin{enumerate}
\item If $\nu = 2$,  \eqref{eq:gsc_def} leads to $\abs{\iprods{\nabla^3f(\xb)[\vb]\ub, \ub}} \leq M_f\norm{\ub}_{\xb}^2\norm{\vb}_2$ which collapses to the definition introduced in \cite{Bach2009} by letting $\ub = \vb$.
\vspace{1ex}
\item If $\nu = 3$ and $\ub =  \vb$, \eqref{eq:gsc_def} reduces to $\abs{\iprods{\nabla^3f(\xb)[\ub]\ub, \ub}} \leq M_f\norm{\ub}_{\xb}^3$, Definition~\ref{de:gsc_def} becomes the standard self-concordant definition introduced in \cite{Nesterov2004,Nesterov1994}.
\end{enumerate}
We emphasize that Definition~\ref{de:gsc_def} is not symmetric, but can avoid the use of multilinear mappings as required in \cite{Bach2009,Nesterov1994}.
However, by \cite[Proposition 9.1.1]{Nesterov1994} or \cite[Lemma 4.1.2]{Nesterov2004}, Definition~\ref{de:gsc_def} with $\nu=3$ is equivalent to \cite[Definition 4.1.1]{Nesterov2004} for standard self-concordant functions.

\beforesubsec
\subsection{\bf Basic properties of generalized self-concordant functions}\label{subsec:basic_properties}
\aftersubsec
We first show that if $f_1$ and $f_2$ are two \gsc functions, then $\beta_1 f_1 + \beta_2 f_2$ is also a \gsc for any $\beta_1, \beta_2 > 0$ according to Definition~\ref{de:gsc_def}.

\begin{proposition}[Sum of \gsc functions]\label{pro:sum_rule}
Let $f_i$ be $(M_{f_i},\nu)$-\gsc functions satisfying \eqref{eq:gsc_def}, where $M_{f_i} \geq  0$ and $\nu \geq 2$ for $i=1,\cdots, m$. 
Then, for $\beta_i > 0$, $i=1,2,\cdots,m$, the function $f(\xb) := \sum_{i=1}^m\beta_i f_i(\xb)$ is well-defined on $\dom{f } = \bigcap_{i=1}^m\dom{f_i}$, and is $(M_f, \nu)$-\gsc with the same order $\nu \geq 2$ and the constant 
\begin{equation*}
M_f := \max\set{\beta_i^{1-\frac{\nu}{2}}M_{f_i} \mid 1 \leq i \leq m} \geq 0.
\end{equation*}
\end{proposition}

\begin{proof}
It is sufficient to prove for $m=2$. For $m > 2$, it follows from $m=2$ by induction.
By \cite[Theorem 3.1.5]{Nesterov2004}, $f$ is a closed and convex function. 
In addition, $\dom{f} = \dom{f_1}\cap\dom{f_2}$.
Let us fix some $\xb\in\dom{f}$ and $\ub,\vb\in\R^p$. Then, by Definition~\ref{de:gsc_def}, we have
\begin{equation*}
\abs{\iprods{\nabla^3f_i(\xb)[\vb]\ub, \ub}} \leq M_{f_i}\iprods{\nabla^2 f_i(\xb)\ub,\ub}\iprods{\nabla^2 f_i(\xb)\vb,\vb}^{\frac{\nu-2}{2}}\norm{\vb}_2^{3 - \nu}, ~~~i=1,2.
\end{equation*}
Denote $w_i := \iprods{\nabla^2{f}_i(\xb)\ub,\ub} \geq 0$ and $s_i :=  \iprods{\nabla^2{f}_i(\xb)\vb,\vb} \geq 0$ for $i=1,2$. 
We can derive
\begin{eqnarray}\label{eq:key_in_sum}
\frac{\abs{\iprods{\nabla^3f(\xb)[\vb]\ub, \ub}}}{\iprods{\nabla^2 f(\xb)\ub,\ub}\iprods{\nabla^2 f(\xb)\vb,\vb}^{\frac{\nu-2}{2}}} & \leq & \frac{\beta_1\abs{\iprods{\nabla^3f_1(\xb)[\vb]\ub, \ub}}+\beta_2\abs{\iprods{\nabla^3f_2(\xb)[\vb]\ub, \ub}}}{\iprods{\nabla^2 f(\xb)\ub,\ub}\iprods{\nabla^2 f(\xb)\vb,\vb}^{\frac{\nu-2}{2}}} \notag\\
& \leq & \left[\frac{M_{f_1}\beta_1w_1s_1^{\frac{\nu-2}{2}} + M_{f_2}\beta_2w_2s_2^{\frac{\nu-2}{2}}}{(\beta_1w_1+\beta_2w_2)(\beta_1s_1+\beta_2s_2)^{\frac{\nu-2}{2}}}\right]_{[T]}\norm{\vb}_2^{3-\nu}.
\end{eqnarray}
Let $\xi := \frac{\beta_1w_1}{\beta_1w_1+\beta_2w_2} \in [0, 1]$ and $\eta := \frac{\beta_1s_1}{\beta_1s_1+\beta_2s_2} \in [0, 1]$. 
Then, $\tfrac{\beta_2w_2}{\beta_1w_1+\beta_2w_2} = 1 - \xi \geq 0$ and $\tfrac{\beta_2s_2}{\beta_1s_1 + \beta_2s_2} = 1 - \eta \geq 0$.
Hence, the term $[T]$ in the square brackets of  \eqref{eq:key_in_sum} becomes
\begin{equation*} 
h(\xi,\eta):=\beta_1^{1-\frac{\nu}{2}}M_{f_1}\xi\eta^{\frac{\nu-2}{2}}+\beta_2^{1-\frac{\nu}{2}}M_{f_2}(1-\xi)(1-\eta)^{\frac{\nu-2}{2}},~~~ \xi,\eta\in [0,1].
\end{equation*}
Since $\nu \geq 2$ and $\xi, \eta\in [0, 1]$, we can upper bound $h(\xi,\eta)$ as 
\begin{equation*} 
h(\xi,\eta) \leq \beta_1^{1-\frac{\nu}{2}}M_{f_1}\xi   + \beta_2^{1-\frac{\nu}{2}}M_{f_2}(1-\xi), ~~~\forall\xi\in[0, 1].
\end{equation*}
The right-hand side function is linear in $\xi$ on $[0, 1]$. It achieves the maximum at its boundary. 
Hence, we have 
\begin{equation*}
\max_{\xi\in[0,1],\eta\in[0,1]}h(\xi,\eta) \leq \max\set{\beta_1^{1-\frac{\nu}{2}}M_{f_1}, \beta_2^{1-\frac{\nu}{2}}M_{f_2}}.
\end{equation*}
Using  this estimate into \eqref{eq:key_in_sum}, we can show that $f(\cdot) := \beta_1f_1(\cdot) + \beta_2f_2(\cdot)$ is $(M_f,\nu)$-generalized self-concordant with $M_f := \max\set{\beta_1^{1-\frac{\nu}{2}}M_{f_1}, \beta_2^{1-\frac{\nu}{2}}M_{f_2}}$.
\Eproof
\end{proof}

Using Proposition~\ref{pro:sum_rule}, we can also see that if $f$ is $(M_f,\nu)$-generalized self-concordant, and $\beta > 0$, then $g(\xb) := \beta f(\xb)$ is also $(M_g,\nu)$-generalized self-concordant with the constant $M_g := \beta^{1-\frac{\nu}{2}}M_f$.
The convex quadratic function $q(\xb) := \frac{1}{2}\iprods{\Qb\xb,\xb} + \cb^{\top}\xb$ with $\Qb\in\Sc^p_{+}$ is $(0, \nu)$-generalized self-concordant for any $\nu > 0$. 
Hence, by Proposition~\ref{pro:sum_rule}, if $f$ is $(M_f,\nu)$-generalized self-concordant, then $f(\xb) +  \frac{1}{2}\iprods{\Qb\xb,\xb} + \cb^{\top}\xb$ is also $(M_f, \nu)$-generalized self-concordant.

Next, we consider an affine transformation of a generalized self-concordant function.

\begin{proposition}[Affine transformation]\label{pro:affine_transform}
Let $\Ac(\xb) := \Ab\xb + \bb$ be an affine transformation from $\R^p$ to $\R^q$, and $f$ be an $(M_f,\nu)$-generalized self-concordant function with $\nu > 0$. 
Then, the following statements hold:
\begin{itemize}
\item[$\mathrm{(a)}$] If $\nu \in (0, 3]$, then $g(\xb) := f(\Ac(\xb))$ is $(M_{g},\nu)$-generalized self-concordant  with $M_g := M_f\norm{\Ab}^{3-\nu}$.

\vspace{1ex}
\item[$\mathrm{(b)}$] If $\nu > 3$ and $\lambda_{\min}(\Ab^{\top}\Ab) > 0$, then $g(\xb) := f(\Ac(\xb))$ is $(M_{g},\nu)$-generalized self-concordant with $M_g := M_f\lambda_{\min}(\Ab^{\top}\Ab)^{\frac{3-\nu}{2}}$, where $\lambda_{\min}(\Ab^{\top}\Ab)$ is the smallest eigenvalue of $\Ab^{\top}\Ab$.
\end{itemize}
\end{proposition}
 
\begin{proof}
Since $g(\xb) = f(\Ac(\xb)) = f(\Ab\xb + \bb)$, it is easy to show that $\nabla^2g(\xb) = \Ab^{\top}\nabla^2f(\Ac(\xb))\Ab$ and $\nabla^3g(\xb)[\vb] = \Ab^{\top}(\nabla^3f(\Ac(\xb)[\Ab\vb])\Ab$.
Let us denote by $\tilde{\xb} := \Ab\xb + \bb$, $\tilde{\ub} := \Ab\ub$, and $\tilde{\vb} := \Ab\vb$. 
Then, using Definition \ref{de:gsc_def}, we have
\begin{equation}\label{eq:affine_transf}
\begin{array}{ll}
\vert\iprods{\nabla^3g(\xb)[\vb]\ub, \ub}\vert &= \vert\iprods{\Ab^{\top}(\nabla^3f(\tilde{\xb})[\tilde{\vb}])\Ab\ub,\ub} \vert = \vert\iprods{\nabla^3f(\tilde{\xb})[\tilde{\vb}]\tilde{\ub},\tilde{\ub}} \vert \vspace{0.5ex}\\
&\overset{\tiny\eqref{eq:gsc_def}}{\leq} M_f\iprods{ \nabla^2 f(\tilde{\xb})\tilde{\ub},\tilde{\ub}}\iprods{\nabla^2f(\tilde{\xb})\tilde{\vb},\tilde{\vb}}^{\tfrac{\nu}{2}-1}\norm{\tilde{\vb}}_2^{3-\nu} \vspace{0.5ex}\\
&= M_f\iprods{\Ab^{\top}\nabla^2f(\Ac(\xb))\Ab\ub, \ub}\iprods{\Ab^{\top}\nabla^2f(\Ac(\xb))\Ab\vb, \vb}^{\tfrac{\nu}{2}-1}\Vert\Ab\vb\Vert_2^{3-\nu} \vspace{0.5ex}\\
&= M_f\iprods{\nabla^2g(\xb)\ub, \ub}\iprods{\nabla^2g(\xb)\vb, \vb}^{\frac{\nu}{2}-1}\Vert\Ab\vb\Vert_2^{3-\nu}.
\end{array}
\end{equation}
(a)~If $\nu \in (0, 3]$, then we have $\Vert\Ab\vb\Vert_2^{3-\nu} \leq \Vert\Ab\Vert^{3-\nu}\Vert\vb\Vert_2^{3-\nu}$.
Hence, the last inequality \eqref{eq:affine_transf} implies
\begin{equation*}
\vert\iprods{\nabla^3g(\xb)[\vb]\ub, \ub}\vert \leq M_f \Vert\Ab\Vert^{3-\nu}\iprods{\nabla^2g(\xb)\ub, \ub}\iprods{\nabla^2g(\xb)\vb, \vb}^{\frac{\nu}{2}-1}\Vert\vb\Vert_2^{3-\nu}, 
\end{equation*}
which shows that $g$ is $(M_g, \nu)$-generalized self-concordant with $M_g :=  M_f \Vert\Ab\Vert^{3-\nu}$.

\vspace{0.75ex}
\noindent (b)~Note that $\Vert\Ab\vb\Vert_2^2 = \vb^{\top}\Ab^{\top}\Ab\vb \geq \lambda_{\min}(\Ab^{\top}\Ab)\norm{\vb}_2^2 \geq 0$, where $\lambda_{\min}(\Ab^{\top}\Ab)$ is the smallest eigenvalue of $\Ab^{\top}\Ab$. 
If $\lambda_{\min}(\Ab^{\top}\Ab) > 0$ and $\nu > 3$, then we have $\Vert \Ab\vb\Vert_2^{3-\nu} \leq  \lambda_{\min}(\Ab^{\top}\Ab)^{\frac{3-\nu}{2}}\norm{\vb}_2^{3-\nu}$.
Combining this estimate and \eqref{eq:affine_transf}, we can show that $g$ is $(M_g, \nu)$-generalized self-concordant with $M_g := M_f\lambda_{\min}(\Ab^{\top}\Ab)^{\frac{3-\nu}{2}}$.
\Eproof
\end{proof}

\begin{remark}\label{re:limitation}
Proposition~\ref{pro:affine_transform} shows that generalized self-concordance is preserved via an affine transformations if $\nu\in (0, 3]$. 
If $\nu > 3$, then it requires $\Ab$ to be over-completed, i.e., $\lambda_{\min}(\Ab^{\top}\Ab) > 0$.
Hence, the theory developed in the sequel remains applicable for $\nu > 3$ if $\Ab$ is over-completed. 
\end{remark}

The following result is an extension of standard self-concordant functions $(\nu = 3)$ whose proof is very similar to \cite[Theorems 4.1.3, 4.1.4]{Nesterov2004} by replacing the parameters $M_f = 2$ and $\nu = 3$ with the general parameters $M_f \geq 0$ and $\nu > 0$ (or $\nu \geq 2$), respectively. We omit the detailed proof.

\begin{proposition}\label{pro:Hessian_nondegenerate}
Let $f$ be an $(M_f,\nu)$-generalized self-concordant function with $\nu > 0$. Then: 
\begin{itemize}
\item[$\mathrm{(a)}$] 
If $\nu \geq 2$ and $\dom{f}$ contains no straight line, then $\nabla^2{f}(\xb) \succ 0$ for any $\xb \in \dom{f}$.

\vspace{1ex}
\item[$\mathrm{(b)}$] 
If there exists $\bar{\xb}\in\mathrm{bd}(\dom{f})$, the boundary of $\dom{f}$,  then, for any $\bar{\xb}\in\mathrm{bd}(\dom{f})$, and any sequence $\set{\xb_k}\subset\dom{f}$ such that $\lim_{k\to\infty}\xb_k = \bar{\xb}$, we have $\lim_{k\to\infty}f(\xb_k) = +\infty$.
\end{itemize} 
\end{proposition}

Note that Proposition~\ref{pro:Hessian_nondegenerate}(a) only holds for $\nu \geq 2$.
If we consider $g(\xb) := f(\Ac(\xb))$ for a given affine operator $\Ac(x) = Ax + b$, then the non-degenerateness of $\nabla^2g$ is only guaranteed if $A$ is full-rank. 
Otherwise, it is non-degenerated in a given subspace of $A$.

\vspace{-3ex}
\subsection{\bf Generalized self-concordant functions with special structures}\label{subsec:special_funcs}
\vspace{-2ex}
We first show that if a generalized self-concordant function is strongly convex or has a Lipschitz gradient, then it can be cast into the special case $\nu = 2$ or $\nu = 3$.

\begin{proposition}\label{pro:scvx_lips_gsc}
Let $f\in\widetilde{\Fc}_{M_f,\nu}$ be an $(M_f,\nu)$-generalized self-concordant with $\nu > 0$. Then:
\begin{itemize}
\item[$\mathrm{(a)}$] If $\nu\in (0, 3]$  and $f$ is also strongly convex on $\dom{f}$ with the strong convexity parameter $\mu_f > 0$ in $\ell_2$-norm, then 
$f$ is also $(\hat{M}_f, \hat{\nu})$-generalized self-concordant with $\hat{\nu} = 3$ and $\hat{M}_f := \frac{M_f}{(\sqrt{\mu_f})^{3-\nu}}$.

\vspace{1ex}
\item[$\mathrm{(b)}$]  If  $\nu \geq 2$ and $\nabla{f}$ is Lipschitz continuous with the Lipschitz constant $L_f\in [0,+\infty)$ in $\ell_2$-norm, then $f$ is also $(\hat{M}_f, \hat{\nu})$-generalized self-concordant with $\hat{\nu} = 2$ and $\hat{M}_f := M_fL_f^{\frac{\nu}{2}-1}$.
\end{itemize}
\end{proposition}

\begin{proof}
(a)~If $f$ is strongly convex with the strong convexity parameter $\mu_f > 0$ in $\ell_2$-norm, then we have $\iprods{\nabla^2{f}(\xb)\vb, \vb} \geq \mu_f\Vert\vb\Vert_2^2$ for any $\vb\in\R^p$.
Hence, $\frac{\norm{\vb}_{2}}{\Vert\vb\Vert_{\xb}} \leq \frac{1}{\sqrt{\mu_f}}$. 
In this case, \eqref{eq:gsc_def} leads to
\begin{equation*} 
\abs{\iprods{\nabla^3f(\xb)[\vb]\ub, \ub}} \leq M_f\norm{\ub}_{\xb}^2\left(\frac{\norm{\vb}_2}{\Vert\vb\Vert_{\xb}}\right)^{3-\nu}\norm{\vb}_{\xb} \leq \frac{M_f}{(\sqrt{\mu_f})^{3-\nu}}\Vert\ub\Vert_{\xb}^2\Vert\vb\Vert_{\xb}.
\end{equation*}
Hence, $f$ is $(\hat{M}_f, \hat{\nu})$ - generalized self-concordant with $\hat{\nu} = 3$ and $\hat{M}_f := \frac{M_f}{(\sqrt{\mu_f})^{3-\nu}}$.

(b)~Since $\nabla{f}$ is Lipschitz continuous with the Lipschitz constant $L_f\in [0,+\infty)$ in $\ell_2$-norm, we have $\Vert\vb\Vert_{\xb}^2 = \iprods{\nabla^2{f}(\xb)\vb, \vb} \leq L_f\Vert\vb\Vert_2^2$ for all $\vb\in\R^p$ which leads to $\frac{\Vert\vb\Vert_{\xb}}{\Vert\vb\Vert_2} \leq \sqrt{L_f}$ for all $\vb\in\R^p$.
On the other hand, $f\in\widetilde{\Fc}_{M_f,\nu}$ with $\nu \geq 2$, we can show that
\begin{equation*}
\abs{\iprods{\nabla^3f(\xb)[\vb]\ub, \ub}} \leq M_f\norm{\ub}_{\xb}^2\left(\frac{\norm{\vb}_{\xb}}{\Vert\vb\Vert_2}\right)^{\nu-2}\norm{\vb}_2 \leq  M_fL_f^{\frac{\nu-2}{2}}\Vert\ub\Vert_{\xb}^2\Vert\vb\Vert_2.
\end{equation*}
Hence, $f$ is also $(\hat{M}_f, \hat{\nu})$-generalized self-concordant with $\hat{\nu} = 2$ and $\hat{M}_f := M_fL_f^{\frac{\nu-2}{2}}$.
\Eproof
\end{proof}

Proposition~\ref{pro:scvx_lips_gsc} provides two important properties. 
If the gradient map $\nabla{f}$ of a generalized self-concordant function $f$ is Lipschitz continuous, we can always classify it into the special case $\nu = 2$.
Therefore, we can exploit both structures: generalized self-concordance and Lipschitz gradient to develop better algorithms.
This idea is also applied to generalized self-concordant and strongly convex functions.

Given $n$ smooth convex univariate functions $\varphi_i : \R\to\R$ satisfying \eqref{eq:gsc_def0} for $i=1,\cdots, n$ with the same order $\nu > 0$, we consider the function $f : \R^p\to \R$ defined by the following form:
\begin{equation}\label{eq:spc_form}
f(\xb) := \frac{1}{n}\sum_{i=1}^n\varphi_i(\ab_i^{\top}\xb + b_i),
\end{equation}
where $\ab_i\in\R^p$ and $b_i\in\R$ are given vectors and numbers, respectively for $i=1,\cdots, n$.
This convex function is called a finite sum and widely used in machine learning and statistics. 
The decomposable structure in \eqref{eq:spc_form} often appears in generalized linear models  \cite{bollapragada2016exact,byrd2016stochastic}, and empirical risk minimization  \cite{zhang2015disco}, where $\varphi_i$ is referred to as a loss function as can be found, e.g., in Table \ref{tbl:examples}.

Next, we show that if $\varphi_i$ is \gsc with $\nu\in [2,3]$, then $f$ is also \textit{generalized self-concordant}. 
This result is a direct consequence of Proposition~\ref{pro:sum_rule}  and Proposition~\ref{pro:affine_transform}.

\begin{corollary}\label{co:generalized_linear_func1}
If $\varphi_i$ in \eqref{eq:spc_form} satisfies  \eqref{eq:gsc_def0} for $i=1,\cdots, n$ with the same order  $\nu \in [2, 3]$ and $M_{\varphi_i} \geq 0$, then $f$ defined by \eqref{eq:spc_form} is also $(M_f,\nu)$-\gsc in the sense of Definition \ref{de:gsc_def} with the same order $\nu$ and the constant $M_f := n^{\frac{\nu}{2}-1}\max\set{M_{\varphi_i}\norm{\ab_i}_2^{3-\nu} \mid 1 \leq i \leq n}$.
\end{corollary}

Finally, we show that if we regularize $f$ in  \eqref{eq:spc_form} by a strongly convex quadratic term, then the resulting function becomes self-concordant.
The proof can follow the same path as \cite[Lemma 2]{zhang2015disco}. 

\begin{proposition}\label{pro:generalized_linear_func_with_regularizer}
Let $f(\xb) := \frac{1}{n}\sum_{i=1}^n\varphi_i(\ab_i^{\top}\xb + b_i) + \psi(\xb)$, where $\psi(\xb) := \frac{1}{2}\iprods{\Qb\xb,\xb} + \cb^{\top}\xb$ is strongly convex quadratic function with $\Qb\in\Sc^p_{++}$.
If $\varphi_i$ satisfies  \eqref{eq:gsc_def0} for $i=1,\cdots, n$ with the same order $\nu \in (0, 3]$ and a constant $M_{\varphi_i} > 0$, then $f$ is $(\hat{M}_f, 3)$-\gsc in the sense of Definition~\ref{de:gsc_def} with  
$\hat{M}_f := \lambda_{\min}(\Qb)^{\frac{\nu - 3}{2}}\max\set{ M_{\varphi_i} \Vert\ab_i\Vert_2^{3 - \nu} \mid 1\leq i\leq n}$. 
\end{proposition}

\vspace{-0.5ex}
\beforesubsec
\subsection{\bf Fenchel's conjugate of \gsc functions}\label{subsec:special_funcs}
\aftersubsec
\vspace{-0.25ex}
Primal-dual theory is fundamental in convex optimization. 
Hence, it is important to study the Fenchel conjugate of \gsc functions.

Let $f : \R^p\to\R$ be an $(M_f, \nu)$-\gsc function.
We consider Fenchel's conjugate $f^{\ast}$ of $f$ as
\begin{equation}\label{eq:conjugate}
f^{\ast}(x) = \sup_u\set{ \iprods{x, u} - f(u) \mid u\in\dom{f} }.
\end{equation}
Since $f$ is proper, closed, and convex, $f^{\ast}$ is well-defined and also proper, closed, and convex.
Moreover, since $f$ is smooth and convex, by Fermat's rule, if $u^{\ast}(x)$ satisfies $\nabla{f}(u^{\ast}(x)) = x$, then $f^{\ast}$ is well-defined at $x$.
This shows that $\dom{f^{\ast}} = \set{ x\in\R^p \mid \nabla{f}(u^{\ast}(x)) = x~\text{is solvable}}$.

\begin{example}\label{ex:exam2}
Let us look at some univariate functions.
By using \eqref{eq:conjugate}, we can directly show that:
\begin{enumerate}
\item If $\varphi(s) = \log(1 + e^{s})$, then $\varphi^{\ast}(t) = t\log(t) + (1-t)\log(1-t)$.
\vspace{0.75ex}
\item If $\varphi(s) = s\log(s)$, then $\varphi^{\ast}(t) = e^{t-1}$.
\vspace{0.75ex}
\item If $\varphi(s) = e^s$, then $\varphi^{\ast}(t) = t\log(t) - t$.
\end{enumerate}
\end{example}
Intuitively, these examples show that if $\varphi$ is generalized self-concordant, then its conjugate $\varphi^{\ast}$ is also generalized self-concordant.
For more examples, we refer to \cite[Chapter 13]{Bauschke2011}.
Let us generalize this result in the following proposition whose proof is given in Appendix \ref{apdx:pro:conjugate}.

\begin{proposition}\label{pro:conjugate}
If $f$ is $(M_f, \nu)$-generalized self-concordant in $\dom{f}\subseteq\R^p$ such that $\nabla^2{f}(x)\succ 0$ for $x\in\dom{f}$, then the conjugate function $f^{\ast}$ of $f$ given by \eqref{eq:conjugate} is well-defined, and  $(M_{f^{\ast}}, \nu_{\ast})$-generalized self-concordant on 
\begin{equation*}
\dom{f^{\ast}} := \set{x\in\R^p \mid f(u) -  \iprods{x, u} ~\textrm{is bounded from below on}~\dom{f}}, 
\end{equation*}
where $M_{f^{\ast}} = M_f$ and $\nu_{\ast} = 6 - \nu$ provided that $\nu \in [3, 6)$ if $p > 1$ and $\nu \in (0, 6)$ if $p = 1$.

Moreover, we have $\nabla{f^{\ast}}(x) = u^{\ast}(x)$ and $\nabla^2{f^{\ast}}(x) = \nabla^2{f}(u^{\ast}(x))^{-1}$, where $u^{\ast}(x)$ is a unique solution of the maximization problem $\max_{u}\set{\iprods{x, u} - f(u) \mid u\in\dom{f}}$ in \eqref{eq:conjugate} for any $x\in \dom{f^{\ast}}$.
\end{proposition}

Proposition~\ref{pro:conjugate} allows us to apply our generalized self-concordance theory in this paper to the dual problem of a convex problem involving generalized self-concordant functions, especially, when the objective function of the primal problem is generalized self-concordant with $\nu \in (3, 4]$.
The Fenchel conjugates are certainly useful when we develop  optimization algorithms to solve constrained convex optimization involving generalized self-concordant functions, see, e.g., \cite{TranDinh2012e,TranDinh2012c}.

\beforesubsec
\subsection{\bf Generalized self-concordant approximation of nonsmooth convex functions}\label{subsec:smooth_thing}
\aftersubsec
Several well-known convex functions are nonsmooth. However, they can be approximated (up to an arbitrary accuracy) by a generalized self-concordant function via smoothing.
Smoothing techniques clearly allow us to enrich the applicability of our theory to nonsmooth convex problems.

Given a proper, closed, possibly nonsmooth, and convex function $f : \R^p\to\Rext$.
One can smooth $f$ using the following Nesterov's smoothing technique \cite{Nesterov2005c}
\begin{equation}\label{eq:smooth_fx} 
f_{\gamma}(x) := \sup_{u\in\dom{f^{\ast}}}\set{ \iprods{x, u} - f^{\ast}(u) - \gamma\omega(u)},
\end{equation}
where $f^{\ast}$ is the Fenchel conjugate of $f$, $\omega : \dom{\omega} \subseteq\R^p\to\R$ is a smooth convex function such that $\dom{f^{\ast}}\subseteq\dom{\omega}$, and $\gamma > 0$ is called a smoothness parameter.
In particular, if $f$ is Lipschitz continuous, then $\dom{f^{\ast}}$ is bounded \cite{Bauschke2011}.
Hence, the $\sup$ operator in \eqref{eq:smooth_fx} reduces to the $\max$ operator.

Our goal is to choose an appropriate smoothing function $\omega$ such that the smoothed function $f_{\gamma}$ is well-defined and  generalized self-concordant for any fixed smoothness parameter $\gamma > 0$.

\begin{example}\label{ex:smoothing1}
Let us provide a few examples with well-known nonsmooth convex functions:
\begin{itemize}
\item[(a)]~Consider the $\ell_1$-norm function $f(x) := \norm{x}_1$ in $\R^p$.
Then, it can be rewritten as
\begin{equation*}
\norm{x}_1 = \max_{u}\set{ \iprods{x, u} \mid \norm{u}_{\infty} \leq 1} =  \max_{u, v}\big\{ \iprods{x, u - v} \mid \sum_{i=1}^p(u_i + v_i) = 1, ~u,v\in\R^p_{+}\big\}.
\end{equation*}
We can smooth this function by $f_{\gamma}$ by choosing $\omega(u, v) := \ln(2p) + \sum_{i=1}^p(u_i\ln(u_i) + v_i\ln(v_i))$.
In this case, we obtain $f_{\gamma}(x) = \gamma\ln\left(\sum_{i=1}^p\left(e^{x_i/\gamma} + e^{-x_i/\gamma}\right)\right) - \gamma\ln(2p)$.
This function is clearly generalized self-concordant with $\nu = 2$, see  \cite[Lemma 4]{TranDinh2014d}.

However, if we choose $\omega(u) := p - \sum_{i=1}^p\sqrt{1 - u_i^2}$, then we get  $f_{\gamma}(x) = \sum_{i=1}^p\sqrt{x_i^2 + \gamma^2} - \gamma p$.
In this case, $f_{\gamma}$ is also generalized self-concordant with $\nu = \frac{8}{3}$ and $M_{f_{\gamma}} = 3\gamma^{-\frac{2}{3}}$.


\vspace{0.75ex}
\item[(b)] The hinge loss function $\varphi(t) := \max\set{0, 1 - t}$ can be written as $\varphi(t) = \frac{1}{2}\abs{1 - t} + \frac{1}{2}(1 - t)$.
Hence, we can smooth this function by  $\varphi_{\gamma}(t) := \gamma\ln\left(\frac{e^{\frac{(1-t)}{\gamma}} + e^{-\frac{(1-t)}{\gamma}}}{2}\right) + \frac{1}{2}(1-t)$  with a smoothness parameter $\gamma > 0$. 
Clearly, $\varphi_{\gamma}$ is generalized self-concordant with $\nu = 2$.
\end{itemize}
\end{example}

In many practical problems, the conjugate $f^{\ast}$ of $f$ can be written as the sum $f^{\ast} = \varphi + \delta_{\Uc}$, where $\varphi$ is a generalized self-concordant function, and $\delta_{\Uc}$ is the indicator function of a given nonempty, closed, and convex set $\Uc$.
In this case, $f_{\gamma}$ in \eqref{eq:smooth_fx} becomes
\begin{equation}\label{eq:hstar}
f_{\gamma}(x) := \sup_{u}\set{ \iprods{x, u} - \varphi(u) - \gamma\omega(u) \mid u\in\Uc}.
\end{equation}
If $\omega$ is a generalized self-concordant such that $\nu_{\varphi} = \nu_{\omega}$, and $ \Uc = \overline{\dom{\omega}\cap\dom{\varphi}}$, then $f_{\gamma}$ is generalized self-concordant with $\nu_{f_{\gamma}} = 6 - \nu_{\varphi}$ as shown in Proposition \ref{pro:conjugate}.

\beforesubsec
\subsection{\bf Key bounds on Hessian map, gradient map, and function values}\label{subsec:key_bounds}
\aftersubsec
Now, we develop some key bounds on the local norms, Hessian map, gradient map,  and function values of generalized self-concordant functions.
In this subsection, we assume that the Hessian map $\nabla^2{f}$ of $f$ is nondegenerate at any point in its domain.

For this purpose, given $\nu \geq 2$, we define the following quantity for any $\xb, \yb\in\dom{f}$:
\begin{equation}\label{eq:dxy_def}
d_{\nu}(\xb,\yb) := \begin{cases}
M_f\norm{\yb - \xb}_2 &\text{if $\nu = 2$} \vspace{1ex}\\
\left(\frac{\nu}{2}-1\right)M_f\norm{\yb-\xb}_2^{3-\nu}\Vert \yb - \xb\Vert_{\xb}^{\nu-2} &\text{if $\nu > 2$}.
\end{cases}
\end{equation}
Here, if $\nu > 3$, then we require $\xb \neq \yb$. Otherwise, we set $d_{\nu}(\xb, \yb) := 0$ if $\xb = \yb$.
In addition, we also define the  function $\bar{\bar{\omega}}_{\nu} : \R \to \R_{+}$ as
\begin{equation}\label{eq:d2omega_func}
\bar{\bar{\omega}}_{\nu}(\tau) := \begin{cases}
\frac{1}{\left(1-\tau\right)^{\frac{2}{\nu-2}}} &\text{if $\nu > 2$} \vspace{1ex}\\
e^{\tau} &\text{if $\nu = 2$},
\end{cases}
\end{equation}
with $\dom{\bar{\bar{\omega}}_{\nu}} = (-\infty, 1)$ if $\nu > 2$, and $\dom{\bar{\bar{\omega}}_{\nu}} = \R$ if $\nu = 2$.
We also adopt the Dikin ellipsoidal notion from \cite{Nesterov2004} as $W^0(\xb;r):=\set{\yb\in\mathbb{R}^p \mid d_{\nu}(\xb,\yb) < r}$. 

The next proposition provides a bound on the local norm defined by a \gsc function $f$. 
This bound is given for the local distances $\Vert\yb - \xb\Vert_{\xb}$ and $\Vert\yb-\xb\Vert_{\yb}$ between two points $\xb$ and $\yb$ in $\dom{f}$.

\begin{proposition}[Bound of local norms]\label{pro:xy_bounds}
If $\nu > 2$, then, for any $\xb\in\dom{f}$, we have $W^0(\xb; 1)\subseteq\dom{f}$.
For any $\xb,\yb\in\dom{f}$, let $d_{\nu}(\xb,\yb)$ be defined by \eqref{eq:dxy_def}, and $\bar{\bar{\omega}}_{\nu}(\cdot)$ be defined by \eqref{eq:d2omega_func}. 
Then, we have
\begin{equation}\label{eq:key_ineq1}
\bar{\bar{\omega}}_{\nu}\left( -d_{\nu}(\xb,\yb)\right)^{\frac{1}{2}}\norm{\yb-\xb}_{\xb} \leq \norm{\yb - \xb}_{\yb} \leq  \bar{\bar{\omega}}_{\nu}\left( d_{\nu}(\xb,\yb)\right)^{\frac{1}{2}}\norm{\yb-\xb}_{\xb}.
\end{equation}
If $\nu > 2$, then the right-hand side inequality of \eqref{eq:key_ineq1} holds if  $d_{\nu}(\xb, \yb) < 1$.
\end{proposition}

\begin{proof}
We first consider the case $\nu > 2$. 
Let $\ub\in\R^p$ and $\ub\neq 0$. Consider the following univariate function
\begin{equation*}
\phi(t) :=  \iprod{\nabla^2f(\xb+t\ub)\ub, \ub}^{1 - \tfrac{\nu}{2}}=\norm{\ub}_{\xb+t\ub}^{2-\nu}.
\end{equation*}
It is easy to compute the derivative of this function, and obtain
\begin{equation*}
\phi'(t) = \left(\frac{2-\nu}{2}\right)\frac{\iprods{\nabla^3f(\xb+t\ub)[\ub]\ub,\ub}}{\iprod{\nabla^2f(\xb+t\ub)\ub, \ub}^{\frac{\nu}{2}}} = \left(\frac{2-\nu}{2}\right)\frac{\iprods{\nabla^3f(\xb+t\ub)[\ub]\ub,\ub}}{\norm{\ub}_{\xb+t\ub}^{\nu}}.
\end{equation*}
Using Definition~\ref{de:gsc_def} with $\ub = \vb$ and $\xb+t\ub$ instead of $\xb$, we have  $\abs{\phi'(t)}\leq \frac{\nu-2}{2}M_f\norm{\ub}_2^{3-\nu}$.
This implies that $\phi(t) \geq \phi(0) - \frac{\nu-2}{2}M_f\norm{\ub}_2^{3-\nu}\abs{t}$.
On the other hand, we can see that $\dom{\phi} = \set{ t \in\R \mid \phi(t) > 0}$. 
Hence, we have $\dom{\phi}$ contains $\left(-\frac{2\phi(0)}{(\nu-2)M_f\norm{\ub}_2^{3-\nu}}, \frac{2\phi(0)}{(\nu-2)M_f\norm{\ub}_2^{3-\nu}}\right)$.
Using this fact and the definition of $\phi$, we can show that $\dom{f}$ contains $\set{ y := x + tu \mid \abs{t} < \frac{2\norm{\ub}_{\xb}^{2-\nu}}{(\nu-2)M_f\norm{\ub}_2^{3-\nu}} }$.
However, since $\abs{t} = \frac{\norm{y-x}_x^{\nu-2}}{\norm{u}_x^{\nu-2}}\frac{\norm{y-x}_2^{3-\nu}}{\norm{u}_2^{3-\nu}}$, the condition $\abs{t} < \frac{2\norm{\ub}_{\xb}^{2-\nu}}{(\nu-2)M_f\norm{\ub}_2^{3-\nu}}$ is equivalent to $d_{\nu}(\xb,\yb) < 1$.
This shows that $W^0(\xb; 1)\subseteq\dom{f}$.

Since $\big\vert\int_0^1\phi'(t)\mathrm{d}t\big\vert\leq \int_0^1\abs{\phi'(t)}\mathrm{d}t$, integrating $\phi'(t)$ over the interval $[0,1]$ we get
\begin{equation*}
\Big\vert\norm{\ub}_{\xb+\ub}^{2-\nu}-\norm{\ub}_{\xb}^{2-\nu} \Big\vert \leq \frac{\nu-2}{2}M_f\norm{\ub}_2^{3-\nu}.
\end{equation*}
Using $\ub=\yb-\xb$ in the last inequality, we get $\vert\norm{\yb-\xb}_{\yb}^{2-\nu}-\norm{\yb-\xb}_{\xb}^{2-\nu}\vert\leq  \frac{\nu-2}{2}M_f\norm{\yb-\xb}_2^{3-\nu}$ which is equivalent to
\begin{equation*}
\begin{array}{ll}
&\norm{\yb-\xb}_{\yb}^{\nu - 2} \leq \norm{\yb-\xb}_{\xb}^{\nu-2}\left(1 - \frac{\nu-2}{2}M_f\norm{\yb-\xb}_{\xb}^{\nu-2}\norm{\xb-\yb}_2^{3-\nu}\right)^{-1} =  \norm{\yb-\xb}_{\xb}^{\nu-2}\left(1 - d_{\nu}(\xb,\yb)\right)^{-1} \vspace{1ex}\\
&\norm{\yb-\xb}_{\yb}^{\nu-2}  \geq \norm{\yb-\xb}_{\xb}^{\nu-2}\left(1 +  \frac{\nu-2}{2}M_f\norm{\yb-\xb}_{\xb}^{\nu-2}\norm{\xb-\yb}_2^{3-\nu}\right)^{-1} =  \norm{\yb-\xb}_{\xb}^{\nu-2}\left(1 + d_{\nu}(\xb,\yb)\right)^{-1},
\end{array}
\end{equation*}
given that $d_{\nu}(\xb,\yb)<1$. 
Taking the power of $\tfrac{1}{\nu-2} > 0$ in both sides, we get \eqref{eq:key_ineq1} for the case $\nu > 2$.

Now, we consider the case $\nu = 2$.
Let $0\neq\ub\in\R^p$. We consider the following function
\begin{equation*}
\phi(t) := \ln\left(\iprod{\nabla^2f(\xb+t\ub)\ub,\ub}\right) = \ln\big(\Vert\ub\Vert_{\xb+t\ub}^2\big).
\end{equation*}
Clearly, it is easy to show that $\phi'(t) = \frac{ \iprods{\nabla^3f(\xb+t\ub)[\ub]\ub,\ub}}{\iprod{\nabla^2f(\xb+t\ub)\ub, \ub}}=\frac{ \iprods{\nabla^3f(\xb+t\ub)[\ub]\ub,\ub}}{\norm{\ub}_{\xb+t\ub}^2}$.
Using again Definition~\ref{de:gsc_def} with $\ub = \vb$ and $\xb+t\ub$ instead of $\xb$, we obtain $\abs{\phi'(t)}\leq M_f\norm{\ub}_2$.\\
Since $\big\vert\int_0^1\phi'(t)\mathrm{d}t \big\vert \leq \int_0^1\abs{\phi'(t)}\mathrm{d}t$, integrating $\phi'(t)$ over the interval $[0,1]$ we get
\begin{equation*}
\abs{\ln\left(\norm{\ub}_{\xb+\ub}^2\right) - \ln\left(\norm{\ub}_{\xb}^2\right)}\leq M_f\norm{\ub}_2.
\end{equation*}
Substituting $\ub=\yb-\xb$ into this inequality, we get $\big\vert \ln\norm{\yb-\xb}_{\yb} - \ln\norm{\yb-\xb}_{\xb} \big\vert \leq \tfrac{M_f}{2}\norm{\yb-\xb}_2$.
Hence, $\ln\norm{\yb-\xb}_{\xb} -\tfrac{M_f}{2}\norm{\yb-\xb}_2\leq \ln\norm{\yb-\xb}_{\yb} \leq \ln\norm{\yb-\xb}_{\xb} + \tfrac{M_f}{2}\norm{\yb-\xb}_2$. 
This inequality leads to \eqref{eq:key_ineq1} for the case $\nu = 2$.
\Eproof
\end{proof}

Next, we develop new bounds for the Hessian map of $f$ in the following proposition.

\begin{proposition}[Bounds of Hessian map]\label{pro:hessian_bounds}
For any $\xb,\yb\in\dom{f}$, let $d_{\nu}(\xb,\yb)$ be defined by \eqref{eq:dxy_def}, and $\bar{\bar{\omega}}_{\nu}(\cdot)$ be defined by \eqref{eq:d2omega_func}. Then, we have 
\begin{eqnarray}
&\left[1 - d_{\nu}(\xb,\yb)\right]^{\frac{2}{\nu-2}}\nabla^2f(\xb) \preceq \nabla^2f(\yb)  \preceq \left[1 - d_{\nu}(\xb, \yb)\right]^{\frac{-2}{\nu-2}}\nabla^2f(\xb)~~~&\text{if $\nu > 2$},\vspace{1ex}\label{eq:hessian_bound1}\\[6pt]
&e^{-d_{\nu}(\xb,\yb)}\nabla^2f(\xb) \preceq \nabla^2f(\yb)  \preceq  e^{d_{\nu}(\xb,\yb)}\nabla^2f(\xb)~~~&\text{if $\nu = 2$},\label{eq:hessian_bound1b}
\end{eqnarray}
where  \eqref{eq:hessian_bound1} holds if  $d_{\nu}(\xb, \yb) < 1$ for the case $\nu > 2$.
\end{proposition}

\begin{proof}
Let $\nu > 2$ and $0\neq\ub\in\R^n$. 
Consider the following univariate function on $[0, 1]$:
\begin{equation*}
\psi(t) := \iprod{\nabla^2f(\xb+t(\yb-\xb))\ub,\ub},~~t\in [0,1].
\end{equation*}
If we denote by $\yb_t := \xb+t(\yb-\xb)$, then $\yb_t - \xb = t(\yb - \xb)$, $\psi(t)=\norm{\ub}_{\yb_t}^2$, and $\psi'(t) = \iprods{\nabla^3f(\yb_t)[\yb-\xb]\ub,\ub}$.
By Definition~\ref{de:gsc_def}, we have
\begin{eqnarray*}
\abs{\psi'(t)} & \leq  M_f\norm{\ub}_{\yb_t}^2\norm{\yb-\xb}_{\yb_t}^{\nu-2}\norm{\yb-\xb}_2^{3-\nu} = M_f \psi(t)\left[\tfrac{\norm{\yb_t-\xb}_{\yb_t}}{t}\right]^{\nu-2}\norm{\yb-\xb}_2^{3-\nu},
\end{eqnarray*}
which implies
\begin{equation}\label{prop2_eqkey}
\abs{\frac{d\ln\psi(t)}{dt}}\leq M_f \left[\tfrac{\norm{\yb_t-\xb}_{\yb_t}}{t}\right]^{\nu-2}\norm{\yb-\xb}_2^{3-\nu}.
\end{equation}
Assume that $d_{\nu}(\xb, \yb) < 1$. 
Then, by the definition of $\yb_t$ and $d_{\nu}(\cdot)$, we have $d_{\nu}(\xb, \yb_t) = td_{\nu}(\xb, \yb)$ and $\norm{\yb_t - \xb}_{\xb} = t\norm{\yb - \xb}_{\xb}$.
Using Proposition~\ref{pro:xy_bounds}, we can derive
\begin{equation*}
\begin{array}{ll}
\frac{1}{t}\norm{\yb_t-\xb}_{\yb_t} & \leq \frac{1}{t}\left[1 - \left(\frac{\nu}{2}-1\right)\norm{\yb_t - \xb}_2^{3-\nu}\norm{\yb_t - \xb}_{\xb}^{\nu-2}\right]^{-\frac{1}{\nu - 2}}\norm{\yb_t - \xb}_{\xb}\vspace{1ex}\\
& = \frac{1}{t}\left[1 - d_{\nu}(\xb, \yb_t)\right]^{-\frac{1}{\nu-2}}\norm{\yb_t - \xb}_{\xb}\vspace{1ex}\\
& =  \left[1 - d_{\nu}(\xb, \yb)t\right]^{-\frac{1}{\nu-2}}\norm{\yb - \xb}_{\xb}.
\end{array}
\end{equation*}
Hence, we can further derive
\begin{equation*} 
\left[\frac{1}{t}\norm{\yb_t-\xb}_{\yb_t}\right]^{\nu-2} \leq \frac{\norm{\yb-\xb}_{\xb}^{\nu-2}}{1 - d_{\nu}(\xb,\yb)t}
\end{equation*}
Integrating $\frac{d\ln\psi(t)}{dt}$ with respect to $t$ on $[0,1]$ and using the last inequality and  \eqref{prop2_eqkey}, we get
\begin{equation*} 
\abs{\int_0^1\frac{d\ln\psi(t)}{dt}dt}\leq \int_0^1\abs{\frac{d\ln\psi(t)}{dt}}dt\leq \norm{\yb-\xb}_{\xb}^{\nu-2}\norm{\yb-\xb}_2^{3-\nu}\int_0^1\frac{dt}{1 - d_{\nu}(\xb,\yb)t}.
\end{equation*}
Clearly, we can compute this integral explicitly as
\begin{equation*}
\abs{\ln\left[\frac{\norm{\ub}^2_{\yb}}{\norm{\ub}^2_{\xb}}\right]}  = \abs{\ln\left[\frac{\psi(1)}{\psi(0)}\right]} \leq \frac{-2d_{\nu}(\xb,\yb)}{(\nu-2)d_{\nu}(\xb,\yb)}\ln\left[ 1 - d_{\nu}(\xb,\yb)\right] = \ln\left[\left(1 - d_{\nu}(\xb,\yb)\right)^{\frac{-2}{\nu-2}}\right].
\end{equation*}
Rearranging this inequality, we obtain
\begin{equation*}
\left[ 1 - d_{\nu}(\xb,\yb)\right]^{\frac{2}{\nu-2}} \leq \frac{\norm{\ub}^2_{\yb}}{\norm{\ub}^2_{\xb}} \equiv \frac{\iprods{\nabla^2f(\yb)\ub,\ub}}{\iprods{\nabla^2f(\xb)\ub,\ub}} \leq \left[1 - d_{\nu}(\xb,\yb)\right]^{\frac{-2}{\nu-2}}.
\end{equation*}
Since this inequality holds for any $0\neq\ub\in\R^p$, it implies \eqref{eq:hessian_bound1}. 
If $\ub=0$, then \eqref{eq:hessian_bound1} obviously holds.

Now, we consider the case $\nu = 2$.
It follows from \eqref{prop2_eqkey} that 
\begin{equation*}
\abs{\ln\left[\frac{\norm{\ub}^2_{\yb}}{\norm{\ub}^2_{\xb}}\right]} = \abs{\int_0^1\frac{d\ln\psi(t)}{dt}dt} \leq \int_0^1\abs{\frac{d\ln\psi(t)}{dt}}dt \leq M_f\int_0^1\norm{\yb-\xb}_2dt = M_f\norm{\yb - \xb}_2.
\end{equation*}
Since this inequality holds for any $\ub\in\R^p$, it implies  \eqref{eq:hessian_bound1b}.
\Eproof
\end{proof}

The following corollary provides a bound on the mean of the Hessian map $G(\xb, \yb) := \int_0^1\nabla^2{f}(\xb + \tau(\yb - \xb))d\tau$ whose proof is moved to Appendix~\ref{apdx:co:hessian_bound2}.

\begin{corollary}\label{co:hessian_bound2}
For any $\xb,\yb\in\dom{f}$, let $d_{\nu}(\xb,\yb)$ be defined by \eqref{eq:dxy_def}.
Then, we have 
\begin{equation}\label{eq:hessian_bound2}
\underline{\kappa}_{\nu}(d_{\nu}(\xb,\yb))\nabla^2f(\xb) \preceq \int_0^1\nabla^2{f}(\xb + \tau(\yb - \xb))d\tau \preceq  \overline{\kappa}_{\nu}(d_{\nu}(\xb,\yb))\nabla^2f(\xb), 
\end{equation}
where
\needcheck{
\begin{equation*}
\begin{array}{ll}
&\underline{\kappa}_{\nu}(t) := \begin{cases}
\frac{1-e^{-t} }{t} &\text{if $\nu = 2$}\vspace{0.75ex}\\
\frac{1-(1-t)^2}{2t} &\text{if $\nu = 4$}\vspace{0.75ex}\\
\frac{(\nu - 2)}{\nu}\left[\frac{1  -  (1 -  t)^{\frac{\nu}{\nu -  2}}}{t}\right] &\text{if $\nu > 2$~ and ~$\nu \neq 4$},
\end{cases}\\
\text{and}~~~~& \\
&\overline{\kappa}_{\nu}(t) := \begin{cases}
\frac{e^{t}-1}{t} &\text{if $\nu = 2$} \vspace{0.75ex}\\
\frac{-\ln(1 - t)}{t} &\text{if $\nu = 4$} \vspace{0.75ex}\\
\left(\frac{\nu - 2}{\nu - 4}\right)\left[\frac{1 - (1 - t)^{\frac{\nu-4}{\nu-2}}}{t}\right] &\text{if $\nu > 2$~and ~$\nu \neq 4$}.
\end{cases}
\end{array}
\end{equation*}
}
Here, if $\nu > 2$, then we require $d_{\nu}(\xb, \yb)$ to satisfy $d_{\nu}(\xb, \yb) < 1$ for $\xb,\yb\in\dom{f}$ in \eqref{eq:hessian_bound2}.
\end{corollary}

We prove a bound on the gradient inner product  of a generalized self-concordant function $f$.

\begin{proposition}[Bounds of  gradient map]\label{pro:gradient_bound1}
For any $\xb, \yb\in\dom{f}$, we have
\begin{equation}\label{eq:gradient_bound1}
\bar{\omega}_{\nu}\left(-d_{\nu}(\xb, \yb)\right)\norm{\yb-\xb}_{\xb}^2 \leq \iprod{\nabla f(\yb) -  \nabla f(\xb),\yb -  \xb} \leq \bar{\omega}_{\nu}\left(d_{\nu}(\xb, \yb)\right)\norm{\yb-\xb}_{\xb}^2,
\end{equation}
where, if $\nu > 2$, then the right-hand side inequality of \eqref{eq:gradient_bound1} holds if $d_{\nu}(\xb, \yb) < 1$, and
\needcheck{\begin{equation}\label{eq:domega_def}
\bar{\omega}_{\nu}(\tau) := \begin{cases}
\frac{e^{\tau}-1}{\tau} &\text{if $\nu = 2$} \vspace{0.75ex}\\
\frac{\ln\left(1 - \tau\right)}{-\tau} &\text{if $\nu = 4$} \vspace{0.75ex}\\
 \left(\tfrac{\nu-2}{\nu-4}\right)\frac{1 - \left(1 - \tau\right)^{\frac{\nu-4}{\nu-2}}}{\tau} &\text{otherwise}.
\end{cases}
\end{equation}}
Here, $\bar{\omega}_{\nu}(\tau) \geq 0$ for all $\tau \in \dom{\bar{\omega}_{\nu}}$.
\end{proposition}

\begin{proof}
Let $\yb_t := \xb + t(\yb - \xb)$. 
By the mean-value theorem, we have 
\begin{equation}\label{eq:estimate1}
\iprod{\nabla f(\yb)-\nabla f(\xb),\yb-\xb} =  \int_0^1\iprod{\nabla^2 f(\yb_t)(\yb-\xb),\yb-\xb}\ud t  = \int_0^1\frac{1}{t^2}\norm{\yb_t-\xb}_{\yb_t}^2\ud t.
\end{equation}
We consider the function $\bar{\bar{\omega}}_{\nu}$ defined by \eqref{eq:d2omega_func}.
It follows from Proposition \ref{pro:xy_bounds} that
\begin{equation*}
\bar{\bar{\omega}}_{\nu}\left(-d_{\nu}(\xb,\yb_t)\right)\norm{\yb_t-\xb}_{\xb}^2 \leq \norm{\yb_t -\xb}_{\yb_t}^2 \leq \bar{\bar{\omega}}_{\nu}\left(d_{\nu}(\xb,\yb_t)\right)\norm{\yb_t-\xb}_{\xb}^2.
\end{equation*}
Now, we note that $d_{\nu}(\xb,\yb_t) = td_{\nu}(\xb,\yb)$ and $\norm{\yb_t-\xb}_{\xb} = t\norm{\yb -\xb}_{\xb}$, the last estimate leads to
\begin{equation*}
\bar{\bar{\omega}}_{\nu}\left(-td_{\nu}(\xb,\yb)\right)\norm{\yb - \xb}_{\xb}^2 \leq \frac{1}{t^2}\norm{\yb_t -\xb}_{\yb_t}^2 \leq \bar{\bar{\omega}}_{\nu}\left(td_{\nu}(\xb,\yb)\right)\norm{\yb-\xb}_{\xb}^2.
\end{equation*}
Substituting this estimate into \eqref{eq:estimate1}, we obtain
\begin{equation*} 
\norm{\yb - \xb}_{\xb}^2\int_0^1\bar{\bar{\omega}}_{\nu}\left(-td_{\nu}(\xb,\yb)\right)dt \leq  \iprod{\nabla f(\yb)-\nabla f(\xb),\yb-\xb} \leq \norm{\yb - \xb}_{\xb}^2\int_0^1\bar{\bar{\omega}}_{\nu}\left(td_{\nu}(\xb,\yb)\right)dt.
\end{equation*}
Using the function $\bar{\bar{\omega}}_{\nu}(\tau)$ from \eqref{eq:d2omega_func} to compute the left-hand side and the right-hand side integrals, we obtain \eqref{eq:gradient_bound1}.
\Eproof
\end{proof}

Finally, we prove a bound  on the function values of an $(M_f, \nu)$-\gsc function $f$ in the following proposition.

\begin{proposition}[Bounds of  function values]\label{pro:fx_bound1}
For any $\xb, \yb\in\dom{f}$, we have
\begin{equation}\label{eq:f_bound1}
\omega_{\nu}\left(-d_{\nu}(\xb, \yb)\right)\norm{\yb-\xb}_{\xb}^2 \leq f(\yb) - f(\xb) - \iprod{\nabla f(\xb),\yb - \xb} \leq \omega_{\nu}\left(d_{\nu}(\xb, \yb)\right)\norm{\yb-\xb}_{\xb}^2,
\end{equation}
where, if $\nu > 2$, then the right-hand side inequality of \eqref{eq:f_bound1} holds if $d_{\nu}(\xb, \yb) < 1$. 
Here, $d_{\nu}(\xb,\yb)$ is defined by \eqref{eq:dxy_def} and $\omega_{\nu}$ is defined by
\needcheck{\begin{equation}\label{eq:omega_def}
\omega_{\nu}(\tau) := \begin{cases}
\frac{e^{\tau}-\tau - 1}{\tau^2} &\text{if $\nu = 2$} \vspace{0.75ex}\\
\frac{-\tau-\ln(1-\tau)}{\tau^2} &\text{if $\nu = 3$} \vspace{0.75ex}\\
\frac{(1-\tau)\ln\left(1 - \tau\right) + \tau}{\tau^2} &\text{if $\nu = 4$} \vspace{0.75ex}\\
 \left(\tfrac{\nu-2}{4-\nu}\right)\frac{1}{\tau}\left[\frac{\nu-2}{2(3-\nu)\tau}\left((1 - \tau)^{\frac{2(3-\nu)}{2-\nu}} - 1\right) - 1\right] &\text{otherwise}.
\end{cases}
\end{equation}}
Note that $\omega_{\nu}(\tau) \geq 0$ for all $\tau \in \dom{\omega_{\nu}}$.
\end{proposition}

\begin{proof}
For any $\xb,\yb\in\dom{f}$, let $\yb_t := \xb + t(\yb - \xb)$. Then, $\yb_t - \xb = t(\yb - \xb)$.
By the mean-value theorem, we have
\begin{equation*} 
f(\yb)-f(\xb)-\iprod{\nabla f(\xb),\yb-\xb} =  \int_0^1 \tfrac{1}{t}\iprods{\nabla f(\yb_t)-\nabla f(\xb),\yb_t-\xb}dt.
\end{equation*}
Now, by Proposition \ref{pro:gradient_bound1}, we have 
\begin{equation*}
\bar{\omega}_{\nu}\left( - d_{\nu}(\xb,\yb_t)\right)\norm{\yb_t - \xb}_{\xb}^2 \leq \iprods{\nabla f(\yb_t)-\nabla f(\xb),\yb_t-\xb} \leq \bar{\omega}_{\nu}\left( d_{\nu}(\xb,\yb_t)\right)\norm{\yb_t - \xb}_{\xb}^2.
\end{equation*}
Clearly, by the definition \eqref{eq:dxy_def}, we have $d_{\nu}(\xb, \yb_t) = td_{\nu}(\xb,\yb)$ and $\norm{\yb_t-\xb}_{\xb} = t\norm{\yb-\xb}_{\xb}$.
Combining these relations, and the above two inequalities, we can show that 
\begin{equation*}
\norm{\yb-\xb}_{\xb}^2\int_0^1 t\bar{\omega}_{\nu}\left(-td_{\nu}(\xb,\yb)\right)dt \leq f(\yb)-f(\xb)-\iprod{\nabla f(\xb),\yb-\xb} \leq  \norm{\yb-\xb}_{\xb}^2\int_0^1 t\bar{\omega}_{\nu}\left(td_{\nu}(\xb,\yb)\right)dt.
\end{equation*}
By integrating the left-hand side and the right-hand side of this estimate using the definition \eqref{eq:domega_def} of $\bar{\omega}_{\nu}(\tau)$, we obtain \eqref{eq:f_bound1}.
\Eproof
\end{proof}

\vspace{-0.5ex}
\beforesec
\section{Generalized self-concordant minimization}\label{sec:gsc_min}
\aftersec
We apply the theory developed in the previous sections to design new Newton-type methods to minimize a generalized self-concordant function. 
More precisely, we consider the following non-composite convex problem:
\begin{equation}\label{eq:gsc_min}
f^{\star} := \min_{\xb\in\R^p} f(\xb),
\end{equation}
where $f : \R^p\to\R$ is an $(M_f,\nu)$-\gsc function in the sense of Definition \ref{de:gsc_def} with $\nu \in [2, 3]$ and $M_f \geq 0$.
Since $f$ is smooth and convex, the optimality condition $\nabla{f}(\xopt_f) = 0$ is necessary and sufficient for $\xopt_f$ to be an optimal solution of \eqref{eq:gsc_min}.

The following theorem shows the existence and uniqueness of the solution $\xopt_f$ of \eqref{eq:gsc_min}.
It can be considered as a special case of Theorem~\ref{th:existence_and_unique} below with $g\equiv 0$. 

\begin{theorem}\label{th:existence_and_unique_gsc_min}
Suppose that $f\in\widetilde{\Fc}_{M_f,\nu}(\dom{f})$ for given parameters $M_f > 0$ and $\nu \in [2, 3]$. 
Denote by $\sigma_{\min}(x) := \lambda_{\min}(\nabla^2 f(\xb))$ and $\lambda(\xb) := \Vert\nabla{f}(\xb)\Vert^{\ast}_{\xb}$ for $x\in\dom{f}$. 
Suppose further that there exists $x\in\dom{f}$ such that $\sigma_{\min}(x) > 0$ and
\begin{equation*}
\lambda(\xb)  <  \frac{2\left[\sigma_{\min}(x)\right]^{\frac{3-\nu}{2}}}{(4-\nu)M_f}. 
\end{equation*}
Then, problem \eqref{eq:gsc_min} has a unique solution $\xb_f^{\star}$ in $\dom{f}$.
\end{theorem}

We say that the unique solution $\xopt_f$ of \eqref{eq:gsc_min} is \textit{strongly regular} if $\nabla^2{f}(\xopt_f) \succ 0$.
The strong regularity of $\xopt_f$ for \eqref{eq:gsc_min} is equivalent to the strong second order optimality condition.
Theorem \ref{th:existence_and_unique_gsc_min} covers \cite[Theorem 4.1.11]{Nesterov2004} for standard self-concordant functions as a special case.

We consider the following Newton-type scheme to solve \eqref{eq:gsc_min}. 
Starting from an arbitrary initial point $\xb^0\in\dom{f}$, we generate a sequence $\set{\xb^k}_{k\geq 0}$ as follows:
\begin{equation}\label{eq:NT_scheme}
\xb^{k+1} := \xb^k + \tau_k\ntdir^k,~~~~~\text{where}~~\ntdir^k := -\nabla^2f(\xb^k)^{-1}\nabla{f}(\xb^k),
\end{equation}
and $\tau_k \in (0, 1]$ is a given step-size. We call $\ntdir^k$ a Newton direction.
\begin{itemize}
\item If $\tau_k = 1$ for all $k\geq 0$, then we call \eqref{eq:NT_scheme} a \textit{full-step} Newton scheme. 
\vspace{1ex}
\item Otherwise, i.e., $\tau_k \in (0, 1)$, we call \eqref{eq:NT_scheme} a \textit{damped-step} Newton scheme.
\end{itemize}
Clearly, computing the Newton direction $\ntdir^k$ requires to solve the following linear system:
\begin{equation}\label{alg_1:cpt}
\nabla^2{f}(\xb^k)\ntdir^k = -\nabla{f}(\xb^k).
\end{equation}
Next, we define a \textit{Newton decrement} $\lambda_k$ and a quantity $\beta_k$, respectively  as 
\begin{equation}\label{eq:NT_decrement}
\lambda_k := \Vert\ntdir^k\Vert_{\xb^k} = \Vert\nabla{f}(\xb^k)\Vert_{\xb^k}^{\ast}~~~\text{and}~~~\beta_k := M_f\Vert\ntdir^k\Vert_2 = M_f\Vert \nabla^2f(\xb^k)^{-1}\nabla{f}(\xb^k)\Vert_2.
\end{equation}
With $\lambda_k$ and $\beta_k$ given by \eqref{eq:NT_decrement}, we also define
\begin{equation}\label{eq:d_k}
\needcheck{
d_k := \begin{cases}
\beta_k &\text{if $\nu = 2$}\vspace{1ex}\\
\left(\frac{\nu}{2} - 1\right)M_f^{\nu-2}\lambda_k^{\nu-2}\beta_k^{3-\nu} &\text{if $\nu \in (2, 3]$}.
\end{cases}}
\end{equation}
Let us first show how to choose a suitable step-size $\tau_k$ in the damped-step Newton scheme and prove its convergence properties in the following theorem whose proof can be found in Appendix~\ref{apdx:th:damped_step_NT}.

\begin{theorem}\label{th:damped_step_NT}
Let $\set{\xb^k}$ be the sequence generated by the damped-step Newton scheme \eqref{eq:NT_scheme} with the following step-size:
\begin{equation}\label{eq:step_size}
\needcheck{
\tau_k := \begin{cases}
\frac{1}{\beta_k}\ln(1 + \beta_k)   &\text{if $\nu = 2$} \vspace{1ex}\\
\frac{1}{d_k}\left[1-\left(1+\frac{4-\nu}{\nu-2}d_k\right)^{-\frac{\nu-2}{4-\nu}}\right]  & \text{if $\nu \in (2, 3]$}, 
\end{cases}}
\end{equation}
where $\lambda_k$, $\beta_k$ are defined by \eqref{eq:NT_decrement}, and $d_k$ is defined by \eqref{eq:d_k}.
Then, $\tau_k\in (0,1]$, $\set{\xb^k}$ in $\dom{f}$, and this step-size guarantees the following descent property
\begin{equation}\label{eq:descent_property}
f(\xb^{k+1}) \leq f(\xb^k) - \Delta_k,
\end{equation}
where \needcheck{$\Delta_k := \lambda_k^2\tau_k - \omega_{\nu}\left( \tau_k d_k\right)\tau^2_k\lambda_k^2 > 0$} with $\omega_{\nu}$  defined by \eqref{eq:omega_def}.

Assume that the unique solution $\xopt_f$ of \eqref{eq:gsc_min} exists.
Then, there exists a neighborhood $\Nc(\xopt_f)$ such that if we initialize the Newton scheme~\eqref{eq:NT_scheme} at $\xb^0\in\Nc(\xopt_f)\cap\dom{f}$, then the whole sequence $\set{\xb^k}$ converges to $\xopt_f$ at a quadratic rate.
\end{theorem}

\begin{example}[\textbf{Better step-size for regularized logistic and exponential models}]\label{ex:compare_step_size}
Consider the minimization problem \eqref{eq:gsc_min} with the objective function $f(\cdot) := \phi(\cdot) + \frac{\gamma}{2}\Vert\cdot\Vert_2^2$, where $\phi$ is defined as in \eqref{eq:spc_form} with $\varphi_i(t) = \log(1 + e^{-t})$ being the logistic loss. 
That is
\begin{equation*}
f(\xb) := \frac{1}{n}\sum_{i=1}^n\log(1 + e^{-\ab_i^{\top}\xb}) + \frac{\gamma}{2}\Vert\xb\Vert_2^2.
\end{equation*}
As we shown in Section~\ref{sec:gsc_background} that $f$ is either generalized self-concordant with $\nu = 2$ or generalized self-concordant with $\nu = 3$ but with different constant $M_f$.

Let us define $R_A := \max\set{\Vert\ab_i\Vert_2 \mid 1 \leq i \leq n}$.
Then, if we consider $\nu = 2$, then we have $M_f^{(2)} = R_A$ due to Corollary \ref{co:generalized_linear_func1}, while if we choose $\nu = 3$, then  $M_f^{(3)} = \frac{1}{\sqrt{\gamma}}R_A$ due to Proposition \ref{pro:scvx_lips_gsc}.
By the definition of $f$, we have $\nabla^2{f}(\xb) \succeq \gamma\Id$.
Hence, using this inequality and the definition of $\lambda_k$ and $\beta_k$ from \eqref{eq:NT_decrement}, we can show that
\begin{equation}\label{eq:compare_beta_lambda}
\beta_k = M_f^{(2)}\Vert\nabla^2{f}(\xb^k)^{-1}\nabla{f}(\xb^k)\Vert_2 \leq \tfrac{R_A}{\sqrt{\gamma}}\lambda_k  = M_f^{(3)}\lambda_k.
\end{equation}
For any $\tau > 0$, we have $\frac{\ln(1 + \tau)}{\tau} > \frac{1}{1 + 0.5\tau}$. 
Using this elementary result and \eqref{eq:compare_beta_lambda}, we obtain 
\begin{equation*}
\tau_k^{(2)} = \tfrac{\ln(1 + \beta_k)}{\beta_k} > \tfrac{1}{1 + 0.5\beta_k} \geq \tfrac{1}{1 +  0.5M^{(3)}_f\lambda_k} = \tau^{(3)}_k.
\end{equation*}
This inequality has shown that the step-size $\tau_k$ given by Theorem~\ref{th:damped_step_NT} satisfies $\tau_k^{(2)} > \tau_k^{(3)}$, where $\tau_k^{(\nu)}$ is a given step-size computed by \eqref{eq:step_size} for $\nu = 2$ and $3$, respectively. 
Such a statement confirms that the damped-step Newton method using $\tau_k^{(2)}$ is theoretically better than using $\tau_k^{(3)}$. 
This result will be empirically confirmed by our  experiments in Section~\ref{sec:num_experiments}.
\Eproof
\end{example}

Next, we study the full-step Newton scheme derived from \eqref{eq:NT_scheme} by setting the step-size $\tau_k = 1$ for all $k\geq 0$ as a full-step. 
Let 
\begin{equation*}
\underline{\sigma}_k := \lambda_{\min}\left(\nabla^2f(\xb^k)\right)
\end{equation*}
be the smallest eigenvalue of $\nabla^2{f}(\xb^k)$. 
Since $\nabla^2{f}(\xb^k)\succ 0$, we have $\underline{\sigma}_k > 0$.
The following theorem shows a local quadratic convergence of the full-step Newton scheme \eqref{eq:NT_scheme} for solving \eqref{eq:gsc_min} whose proof can be found in Appendix~\ref{apdx:th:full_step_NT_scheme_converg}.

\needcheck{
\begin{theorem}\label{th:full_step_NT_scheme_converg}
Let $\set{\xb^k}$ be the sequence generated by the full-step Newton scheme \eqref{eq:NT_scheme} by setting the step-size $\tau_k = 1$ for $k\geq 0$.
Let $d_{\nu}^k := d_{\nu}(\xb^k, \xb^{k+1})$ be defined by \eqref{eq:dxy_def} and $\lambda_k$ be defined by \eqref{eq:NT_decrement}.
Then, the following statements hold:
\begin{enumerate}
\item[$\mathrm{(a)}$] If $\nu = 2$ and the starting point $\xb^0$ satisfies $\underline{\sigma}_0^{-1/2}\lambda_0 < \frac{d_2^{\star}}{M_f}$, then both sequences $\set{\underline{\sigma}_k^{-1/2}\lambda_k}$ and $\set{d_2^k}$ decrease and quadratically converge  to zero, where $d_2^{\star}\approx 0.12964$.

\vspace{1ex}
\item[$\mathrm{(b)}$] If $2 < \nu < 3$, and the starting point $\xb^0$ satisfies $\underline{\sigma}_0^{-\frac{3-\nu}{2}}\lambda_0 < \frac{1}{M_f}\min\set{\frac{2d_{\nu}^{\star}}{\nu-2},\frac{1}{2}}$, then both  sequences $\set{\underline{\sigma}_k^{-\frac{3-\nu}{2}}\lambda_k}$ and $\set{d_{\nu}^k}$ decrease and quadratically converge to zero, 
where $d_{\nu}^{\star}$ is the unique solution of the equation $\left(\nu -2\right)R_{\nu}(d_{\nu})= 4(1-d_{\nu})^{\frac{4-\nu}{\nu-2}}$ in $d_{\nu}$ with $R_{\nu}(\cdot)$ given by \eqref{eq:R_alpha}.

\vspace{1ex}
\item[$\mathrm{(c)}$] If $\nu = 3$ and the starting point $\xb^0$ satisfies $\lambda_0 < \frac{1}{2M_f}$, then the sequence $\set{\lambda_k}$  decreases and quadratically converges to zero.
\end{enumerate}
As a consequence, if $\set{d^k_{\nu}}$ locally converges  to zero at a quadratic rate, then $\big\{\Vert\xb^k - \xopt_f\Vert_{\Hb_k}\big\}$ also locally converges to zero at a quadratic rate, where $\Hb_k = \Id$, the identity matrix, if $\nu = 2$; $\Hb_k = \nabla^2{f}(\xb^k)$ if $\nu = 3$; and $\Hb_k = \nabla^2f(\xb^k)^{\frac{\nu}{2}-1}$ if $2 < \nu < 3$.
Hence, $\set{\xb^k}$ locally converges to $\xopt_f$, the unique solution of \eqref{eq:gsc_min}, at a quadratic rate.
\end{theorem}}

If we combine the results of Theorem~\ref{th:damped_step_NT} and Theorem~\ref{th:full_step_NT_scheme_converg}, then we can design a two-phase Newton algorithm for solving \eqref{eq:gsc_min} as follows:
\begin{itemize}
\item \textit{Phase 1:} Starting from an arbitrary initial point $\xb^0 \in\dom{f}$, we perform the damped-step Newton scheme \eqref{eq:NT_scheme} until the condition in Theorem~\ref{th:full_step_NT_scheme_converg} is satisfied.
\vspace{1ex}
\item \textit{Phase 2:} Using the output $\xb^j$ of \textit{Phase 1} as an initial point for the full-step Newton scheme  \eqref{eq:NT_scheme} with $\tau_k = 1$, and perform this scheme until it achieves an $\varepsilon$-solution $\xb^k$ to \eqref{eq:gsc_min}.
\end{itemize}
We also note that the damped-step Newton scheme \eqref{eq:NT_scheme} can also achieve a local quadratic convergence as shown in Theorem \ref{th:damped_step_NT}.
Hence, we combine this fact and the above two-phase scheme to derive the Newton algorithm as shown in Algorithm~\ref{alg:Newton_alg} below.

\begin{algorithm}[ht!]\caption{(\textit{Newton algorithm for generalized self-concordant minimization})}\label{alg:Newton_alg}
\begin{normalsize}
\begin{algorithmic}[1]
   \State {\bfseries Inputs:} Choose an arbitrary initial point $\xb^0 \in\dom{f}$ and a desired accuracy $\varepsilon > 0$.
   \vspace{0.75ex}
   \State {\bfseries Output:}  An $\varepsilon$-solution $\xb^k$ of \eqref{eq:gsc_min}.
    \vspace{0.75ex}
   \State {\bfseries Initialization:} Compute $d_{\nu}^{\star}$ according to Theorem \ref{th:full_step_NT_scheme_converg} if needed.
   \vspace{0.75ex}
   \State{\bfseries For} $k = 0, \cdots, k_{\max}$, \textbf{perform:}
   \vspace{0.75ex}
   \State\hspace{0.16cm}\label{step:alg1_step1}~Compute the Newton direction $\ntdir^k$ by solving $\nabla^2{f}(\xb^k)\ntdir^k = -\nabla{f}(\xb^k)$.
   \vspace{0.75ex}
   \State\hspace{0.16cm} Compute $\lambda_k := \Vert\ntdir^k\Vert_{\xb^k}$, and compute $\beta_k := M_f\Vert \ntdir^k\Vert_2$ if $\nu \neq 3$.
   \vspace{0.75ex}
   \State\hspace{0.16cm} If $\lambda_k \leq \varepsilon$, then TERMINATE and return $\xb^k$.
   \vspace{0.75ex}
   \State\hspace{0.16cm} If \textit{Phase 2 is used}, then compute $\underline{\sigma}_k =\lambda_{\min}(\nabla^2 f(\xb^k))$ if $2 \leq \nu < 3$.
   \vspace{0.75ex}
   \State\hspace{0.16cm} If \textit{Phase 2 is used} and $(\lambda_k, \underline{\sigma}_k)$  satisfies Theorem~\ref{th:full_step_NT_scheme_converg}, then set $\tau_k \!:=\! 1$ (\textbf{full-step}).{\!\!}
   \vspace{0.75ex}
   \Statex\hspace{0.16cm} Otherwise, compute the step-size $\tau_k$ by \eqref{eq:step_size} (\textbf{damped-step})
   \vspace{0.75ex}
   \State\hspace{0.16cm} Update $\xb^{k+1}:=\xb^k + \tau_k \ntdir^k$.
   \vspace{0.5ex}
   \State {\bfseries End for}
\end{algorithmic}
\end{normalsize}
\end{algorithm}

\beforepar
\paragraph{\textbf{Per-iteration complexity:}}
The main step of Algorithm~\ref{alg:Newton_alg} is the solution of the symmetric positive definite linear system \eqref{alg_1:cpt}. 
This system can be solved by using either Cholesky factorization or conjugate gradient methods which, in the worst-case, requires $\mathcal{O}(p^3)$ operations.
Computing $\lambda_k$ requires the inner product $\iprods{\ntdir^k, \nabla{f}(\xb^k)}$ which needs $\mathcal{O}(p)$ operations.

Conceptually, the two-phase option of Algorithm~\ref{alg:Newton_alg} requires the smallest eigenvalue of $\nabla^2{f}(\xb^k)$ to terminate Phase 1. 
However, switching from Phase~1 to Phase~2 can be done automatically allowing some tolerance in the step-size $\tau_k$.
Indeed, the step-size $\tau_k$ given by \eqref{eq:step_size} converges to $1$ as $k\to\infty$. 
Hence, when $\tau_k$ is closed to $1$, e.g., $\tau_k \geq 0.9$, we can automatically set it to $1$ and remove the computation of $\lambda_k$ to reduce the computational time.

In the one-phase option, we can always perform only Phase~1 until achieving an $\varepsilon$-optimal solution as shown in Theorem~\ref{th:damped_step_NT}. 
Therefore, the per-iteration complexity of Algorithm~\ref{alg:Newton_alg} is $\mathcal{O}(p^3) + \mathcal{O}(p)$ in the worst-case. 
A careful implementation of conjugate gradient methods with a warm-start can significantly reduce this per-iteration computation complexity.

\begin{remark}[\textbf{Inexact Newton methods}]\label{re:inexact_NT}
We can allow Algorithm~\ref{alg:Newton_alg} to compute the Newton direction $\ntdir^k$ approximately. 
In this case, we approximately solve the symmetric positive definite system \eqref{alg_1:cpt}.
By an appropriate choice of stopping criterion, we can still prove convergence of Algorithm~\ref{alg:Newton_alg} under inexact computation of $\ntdir^k$.
For instance, the following criterion is often used in inexact Newton methods \cite{Deuflhard2006}, but defined via the local dual norm of $f$:
\begin{equation*}
\Vert \nabla^2{f}(\xb^k)\ntdir^k + \nabla{f}(\xb^k)\Vert_{\xb^k}^{\ast} \leq \kappa\Vert\nabla{f}(\xb^k)\Vert_{\xb^k}^{\ast},
\end{equation*}
for a given relaxation parameter $\kappa \in [0, 1)$.
This extension can be found in our forthcoming work.
\end{remark}

\vspace{-0.25ex}
\beforesec
\section{Composite generalized self-concordant minimization}\label{sec:gsc_composite_min}
\aftersec
Let $f\in\widetilde{\Fc}_{M_f,\nu}(\dom{f})$, and $g$ be a proper, closed, and convex function. 
We consider the composite convex minimization problem  \eqref{eq:composite_cvx0} which we recall here for our convenience of references:
\begin{equation}\label{eq:composite_cvx}
F^{\star} := \min_{\xb\in\R^p}\Big\{ F(\xb) := f(\xb) + g(\xb) \Big\}.
\end{equation}
Note that $\dom{F} := \dom{f} \cap\dom{g}$ may be empty. 
To make this problem nontrivial, we assume that $\dom{F}$ is nonempty.
The optimality condition for \eqref{eq:composite_cvx} can be written as follows:
\begin{equation}\label{eq:composite_cvx_opt_cond}
0\in\nabla f(\xb^{\star})+\partial g(\xb^{\star}).
\end{equation}
Under the qualification condition $0 \in\ri{\dom{g} - \dom{f}}$, \eqref{eq:composite_cvx_opt_cond} is necessary and sufficient for $\xb^{\star}$ to be an optimal solution of \eqref{eq:composite_cvx}, where $\ri{\Xc}$ is the relative interior of $\Xc$.

\beforesubsec
\subsection{\bf Existence, uniqueness, and regularity of optimal solutions}
\aftersubsec
Assume that $\nabla^2{f}(\xb)$ is positive definite (i.e., nonsingular) at some point $\xb\in\dom{F}$. 
We prove in the following theorem that  problem~\eqref{eq:composite_cvx} has a unique solution $\xopt$.
The proof can be found in Appendix~\ref{apdx:th:existence_and_unique}. 
This theorem can also be considered as a generalization of \cite[Theorem 4.1.11]{Nesterov2004} and \cite[Lemma 4]{Tran-Dinh2013a} in standard self-concordant settings in \cite{Nesterov2004,Tran-Dinh2013a}.

\begin{theorem}\label{th:existence_and_unique}
Suppose that the function $f$ of \eqref{eq:composite_cvx} is $(M_f, \nu)$-generalized self-concordant with $M_f > 0$ and $\nu \in [2, 3]$. 
Denote by $\sigma_{\min}(x) := \lambda_{\min}(\nabla^2 f(\xb))$ and $\lambda(\xb) := \Vert\nabla{f}(\xb) + \vb\Vert^{\ast}_{\xb}$ for $x\in\dom{F}$ and $\vb\in\partial{g}(\xb)$. 
Suppose further that there exists $x\in\dom{F}$ such that $\sigma_{\min}(x) > 0$ and
\begin{equation*}
\lambda(\xb)  <  \frac{2\left[\sigma_{\min}(x)\right]^{\frac{3-\nu}{2}}}{(4-\nu)M_f}. 
\end{equation*}
Then, problem \eqref{eq:composite_cvx} has a unique solution $\xopt$ in $\dom{F}$.
\end{theorem}

Now, we recall a condition such that the solution $\xopt$ of \eqref{eq:composite_cvx} is strongly regular in the following Robinson's sense \cite{Robinson1980}. 
We say that the optimal solution $\xopt$ of \eqref{eq:composite_cvx} is \textit{strongly regular} if there exists a neighborhood $\Uc(\boldsymbol{0})$ of zero such that for any $\delta\in\Uc(\boldsymbol{0})$, the following perturbed problem
\begin{equation*}
\min_{\xb\in\R^p} \set{ \iprods{\nabla{f}(\xopt) - \delta, \xb - \xopt} + \tfrac{1}{2}\iprods{\nabla^2{f}(\xopt)(\xb - \xopt), \xb - \xopt} + g(\xb) }
\end{equation*}
has a unique solution $\xb^{\ast}(\delta)$, and this solution is Lipschitz continuous on $\Uc(\boldsymbol{0})$.

If $\nabla^2{f}(\xopt) \succ 0$, then $\xopt$ is strongly regular. 
While the strong regularity of the solution $\xopt$ requires a weaker condition than $\nabla^2{f}(\xopt) \succ 0$.
For further details of the regularity theory, we refer the reader to  \cite{Robinson1980}.

\beforesubsec
\subsection{\bf Scaled proximal operators}
\aftersubsec
Given a matrix $\Hb\in\Sc^p_{++}$, we define a scaled proximal operator of $g$ in \eqref{eq:composite_cvx} as
\begin{equation}\label{eq:def_gprox}
\prox_{\Hb^{-1} g}(\xb):=\argmin_{\zb}\set{g(\zb) + \tfrac{1}{2}\norm{\zb-\xb}^2_{\Hb}}.
\end{equation}
Using the optimality condition of the minimization problem under \eqref{eq:def_gprox}, we can show that
\begin{equation*}
\yb =  \prox_{\Hb^{-1} g}(\xb)\iff  0 \in \Hb(\yb - \xb) + \partial{g}(\yb) \iff \xb \in \yb + \Hb^{-1}\partial g(\yb) \equiv (\Id + \Hb^{-1}\partial{g})(\yb).
\end{equation*}
Since $g$ is proper, closed, and convex, $\prox_{\Hb^{-1} g}$ is well-defined and single-valued. 
In particular, if we take $\Hb=\Id$, the identity matrix, then $\prox_{\Hb^{-1}g}(\cdot) = \prox_{g}(\cdot)$, the standard proximal operator of $g$. 
If we can efficiently  compute $\prox_{\Hb^{-1}g}(\cdot)$ by a closed form or by polynomial time algorithms, then we say that $g$ is \textit{proximally tractable}. 
There exist several convex functions whose proximal operator is tractable. 
Examples such as $\ell_1$-norm, coordinate-wise separable convex functions, and the indicator of simple convex sets can be found in the literature including \cite{Bauschke2011,friedlander2016efficient,Parikh2013}.

\beforesubsec
\subsection{\bf Proximal Newton methods}
\aftersubsec
The proximal Newton method can be considered as a special case of the variable metric proximal method in the literature \cite{Bonnans1994a}.
This method has previously been studied  by many authors, see, e.g., \cite{Bonnans1994a,Lee2014}. 
However, the convergence guarantee often requires certain assumptions as used in standard Newton-type methods.
In this section, we develop a proximal Newton algorithm to solve the composite convex minimization problem \eqref{eq:composite_cvx} where $f$ is a generalized self-concordant function.
This problem covers \cite{Tran-Dinh2013a,Tran-Dinh2013} as special cases.

Given $\xb^k\in\dom{F}$, we first approximate $f$ at $\xb^k$ by the following convex quadratic surrogate:
\begin{equation*} 
Q_f(\xb;\xb^k):= f(\xb^k)+\iprod{\nabla f(\xb^k),\xb-\xb^k}+\tfrac{1}{2}\iprod{\nabla^2 f(\xb^k)(\xb-\xb^k),\xb-\xb^k}.
\end{equation*}
Next, the main step of the proximal Newton method requires to solve the following subproblem:
\begin{equation}\label{eq:dir_ds}
\zb^k := \mathrm{arg}{\!\!\!\!\!\!}\min_{\xb\in\dom{g}}\Big\{Q_f(\xb;\xb^k)+g(\xb) \Big\} = \prox_{\nabla^2{f}(\xb^k)^{-1}g}\Big(\xb^k - \nabla^2{f}(\xb^k)^{-1}\nabla{f}(\xb^k)\Big).
\end{equation}
The optimality condition for this subproblem is the following linear monotone inclusion:
\begin{equation}\label{eq:opt_cp_sub}
0\in \nabla f(\xb^k)+\nabla^2 f(\xb^k)(\zb^k-\xb^k)+\partial g(\zb^k).
\end{equation}
Here, we note that $\dom{Q_f(\cdot;\xb^k)} = \R^p$. Hence, $\dom{Q_f(\cdot;\xb^k) + g(\cdot)} = \dom{g}$. 
In the setting \eqref{eq:composite_cvx}, $\zb^k$ may not be in $\dom{F}$.
Our next step is to update the next iteration $\xb^{k+1}$ as
\begin{equation}\label{eq:proximal_Newton_scheme}
\xb^{k+1}:=\xb^k+\tau_k\pntdir^k=(1-\tau_k)\xb^k+\tau_k\zb^k,
\end{equation}
where $\pntdir^k := \zb^k - \xb^k$ is the proximal Newton direction, and $\tau_k\in (0,1]$ is a given step size.

Associated with the proximal Newton direction $\pntdir^k$, we define the following proximal Newton decrement and the $\ell_2$-norm quantity of $\pntdir^k$ as
\begin{equation}\label{eq:PNT_decrement}
\lambda_k := \norms{\pntdir^k}_{\xb^k}~~~~\textrm{ and }~~~~\beta_k := M_f\norms{\pntdir^k}_2.
\end{equation}
Our first goal is to show that we can explicitly compute the step-size $\tau_k$ in \eqref{eq:proximal_Newton_scheme} using $\lambda_k$ and $\beta_k$ such that we obtain a descent property for $F$.
This statement is presented in the following theorem whose proof is deferred to Appendix~\ref{apdx:th:comp_decr}.

\needcheck{
\begin{theorem}\label{th:comp_decr}
Let $\set{\xb^k}$ be the sequence generated by the proximal Newton scheme \eqref{eq:proximal_Newton_scheme} starting from $\xb^0\in\dom{F}$. 
If we choose the step-size $\tau_k$ as in \eqref{eq:step_size} of Theorem \ref{th:damped_step_NT}, then $\tau_k \in (0, 1]$, $\set{\xb^k}$ in $\dom{F}$ and
\begin{equation}\label{eq:stp_comp}
F(\xb^{k+1}) \leq F(\xb^k) - \Delta_k,
\end{equation}
where $\Delta_k := \lambda_k^2\tau_k - \omega_{\nu}\left( \tau_k d_k\right)\tau^2_k\lambda_k^2 > 0$ for $\tau_k > 0$ and $d_k$ as defined in Theorem \ref{th:damped_step_NT}.

There exists a neighborhood $\Nc(\xopt)$ of  the unique solution $\xopt$ of \eqref{eq:composite_cvx} such that if we initialize the scheme~\eqref{eq:proximal_Newton_scheme} at $\xb^0\in\Nc(\xopt)\cap\dom{F}$, then $\set{\xb^k}$ quadratically converges to $\xopt$.
\end{theorem}}

Next, we prove a local quadratic convergence of the full-step proximal Newton method \eqref{eq:proximal_Newton_scheme} with the unit step-size $\tau_k = 1$ for all $k\geq 0$.
The proof is given in Appendix~\ref{apdx:th:comp_full_pNT_scheme}.

\needcheck{
\begin{theorem}\label{th:comp_full_pNT_scheme}
Suppose that the sequence $\set{\xb^k}$ is generated by \eqref{eq:proximal_Newton_scheme} with full-step, i.e., $\tau_k=1$ for $k\geq 0$. 
Let $d_{\nu}^k:=d_{\nu}(\xb^k,\xb^{k+1})$ be defined by \eqref{eq:dxy_def} and $\lambda_k$ be defined by \eqref{eq:PNT_decrement}. 
Then, the following statements hold:
\begin{enumerate}
\item[$\mathrm{(a)}$] If $\nu = 2$ and the starting point $\xb^0$ satisfies $\underline{\sigma}_0^{-1/2}\lambda_0 < d_2^{\star}/M_f$, then both sequences $\set{\underline{\sigma}_k^{-1/2}\lambda_k}$ and $\set{d_2^k}$ decrease and quadratically converge  to zero, where $d_2^{\star}\approx 0.35482$. 

\vspace{1ex}
\item[$\mathrm{(b)}$] If $2 < \nu < 3$, and the starting point $\xb^0$ satisfies $\underline{\sigma}_0^{-\frac{3-\nu}{2}}\lambda_0 < \frac{1}{M_f}\min\set{\frac{2d_{\nu}^{\star}}{\nu-2},\frac{1}{2}}$, then both sequences $\set{\underline{\sigma}_k^{-\frac{3-\nu}{2}}\lambda_k}$ and $\set{d_{\nu}^k}$ decrease and quadratically converge to zero, where $d_{\nu}^{\star}$ is the unique solution to the equation $\left(\nu -2\right)R_{\nu}(d_{\nu})= 4(1-d_{\nu})^{\frac{4-\nu}{\nu-2}}$ in $d_{\nu}$ with $R_{\nu}(\cdot)$ given in \eqref{eq:R_alpha}.

\vspace{1ex}
\item[$\mathrm{(c)}$] If $\nu = 3$ and the starting point $\xb^0$ satisfies $\lambda_0 < \frac{2d_3^{\star}}{M_f}$, then the sequence $\set{\lambda_k}$ decreases and quadratically converges to zero, where $d_3^{\star}\approx 0.20943$.
\end{enumerate}
As a consequence, if $\set{d^k_{\nu}}$ locally converges  to zero at a quadratic rate, then $\big\{\Vert\xb^k - \xopt\Vert_{\Hb_k}\big\}$ also locally converges to zero at a quadratic rate, where $\Hb_k = \Id$, the identity matrix, if $\nu = 2$; $\Hb_k = \nabla^2{f}(\xb^k)$ if $\nu = 3$; and $\Hb_k = \nabla^2f(\xb^k)^{\frac{\nu}{2}-1}$ if $2 < \nu < 3$.
Hence, $\set{\xb^k}$ locally converges to $\xopt$, the unique solution of \eqref{eq:composite_cvx}, at a quadratic rate.
\end{theorem}}

Similar to Algorithm~\ref{alg:Newton_alg}, we can also combine the results of Theorems~\ref{th:comp_decr} and \ref{th:comp_full_pNT_scheme} to design a proximal Newton algorithm for solving \eqref{eq:composite_cvx}.
This algorithm is described in Algorithm~\ref{alg:prox_NT_alg} below.

\begin{algorithm}[ht!]\caption{(\textit{Proximal Newton algorithm for composite generalized self-concordant minimization})}\label{alg:prox_NT_alg}
\begin{normalsize}
\begin{algorithmic}[1]
   \State {\bfseries Inputs:} Choose an arbitrary initial point $\xb^0 \in\dom{F}$ and a desired accuracy $\varepsilon > 0$.
   \vspace{0.75ex}
   \State {\bfseries Output:}  An $\varepsilon$-solution $\xb^k$ of \eqref{eq:composite_cvx}.
   \vspace{0.75ex}
   \State {\bfseries Initialization:} Compute $d_{\nu}^{\star}$ according to Theorem \ref{th:comp_full_pNT_scheme} if needed.
   \vspace{0.75ex}
   \State{\bfseries For} $k = 0, \cdots, k_{\max}$, \textbf{perform:}
   \vspace{0.75ex}
   \State\hspace{0.16cm}\label{step:alg2_step1}~Compute the proximal Newton direction $\pntdir^k$ by solving \eqref{eq:dir_ds}.
   \vspace{0.75ex}
   \State\hspace{0.16cm} Compute $\lambda_k := \Vert\pntdir^k\Vert_{\xb^k}$, and compute $\beta_k := M_f\Vert \pntdir^k\Vert_2$ if $\nu \neq 3$.
   \vspace{0.75ex}
   \State\hspace{0.16cm} If $\lambda_k \leq \varepsilon$, then TERMINATE.
   \vspace{0.75ex}
   \State\hspace{0.16cm} If \textit{Phase 2 is used}, then compute $\underline{\sigma}_k =\lambda_{\min}(\nabla^2 f(\xb^k))$ if $2 \leq \nu < 3$.
   \vspace{0.75ex}
   \State\hspace{0.16cm} If \textit{Phase 2 is used} and $(\lambda_k, \underline{\sigma}_k)$  satisfies  Theorem~\ref{th:comp_full_pNT_scheme}, then set $\tau_k \!:=\! 1$ (\textbf{full-step}).
   \vspace{0.75ex}
   \Statex\hspace{0.16cm} Otherwise, compute the step-size $\tau_k$ by \eqref{eq:step_size} (\textbf{damped-step}).
   \vspace{0.75ex}
   \State\hspace{0.16cm} Update $\xb^{k+1}:=\xb^k + \tau_k \pntdir^k$.
   \vspace{0.75ex}
   \State {\bfseries End for}
\end{algorithmic}
\end{normalsize}
\end{algorithm}

\beforepar
\paragraph{\textbf{Implementation remarks:}}
The main step of Algorithm~\ref{alg:prox_NT_alg} is the computation of the proximal Newton step $\pntdir^k$, or the trial point $\zb^k$ in \eqref{eq:dir_ds}.
This step requires to solve a composite quadratic-convex minimization problem  \eqref{eq:dir_ds} with strongly convex objective function. 
If $g$ is proximally tractable, then we can apply proximal-gradient methods or splitting techniques  \cite{Bauschke2011,Beck2009,Nesterov2007} to solve this problem.
We can also combine accelerated proximal-gradient methods with a restarting strategy \cite{fercoq2016restarting,Giselsson2014,Odonoghue2012} to accelerate the performance of these algorithms.
These methods will be used in our numerical experiments in Section~\ref{sec:num_experiments}.

As noticed in Remark~\ref{re:inexact_NT}, we can also develop an inexact proximal Newton variant for Algorithm~\ref{alg:prox_NT_alg} by approximately solving the subproblem~\eqref{eq:dir_ds}. 
We leave this extension to our forthcoming work.

\beforesec
\section{Quasi-Newton methods for generalized self-concordant minimization}\label{sec:quasi_newton}
\aftersec
This section studies quasi-Newton variants of Algorithm~\ref{alg:Newton_alg} for solving \eqref{eq:gsc_min}.
Extensions to the composite form \eqref{eq:composite_cvx} can be done by combining the result in this section and the approach in \cite{Tran-Dinh2013a}.

A quasi-Newton method for solving \eqref{eq:gsc_min} updates the sequence $\set{\xb^k}$ using 
\begin{equation}\label{eq:quasi_Newton_scheme}
\xb^{k+1} := \xb^k - \tau_k\Bb_k\nabla{f}(\xb^k),~~~\text{where}~~\Bb_k := \Hb_k^{-1}~\text{and}~~\Hb_k \approx \nabla^2{f}(\xb^k),
\end{equation} 
where the step-size $\tau_k\in (0, 1]$ is appropriately chosen, and $\xb^0\in\dom{f}$ is a given starting point.

Matrix $\Hb_k$ is symmetric and positive definite, and it approximates the Hessian matrix $\nabla^2{f}(\xb^k)$ of $f$ at the iteration $\xb^k$ in some sense. 
The most common approximation sense is that $\Hb_k$ satisfies the well-known Dennis-Mor\'{e} condition \cite{Dennis1974}. 
In the context of generalized self-concordant functions, we can modify this condition by imposing:
\begin{equation}\label{eq:DM_cond}
\lim_{k\to\infty} \frac{\Vert(\Hb_k - \nabla^2{f}(\xopt_f))(\xb^k - \xopt_f)\Vert_{\hat{\xb}}^{\ast}}{\Vert\xb^k - \xopt_f\Vert_{\hat{\xb}}}  = 0,~~\text{where $\hat{\xb} = \xopt_f$ or $\hat{\xb} = \xb^k$}.
\end{equation}
Clearly, if we have $\lim_{k\to\infty}\Vert\Hb_k - \nabla^2{f}(\xb^k)\Vert_{\hat{\xb}} = 0$, then, with a simple argument, we can show that \eqref{eq:DM_cond} automatically holds.
In practice, we can update $\Hb_k$ to maintain the following \textit{secant equation}:
\begin{equation}\label{eq:secant_eq}
\Hb_{k+1}\sbb^k = \yb^k, ~~~\text{where}~~\sbb^k := \xb^{k+1} - \xb^k, ~~\text{and}~~~\yb^k := \nabla{f}(\xb^{k+1}) - \nabla{f}(\xb^k).
\end{equation}
There are several candidates to update $\Hb_k$ to maintain this secant equation, see, e.g., \cite{Nocedal2006}. Here, we propose to use a BFGS update as
\begin{equation}\label{eq:bfgs_update}
\Hb_{k+1} := \Hb_k + \frac{\yb^k(\yb^k)^{\top}}{\iprods{\yb^k,\sbb^k}} - \frac{(\Hb_k\sbb^k)(\Hb_k\sbb^k)^{\top}}{(\iprods{\Hb_k\sbb^k, \sbb^k}}.
\end{equation}
In practice, to avoid the inverse $\Bb_k = \Hb^{-1}_k$, we can update this inverse directly \cite{Nocedal2006} in lieu of updating $\Hb_k$ as in \eqref{eq:bfgs_update}.
Note that the BFGS update \eqref{eq:bfgs_update} or its inverse version may not maintain the sparsity or block pattern structures of the sequence $\set{\Hb_k}$ or $\set{\Bb_k}$ even if $\nabla^2{f}$ is sparse. 

The following result shows that the quasi-Newton method \eqref{eq:quasi_Newton_scheme} achieves a superlinear convergence whose proof can be found in Appendix~\ref{apdx:th:quasi_newton_alg}.

\needcheck{
\begin{theorem}\label{th:quasi_newton_alg}
Assume that $\xopt_f \in \dom{f}$ is the unique solution of \eqref{eq:gsc_min} and is strongly regular.
Let $\set{\xb^k}$ be the sequence generated by \eqref{eq:quasi_Newton_scheme}. Then, the following statements hold:
\begin{itemize}
\item[$\mathrm{(a)}$]
Assume, in addition, that the sequence of matrices $\set{\Hb_k}$ satisfies the Dennis-Mor\'{e} condtion~\eqref{eq:DM_cond} with $\hat{\xb} = \xopt_f$.
Then, there exist $\bar{r} > 0$, and $\bar{k} \geq 0$ such that,  for all $k \geq \bar{k}$, we have $\Vert \xb^k - \xopt_f\Vert_{\xopt_f} \leq \bar{r}$ and $\set{\xb^k}$ locally converges to $\xopt_f$ at a superlinear rate.

\vspace{1ex}
\item[$\mathrm{(b)}$]
Suppose that $\Hb_0$ is chosen such that $\Hb_0\in\Sc^p_{++}$. 
Then, $\iprods{\yb^k, \zb^k} > 0$ for all $k\geq 0$, and hence, the sequence $\set{\Hb_k}$ generated by \eqref{eq:bfgs_update} is symmetric positive definite, and satisfies the secant equation~\eqref{eq:secant_eq}. 
Moreover, if the sequence $\set{\xb^k}$ generated by \eqref{eq:quasi_Newton_scheme} satisfies $\sum_{k=0}^{\infty}\Vert\xb^k - \xopt_f\Vert_{\xopt_f} < +\infty$, then $\set{\xb^k}$ locally converges to the unique solution $\xopt_f$ of \eqref{eq:gsc_min} at a superlinear rate.
\end{itemize}
\end{theorem}}

Note that the condition $\sum_{k=0}^{\infty}\Vert\xb^k - \xopt_f\Vert_{\xopt_f} < +\infty$ in Theorem~\ref{th:quasi_newton_alg}(b) can be guaranteed if $\Vert\xb^{k+1} - \xopt_f\Vert_{\xopt_f} \leq \rho\Vert\xb^k -\xopt_f\Vert_{\xopt_f}$ for some $\rho\in (0, 1)$ and $k\geq \bar{k} \geq 0$.
Hence, if $\set{\xb^k}$ locally converges to $\xopt_f$ at a linear rate, then it also locally converges to $\xopt_f$ at a superlinear rate.

\beforesec
\section{Numerical experiments}\label{sec:num_experiments}
\aftersec
We provide five examples to verify our theoretical results and compare our methods  with existing methods in the leterature.
Our algorithms are implemented in Matlab 2014b running on a MacBook Pro. Retina, 2.7 GHz Intel Core i5 with 16Gb 1867 MHz DDR3 memory.

\beforesubsec
\subsection{\bf Comparison with \cite{zhang2015disco} on regularized logistic regression}\label{subsec:num_ex_logistic}
\aftersubsec
In this example, we empirically show that our theory provides a better step-size for logistic regression compared to \cite{zhang2015disco} as theoretically shown in Example~\ref{ex:compare_step_size}. 
In addition, our step-size can be used to guarantee a global convergence of Newton method without linesearch.
It can also be used as a lower bound for backtracking or forward linesearch to enhance the performance of Algorithm~\ref{alg:Newton_alg}.

To illustrate these aspects, we consider the following regularized logistic regression problem:
\begin{equation}\label{eq:logistic_reg_exam}
f^{\star} := \min_{\xb\in\R^p}\Big\{ f(\xb) := \frac{1}{n}\sum_{i=1}^n\ell( \yb_i(\ab_i^{\top}\xb + \mu)) + \frac{\gamma}{2}\Vert\xb\Vert_2^2 \Big\},
\end{equation}
where $\ell(s) = \log(1 + e^{-s})$ is the logistic loss, $\mu$ is a given intercept, $\yb_i \in \set{-1,1}$ and $\ab_i\in \R^p$ are  given as input data for $i=1,\cdots, n$, and $\gamma > 0$ is a given regularization parameter.

As shown previously in Proposition~\ref{pro:generalized_linear_func_with_regularizer}, $f$ can be cast into an $(M^{(3)}_f, 3)$-generalized self-concordant function with $M^{(3)}_f = \frac{1}{\sqrt{\gamma}}\max\set{\Vert\ab_i\Vert_2 \mid 1\leq i\leq n}$.
On the other hand, $f$ can also be considered as an $(M_f^{(2)}, 2)$-generalized self-concordant with $M_f^{(2)} := \max\set{\Vert\ab_i\Vert_2 \mid 1\leq i\leq n}$.

We implement Algorithm~\ref{alg:Newton_alg} using two different step-sizes $\tau_k^{(2)} = \frac{\ln(1 + \beta_k)}{\beta_k}$ and $\tau^{(3)}_k := \frac{1}{1 + 0.5M^{(3)}_f\lambda_k}$ as suggested by Theorem~\ref{th:damped_step_NT} for $\nu = 2$ and $\nu = 3$, respectively.
We terminate Algorithm \ref{alg:Newton_alg} if $\Vert\nabla{f}(\xb^k)\Vert_2 \leq 10^{-8}\max\set{1, \Vert\nabla{f}(\xb^0)\Vert_2}$, where $\xb^0 = \boldsymbol{0}$ is an initial point.
To solve the linear system \eqref{alg_1:cpt}, we apply a conjugate gradient method to avoid computing the inverse $\nabla^2{f}(x^k)^{-1}$ of the Hessian matrix $\nabla^2{f}(x^k)$ in large-scale problems.
We also compare our algorithms with the fast gradient method in \cite{Nesterov2004} using the optimal step-size for strongly convex functions which has an optimal linear convergence rate.

We test all algorithms on a binary classification dataset downloaded from \cite{CC01a} at \url{https://www.csie.ntu.edu.tw/~cjlin/libsvm/}.
As suggested in \cite{zhang2015disco}, we normalize the data such that each row $\ab_i$ has $\norms{\ab_i}_2 = 1$ for $i=1,\cdots, n$. 
The  parameter is set to $\gamma := 10^{-5}$ as in \cite{zhang2015disco}.

The convergence behavior of Algorithm~\ref{alg:Newton_alg} for $\nu = 2$ and $\nu = 3$ is plotted in Figure~\ref{fig:logistic_convg} for the \texttt{news20} problem.
\begin{figure}[ht!]
\begin{center}
\includegraphics[width = 1\textwidth]{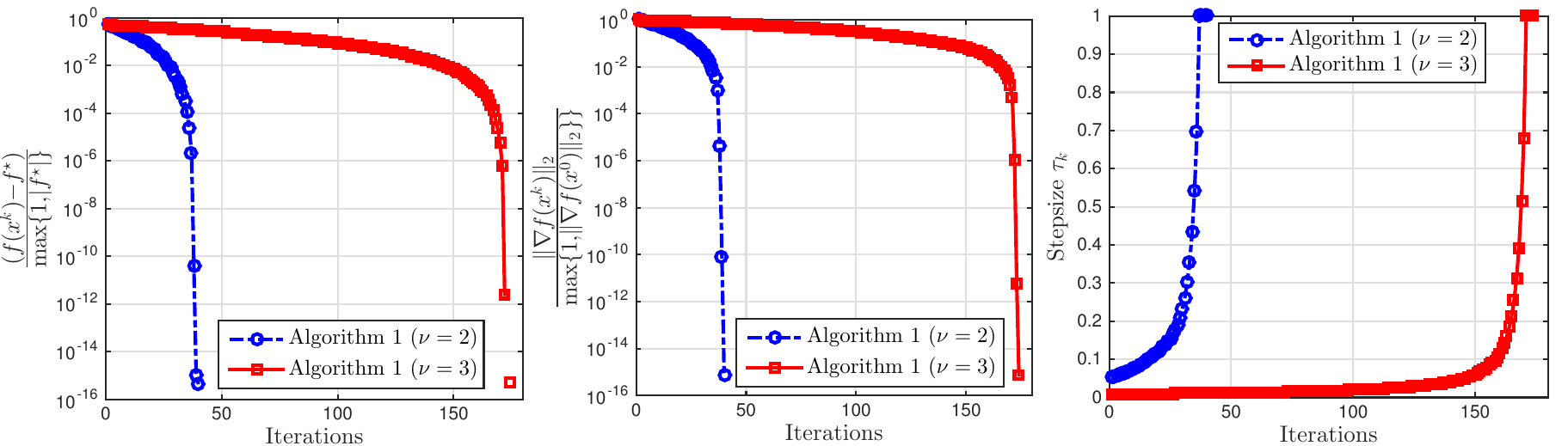}
\vspace{-3ex}
\caption{The convergence of Algorithm~\ref{alg:Newton_alg} for \texttt{news20.binary} (Left: Relative objective residuals, Middle: Relative norms of gradient, and Right: Stepsizes).}\label{fig:logistic_convg}
\end{center}
\vspace{-5ex}
\end{figure}
As we can see from this figure that Algorithm~\ref{alg:Newton_alg} with $\nu = 2$ outperforms the case $\nu = 3$. 
The right-most plot reveals the relative objective residual $\frac{f(\xb^k) - f^{\star}}{\max\set{1, \abs{f^{\star}}}}$, the middle one shows the relative gradient norm $\frac{\Vert\nabla{f}(\xb^k)\Vert_2}{\max\set{1, \Vert\nabla{f}(\xb^0)\Vert_2} }$, and the left-most figure displays the step-size $\tau_k^{(2)}$ and $\tau_k^{(3)}$.
Note that the step-size $\tau_k^{(3)}$ of Algorithm~\ref{alg:Newton_alg} depends on the regularization parameter $\gamma$.
If $\gamma$ is small, then $\tau_k^{(3)}$ is also small.
In contrast, the step-size $\tau_k^{(2)}$ of Algorithm~\ref{alg:Newton_alg} is independent of $\gamma$.

Our second test is performed  on  six problems with different sizes. 
Table~\ref{tbl:logis_ncp} shows the performance and results of the $3$ algorithms: Algorithm~\ref{alg:Newton_alg} with $\nu = 2$, Algorithm~\ref{alg:Newton_alg} with $\nu = 3$, and the fast-gradient method in \cite{Nesterov2004}. 
Here, $n$ is the number of data points, $p$ is the number of variables, \texttt{iter} is the number of iterations, \texttt{error} is the training error measured by $\frac{1}{2n}\sum_{i=1}^n(1-\mathrm{sign}(y_i(a_i^{\top}x + \mu)))$, and $f(x^k)$ is the objective value achieved by these three algorithms.  
\begin{table}[H]
\vspace{-3ex}
\newcommand{\cell}[1]{{\!\!\!}#1{\!\!}}
\newcommand{\cellf}[1]{{\!\!\!}#1{\!\!\!}}
\newcommand{\cellbf}[1]{{\!\!}{#1}{\!\!\!}}
\rowcolors{2}{white}{black!15!white}
\begin{center}
\caption{The performance and results of the three algorithms for solving the logistic regression problem~\eqref{eq:logistic_reg_exam}.}\label{tbl:logis_ncp}
\begin{scriptsize}
\begin{tabular}{ lrr| rrrr | rrrr | rrrrr }
\toprule
\multicolumn{3}{c|}{Problem} & \multicolumn{4}{c|}{ \cell{Algorithm~\ref{alg:Newton_alg}~($\nu = 2$)} } & \multicolumn{4}{c|}{\cell{Algorithm~\ref{alg:Newton_alg}~($\nu = 3$)}} & \multicolumn{4}{c}{Fast gradient method \cite{Nesterov2004}}  \\  \midrule
\multicolumn{1}{c}{\cell{Name}} & \multicolumn{1}{c}{\cell{$p$}} & \multicolumn{1}{c|}{\cell{$n$}} & \cell{iter} & \cell{\!time[s]\!}  & \cell{$f(x^k)$} & \cell{error} &  \cell{iter} & \cell{\!time[s]\!}  & \cell{$f(x^k)$} & \cell{error} &  \cell{iter} &  \cell{\!time[s]\!} &  \cell{$f(x^k)$} & \cell{error} \\ \midrule
\cell{a4a} & \cell{122} & \cell{4781} & \cellbf{22} & \cell{0.57} & \cell{3.250e-01} & \cell{0.150} & \cell{177} & \cell{4.99} & \cell{3.250e-01} & \cell{0.150} & \cell{1396} & \cell{2.13} & \cell{3.250e-01} & \cell{0.150} \\ \midrule
\cell{w4a} & \cell{300} & \cell{6760} & \cellbf{27} & \cell{1.14} & \cell{5.297e-02} & \cell{0.013} & \cell{246} & \cell{8.41} & \cell{5.297e-02} & \cell{0.013} & \cell{863} & \cell{1.71} & \cell{5.297e-02} & \cell{0.013} \\ \midrule
\cell{covtype} & \cell{54} & \cell{581012} & \cellbf{23} & \cell{17.22} & \cell{7.034e-04} & \cell{0.488} & \cell{272} & \cell{235.40} & \cell{7.034e-04} & \cell{0.488} & \cell{1896} & \cell{318.32} & \cell{7.034e-04} & \cell{0.488} \\ \midrule
\cell{rcv1} & \cell{47236} & \cell{20242} & \cellbf{39} & \cell{12.45} & \cell{1.085e-01} & \cell{0.009} & \cell{218} & \cell{60.80} & \cell{1.085e-01} & \cell{0.009} & \cell{366} & \cell{9.69} & \cell{1.085e-01} & \cell{0.009} \\ \midrule
\cell{gisette} & \cell{5000} & \cell{6000} & \cellbf{40} & \cell{109.23} & \cell{1.090e-01} & \cell{0.008} & \cell{220} & \cell{507.03} & \cell{1.090e-01} & \cell{0.008} & \cell{2180} & \cell{1183.67} & \cell{1.090e-01} & \cell{0.008} \\ \midrule
\cell{real-sim} & \cell{20958} & \cell{72201} & \cellbf{39} & \cell{22.69} & \cell{1.287e-01} & \cell{0.016} & \cell{218} & \cell{124.37} & \cell{1.287e-01} & \cell{0.016} & \cell{271} & \cell{24.74} & \cell{1.287e-01} & \cell{0.016} \\  \midrule
\cell{news20} & \cell{1355191} & \cell{19954} & \cellbf{42} & \cell{86.47} & \cell{1.602e-01} & \cell{0.005} & \cell{197} & \cell{420.87} & \cell{1.602e-01} & \cell{0.005} & \cell{623} & \cell{153.22} & \cell{1.602e-01} & \cell{0.005} \\\bottomrule 
\end{tabular}
\end{scriptsize}
\end{center}
\vspace{-6ex}
\end{table}

We observe that our step-size $\tau_k^{(2)}$ using $\nu = 2$ works much better than $\tau^{(3)}_k$ using $\nu = 3$ as in \cite{zhang2015disco}. 
This confirms the theoretical analysis in Example~\ref{ex:compare_step_size}.
This step-size can be useful for parallel and distributed implementation, where evaluating the objective values often requires high computational effort due to communication and data transferring.
Note that the computation of the step-size $\tau^{(2)}_k$ in Algorithm~\ref{alg:Newton_alg} only needs $\mathcal{O}(p)$ operations, and do not require to pass over all data points.
Algorithm~\ref{alg:Newton_alg} with $\nu = 2$ also works better than the fast gradient method \cite{Nesterov2004} in this experiment, especially for the case $n\gg 1$. 
Note that the fast gradient method uses the optimal step-size and has a linear convergence rate in this case.

Finally, we show that our step-size $\tau_k^{(2)}$ can be used as a lower bound to enhance a backtracking linesearch procedure in Newton methods.
The Armijo linesearch condition is given as 
\begin{equation}\label{eq:ls_cond}
f(\xb^k + \tau_k\ntdir^k) \leq f(\xb^k) - c_1\tau_k\nabla{f}(\xb^k)^{\top}\ntdir^k,
\end{equation}
where $c_1 \in (0, 1)$ is a given constant.
Here, we use $c_1 = 10^{-6}$ which is sufficiently small.
\begin{itemize}
\item In our backtracking linesearch variant, we search for the best step-size $\tau \in [\tau_k^{(2)}, 1]$. 
This variant requires to compute $\tau_k^{(2)}$ which needs $\mathcal{O}(p)$ operations.
\vspace{1ex}
\item In the standard backtracking linesearch routine, we search for the best step-size $\tau \in (0, 1]$. 
\end{itemize}
Both strategies use a bisection section rule as $\tau \leftarrow \tau/2$ starting from $\tau \leftarrow 1$.
The results on $3$ problems are reported in Table~\ref{tbl:logis_ncp2}.

\begin{table}[ht!]
\newcommand{\cell}[1]{{\!\!\!}#1{\!\!}}
\newcommand{\cellf}[1]{{\!\!\!}#1{\!\!\!}}
\newcommand{\cellbf}[1]{{\!\!}{#1}{\!\!\!}}
\rowcolors{2}{white}{black!15!white}
\begin{center}
\caption{The performance and results of the two linesearch variants of Algorithm~\ref{alg:Newton_alg} for solving~\eqref{eq:logistic_reg_exam}.}\label{tbl:logis_ncp2}
\vspace{-1ex}
\begin{scriptsize}
\begin{tabular}{ lrr| rrrrrr | rrrrrr  }
\toprule
\multicolumn{3}{c|}{Problem} & \multicolumn{6}{c|}{ \cell{Algorithm~\ref{alg:Newton_alg}~(Standard linesearch)} } & \multicolumn{6}{c}{\cell{Algorithm~\ref{alg:Newton_alg}~(Linesearch with $\tau_k^{(2)}$)}} \\  \midrule
 \multicolumn{1}{c}{\cell{Name}} &  \multicolumn{1}{c}{\cell{$p$}} &  \multicolumn{1}{c|}{\cell{$n$}} & \cell{iter} & \cell{nfval} & \cell{\!time[s]\!}  & \cell{$\frac{\Vert\nabla{f}(\xb^k)\Vert_2}{\Vert\nabla{f}(\xb^0)\Vert_2}$} &  \cell{$f(x^k)$} & \cell{error} &  \cell{iter} & \cell{nfval} & \cell{\!time[s]\!} & \cell{$\frac{\Vert\nabla{f}(\xb^k)\Vert_2}{\Vert\nabla{f}(\xb^0)\Vert_2}$}   & \cell{$f(x^k)$} & \cell{error} \\ \midrule
\cell{covtype} & \cell{54} & \cell{581012} & \cell{25} & \cell{68} & \cell{14.99} & \cell{5.8190e-09} & \cell{7.034e-04} & \cell{0.488} & \cell{14} & \cellbf{31} & \cell{9.89} & \cell{1.3963e-11} & \cell{7.034e-04} & \cell{0.488} \\ \midrule
\cell{rcv1} & \cell{47236} & \cell{20242} & \cell{9} & \cell{21} & \cell{1.85} & \cell{1.3336e-11} & \cell{1.085e-01} & \cell{0.009} & \cell{9} & \cellbf{19} & \cell{1.88} & \cell{1.3336e-11} & \cell{1.085e-01} & \cell{0.009} \\ \midrule
\cell{gisette} & \cell{5000} & \cell{6000} & \cell{8} & \cell{22} & \cell{18.28} & \cell{1.2088e-09} & \cell{1.090e-01} & \cell{0.008} & \cell{8} & \cellbf{17} & \cell{19.68} & \cell{1.2088e-09} & \cell{1.090e-01} & \cell{0.008} \\ 
\bottomrule 
\end{tabular}
\end{scriptsize}
\end{center}
\vspace{-4ex}
\end{table}

As shown in Table~\ref{tbl:logis_ncp2}, using the step-size $\tau_k^{(2)}$ as a lower bound for backtracking linesearch also reduces the number of function evaluations in these three problems.
Note that the number of function evaluations depends on the starting point $\xb^0$ as well as the factor $c_1$ in \eqref{eq:ls_cond}. 
If we set $c_1$ too small, then the decrease on $f$ can be small. Otherwise, if we set $c_1$ too high, then our decrement $c_1\tau_k\nabla{f}(\xb^k)^{\top}\ntdir^k$ may never be achieved, and the linesearch condition fails to hold.
If we change the starting point $\xb^0$, the number of function evaluations can significantly be increased.

\beforesubsec
\subsection{\bf The case $\nu = 2$: Matrix balancing}\label{subsec:num_matrix_balancing}
\aftersubsec
We consider the following convex optimization problem originated from matrix balancing \cite{cohen2017matrix}:
\begin{equation}\label{eq:matrix_balancing}
f^{\star} := \min_{x\in\R^p}\Big\{ f(x):=\sum_{1\leq i,j\leq p}a_{ij}e^{x_i-x_j} \Big\},
\end{equation}
where $A = (a_{ij})_{p\times p}$ is a nonnegative square matrix in $\R^{p\times p}$.
Although \eqref{eq:matrix_balancing} is a unconstrained smooth convex problem, its objective function $f$ is not strongly convex and does not have Lipschitz gradient.
Existing gradient-type methods do not have a theoretical convergence guarantee as well as a rule to compute  step-sizes.
However, \eqref{eq:matrix_balancing} is an important problem in scientific computing.

By Proposition~\ref{pro:sum_rule} and Corollary \ref{co:generalized_linear_func1}, $f$ is generalized self-concordant with $M=\sqrt{2}$ and $\nu=2$. 
We implement Algorithm~\ref{alg:Newton_alg} and the most recent method proposed in \cite{cohen2017matrix} (called Boxed-constrained Newton method (BCNM)) to solve \eqref{eq:matrix_balancing}.
Note that \cite{cohen2017matrix} is not directly applicable to  \eqref{eq:matrix_balancing}, but it solves a regularization of this problem.
Since $\nabla^2f(x)$ is not positive definite, we use a projected conjugate gradient gradient (CG) method to solve the linear system in Algorithm~\ref{alg:Newton_alg}.
We use an accelerated projected gradient method (FISTA) \cite{Beck2009} to solve the subproblem for the method in \cite{cohen2017matrix}. 
We terminate these subsolvers using either a tolerance $10^{-9}$ or a maximum of $200$ iterations.
For the outer loop, we terminate Algorithm~\ref{alg:Newton_alg} and BCNM using the same stopping criterion: $\delta{f_k'} :=\norms{\nabla f(x^k)}_2/\max\set{1,\norms{\nabla f(x^0)}_2} \leq 10^{-8}$.
We choose $x^0 := \boldsymbol{0}^p$ as an initial point.

We test both algorithms on several synthetic and real datasets.
The synthetic data is generated as in \cite{parlett1982methods} with different structures. 
The basic matrix $H = (H_{ij})_{p\times p}$ is a $p\times p$ upper Hessenberg matrix defined as $H_{ij}=0$ if $j< i-1$, and $H_{ij} = 1$ otherwise.
$H_1$ differs from $H$ only in that $H_{11}$ is replaced by $p^2$; $H_2$ differs from $H$ only in that $H_{12}$ is replaced by $p^2$; and $H_3 = H + (p^2-1)\Id_p$. 
We use these matrices for $A$ in \eqref{eq:matrix_balancing}.
We take $p = 1000$, $5000$, $10000$, and $15000$. 
We name each problem instance by ``\texttt{Hdy}'', where \texttt{H} stands for Hessenberg, and $\texttt{y} = 10^{-3}p$. 

The real data is downloaded from \href{https://math.nist.gov/MatrixMarket/searchtool.html}{https://math.nist.gov/MatrixMarket/searchtool.html} with different structures from different application fields, suggested by \cite{chen2000balancing}. 
Since we require the matrix $A$ to be nonnegative, we take $A_0 := \max\set{0, A}$ (entry-wise). 
For the real data, if $A$ is highly ill-conditioned, then we add uniform noise $\mathcal{U}[0, \sigma]$ to $A$, where $\sigma = 10^{-5}\max\set{ A_{ij} \mid 1 \leq i, j\leq p}$.

The final results of both algorithms are reported in Table~\ref{tbl:data_real}, where $p$ is the size of matrix $A$; 
\texttt{iter/siter} is the maximum number of Newton-type iterations / CG or FISTA iterations; 
\texttt{time[s]} is the computational time in second;
$\delta{f_k'}$ is the relative gradient norm defined above; $t_{\mathrm{rat}}$ is the ratio of the computational time between Algorithm~\ref{alg:Newton_alg} and BCNM;
and $\delta{x^k}$ is the relative difference between $x^k$ given by Algorithm~\ref{alg:Newton_alg} and BCNM.

\begin{table}[!ht]
\newcommand{\cell}[1]{{\!}#1{\!\!}}
\newcommand{\cells}[1]{{\!\!\!\!}#1{\!\!\!}}
\newcommand{\cella}[1]{{\!\!\!\!\!}#1{\!\!\!\!}}
\newcommand{\cellbf}[1]{{\!\!}{\color{blue}#1}{\!\!\!}}
\rowcolors{2}{white}{black!15!white}
\begin{center}
\caption{Summary of the results of Algorithm~\ref{alg:Newton_alg} and BCNM on $10$ synthetic and $30$ real problem instances}\label{tbl:data_real}
\begin{scriptsize}
\begin{tabular}{lr | rrrr | rrrr | rr}
\toprule
\multicolumn{2}{c|}{ Datasets } & \multicolumn{4}{c|}{ Algorithm~\ref{alg:Newton_alg} } & \multicolumn{4}{c|}{BCNM} & \multicolumn{2}{c}{ Comparison}\\ \midrule
\cell{Name} & \cell{p} & \cell{iter/siter} & \cell{time[s]} & $f(x^k)$ & \cell{$\delta{f_k'}$} & \cell{iter/siter} & \cell{time[s]} & \cell{$f(x^k)$} & \cell{$\delta{f_k'}$} & \cell{t$_{\textrm{rat}}$} & \cell{$\delta{x^k}$}\\ \midrule
\multicolumn{12}{c}{ Synthetic datasets } \\ \midrule
\cell{\texttt{H1d1}} & \cell{1000} & \cell{8/77} & \cell{0.32} & \cell{5.07e+05} & \cell{3.52e-09} & \cell{8/1028} & \cell{1.55} & \cell{5.07e+05} & \cell{1.82e-10} & \cell{4.88} & \cell{4.0e-07}\\ 
\cell{\texttt{H1d5}} & \cell{5000} & \cell{7/66} & \cell{2.54} & \cell{1.45e+07} & \cell{2.50e-10} & \cell{7/648} & \cell{24.99} & \cell{1.45e+07} & \cell{1.73e-10} & \cell{9.84} & \cell{3.8e-08}\\ 
\cell{\texttt{H1d10}} & \cell{10000} & \cell{7/64} & \cell{8.74} & \cell{6.24e+07} & \cell{8.62e-14} & \cell{6/461} & \cell{61.61} & \cell{6.24e+07} & \cell{4.82e-09} & \cell{7.05} & \cell{7.6e-07}\\ 
\cell{\texttt{H1d15}} & \cell{15000} & \cell{7/63} & \cell{18.63} & \cell{1.48e+08} & \cell{3.55e-14} & \cell{6/395} & \cell{120.41} & \cell{1.48e+08} & \cell{3.66e-10} & \cell{6.47} & \cell{2.1e-08}\\ 
\cell{\texttt{H2d5}} & \cell{5000} & \cell{7/62} & \cell{2.53} & \cell{1.45e+07} & \cell{7.34e-10} & \cell{7/640} & \cell{20.36} & \cell{1.45e+07} & \cell{1.88e-10} & \cell{8.04} & \cell{1.1e-07}\\ 
\cell{\texttt{H2d10}} & \cell{10000} & \cell{7/64} & \cell{9.16} & \cell{6.24e+07} & \cell{2.07e-13} & \cell{6/467} & \cell{61.44} & \cell{6.24e+07} & \cell{4.75e-09} & \cell{6.71} & \cell{7.6e-07}\\ 
\cell{\texttt{H2d15}} & \cell{15000} & \cell{7/63} & \cell{19.66} & \cell{1.48e+08} & \cell{3.18e-14} & \cell{6/395} & \cell{119.16} & \cell{1.48e+08} & \cell{3.52e-10} & \cell{6.06} & \cell{1.9e-08}\\ 
\cell{\texttt{H3d5}} & \cell{5000} & \cell{4/32} & \cell{1.34} & \cell{1.25e+11} & \cell{1.22e-11} & \cell{3/15} & \cell{2.28} & \cell{1.25e+11} & \cell{2.47e-11} & \cell{1.70} & \cell{6.7e-11}\\ 
\cell{\texttt{H3d10}} & \cell{10000} & \cell{4/32} & \cell{4.52} & \cell{1.00e+12} & \cell{1.79e-11} & \cell{3/14} & \cell{8.21} & \cell{1.00e+12} & \cell{2.29e-11} & \cell{1.82} & \cell{2.6e-11}\\ 
\cell{\texttt{H3d15}} & \cell{15000} & \cell{4/28} & \cell{8.72} & \cell{3.38e+12} & \cell{1.15e-11} & \cell{3/12} & \cell{18.06} & \cell{3.38e+12} & \cell{2.59e-10} & \cell{2.07} & \cell{4.9e-10}\\ 
\midrule
\multicolumn{12}{c}{ Real datasets } \\ \midrule
\cell{\texttt{bcs}} & \cell{10974} & \cell{4/362} & \cell{43.95} & \cell{2.28e+12} & \cell{2.39e-12} & \cell{9/438} & \cell{87.89} & \cell{2.28e+12} & \cell{9.83e-09} & \cell{2.00} & \cell{2.1e-08}\\ 
\cell{\texttt{bcs}} & \cell{11948} & \cell{4/204} & \cell{31.23} & \cell{9.30e+12} & \cell{1.85e-12} & \cell{14/305} & \cell{91.19} & \cell{9.30e+12} & \cell{8.76e-09} & \cell{2.92} & \cell{4.8e-08}\\ 
\cell{\texttt{bcs}} & \cell{15439} & \cell{4/36} & \cell{11.89} & \cell{1.53e+16} & \cell{1.21e-12} & \cell{3/16} & \cell{19.13} & \cell{1.53e+16} & \cell{1.13e-10} & \cell{1.61} & \cell{4.4e-11}\\ 
\cell{\texttt{bcsm}} & \cell{15439} & \cell{4/28} & \cell{9.86} & \cell{2.18e+11} & \cell{1.98e-12} & \cell{3/12} & \cell{18.06} & \cell{2.18e+11} & \cell{2.52e-10} & \cell{1.83} & \cell{3.3e-10}\\ 
\cell{\texttt{bwm}} & \cell{2000} & \cell{4/800} & \cell{4.06} & \cell{9.13e+07} & \cell{2.62e-11} & \cell{500/1680} & \cell{72.15} & \cell{9.13e+07} & \cell{1.05e-08} & \cell{17.77} & \cell{7.3e-09}\\ 
\cell{\texttt{e40r01}} & \cell{17281} & \cell{5/178} & \cell{59.65} & \cell{9.86e+04} & \cell{3.49e-12} & \cell{4/230} & \cell{92.36} & \cell{9.86e+04} & \cell{1.20e-09} & \cell{1.55} & \cell{4.6e-08}\\ 
\cell{\texttt{e40r05}} & \cell{17281} & \cell{6/279} & \cell{92.71} & \cell{1.02e+05} & \cell{5.09e-13} & \cell{5/476} & \cell{170.58} & \cell{1.02e+05} & \cell{7.07e-10} & \cell{1.84} & \cell{3.0e-08}\\ 
\cell{\texttt{e40r20}} & \cell{17281} & \cell{7/489} & \cell{160.63} & \cell{1.48e+05} & \cell{7.86e-14} & \cell{6/751} & \cell{278.32} & \cell{1.48e+05} & \cell{1.14e-09} & \cell{1.73} & \cell{1.6e-09}\\ 
\cell{\texttt{e40r30}} & \cell{17281} & \cell{7/492} & \cell{159.09} & \cell{1.90e+05} & \cell{6.21e-14} & \cell{6/759} & \cell{260.82} & \cell{1.90e+05} & \cell{1.11e-09} & \cell{1.64} & \cell{2.0e-09}\\ 
\cell{\texttt{e40r40}} & \cell{17281} & \cell{7/486} & \cell{152.54} & \cell{2.36e+05} & \cell{6.09e-14} & \cell{6/726} & \cell{247.59} & \cell{2.36e+05} & \cell{3.15e-09} & \cell{1.62} & \cell{3.8e-09}\\ 
\cell{\texttt{fid011}} & \cell{16614} & \cell{4/434} & \cell{122.21} & \cell{4.55e+11} & \cell{7.23e-12} & \cell{21/465} & \cell{268.17} & \cell{4.55e+11} & \cell{9.56e-09} & \cell{2.19} & \cell{3.6e-09}\\ 
\cell{\texttt{fid019}} & \cell{12005} & \cell{4/241} & \cell{37.62} & \cell{1.69e+10} & \cell{2.06e-12} & \cell{13/306} & \cell{84.94} & \cell{1.69e+10} & \cell{9.18e-09} & \cell{2.26} & \cell{5.3e-08}\\ 
\cell{\texttt{fid035}} & \cell{19716} & \cell{4/261} & \cell{116.65} & \cell{2.78e+10} & \cell{5.24e-12} & \cell{4/295} & \cell{164.79} & \cell{2.78e+10} & \cell{3.67e-09} & \cell{1.41} & \cell{1.1e-08}\\ 
\cell{\texttt{fidm09}} & \cell{4683} & \cell{4/685} & \cell{16.14} & \cell{1.65e+05} & \cell{2.60e-12} & \cell{93/829} & \cell{67.09} & \cell{1.65e+05} & \cell{9.85e-09} & \cell{4.16} & \cell{2.5e-08}\\ 
\cell{\texttt{fidm11}} & \cell{22294} & \cell{3/222} & \cell{118.68} & \cell{4.63e+03} & \cell{2.93e-09} & \cell{3/299} & \cell{178.42} & \cell{4.63e+03} & \cell{9.16e-10} & \cell{1.50} & \cell{1.3e-07}\\ 
\cell{\texttt{fidm13}} & \cell{3549} & \cell{4/667} & \cell{9.17} & \cell{8.73e+02} & \cell{9.86e-14} & \cell{5/653} & \cell{9.49} & \cell{8.73e+02} & \cell{1.68e-09} & \cell{1.03} & \cell{2.7e-08}\\ 
\cell{\texttt{fidm15}} & \cell{9287} & \cell{3/231} & \cell{21.43} & \cell{2.23e+03} & \cell{7.48e-09} & \cell{3/321} & \cell{32.61} & \cell{2.23e+03} & \cell{2.03e-09} & \cell{1.52} & \cell{6.7e-07}\\ 
\cell{\texttt{fidm29}} & \cell{13668} & \cell{4/451} & \cell{82.61} & \cell{1.07e+04} & \cell{1.51e-12} & \cell{12/452} & \cell{135.98} & \cell{1.07e+04} & \cell{9.67e-09} & \cell{1.65} & \cell{1.8e-08}\\ 
\cell{\texttt{fidm33}} & \cell{2353} & \cell{4/397} & \cell{2.62} & \cell{9.70e+03} & \cell{1.31e-12} & \cell{5/585} & \cell{3.99} & \cell{9.70e+03} & \cell{9.88e-09} & \cell{1.53} & \cell{2.4e-08}\\ 
\cell{\texttt{fidm37}} & \cell{9152} & \cell{4/483} & \cell{44.73} & \cell{1.61e+10} & \cell{1.23e-11} & \cell{70/614} & \cell{212.39} & \cell{1.61e+10} & \cell{9.84e-09} & \cell{4.75} & \cell{2.3e-08}\\ 
\cell{\texttt{gre}} & \cell{1107} & \cell{6/595} & \cell{1.23} & \cell{1.07e+03} & \cell{4.27e-10} & \cell{6/927} & \cell{1.93} & \cell{1.07e+03} & \cell{4.72e-09} & \cell{1.57} & \cell{5.6e-08}\\ 
\cell{\texttt{lnsp}} & \cell{3937} & \cell{8/402} & \cell{7.43} & \cell{2.56e+12} & \cell{4.03e-14} & \cell{7/669} & \cell{13.60} & \cell{2.56e+12} & \cell{3.10e-10} & \cell{1.83} & \cell{1.5e-08}\\ 
\cell{\texttt{mah}} & \cell{1258} & \cell{8/77} & \cell{0.45} & \cell{4.57e+05} & \cell{1.97e-11} & \cell{8/1001} & \cell{3.00} & \cell{4.57e+05} & \cell{7.25e-11} & \cell{6.63} & \cell{4.7e-09}\\ 
\cell{\texttt{mem}} & \cell{17758} & \cell{4/32} & \cell{14.51} & \cell{4.57e+02} & \cell{1.53e-13} & \cell{3/15} & \cell{26.57} & \cell{4.57e+02} & \cell{1.19e-11} & \cell{1.83} & \cell{4.8e-11}\\ 
\cell{\texttt{mhd}} & \cell{3200} & \cell{4/165} & \cell{2.22} & \cell{5.09e+01} & \cell{2.39e-14} & \cell{4/437} & \cell{6.26} & \cell{5.09e+01} & \cell{1.94e-09} & \cell{2.82} & \cell{1.7e-07}\\ 
\cell{\texttt{mhd}} & \cell{4800} & \cell{4/136} & \cell{3.97} & \cell{5.30e+01} & \cell{4.79e-14} & \cell{3/423} & \cell{11.88} & \cell{5.30e+01} & \cell{3.30e-09} & \cell{2.99} & \cell{1.3e-07}\\ 
\cell{\texttt{olm}} & \cell{2000} & \cell{8/640} & \cell{3.27} & \cell{2.94e+07} & \cell{2.05e-15} & \cell{7/846} & \cell{4.80} & \cell{2.94e+07} & \cell{1.30e-10} & \cell{1.47} & \cell{2.7e-09}\\ 
\cell{\texttt{olm}} & \cell{5000} & \cell{7/426} & \cell{11.42} & \cell{5.41e+08} & \cell{9.14e-11} & \cell{6/651} & \cell{20.75} & \cell{5.41e+08} & \cell{4.85e-10} & \cell{1.82} & \cell{3.5e-09}\\ 
\cell{\texttt{ora678}} & \cell{2529} & \cell{9/898} & \cell{6.95} & \cell{3.16e+02} & \cell{9.95e-11} & \cell{8/1512} & \cell{11.92} & \cell{3.16e+02} & \cell{8.06e-09} & \cell{1.71} & \cell{1.1e-06}\\ 
\cell{\texttt{pde}} & \cell{2961} & \cell{6/197} & \cell{2.56} & \cell{1.05e+04} & \cell{5.65e-13} & \cell{5/311} & \cell{4.17} & \cell{1.05e+04} & \cell{6.14e-10} & \cell{1.63} & \cell{8.4e-09}\\ 
\bottomrule
\end{tabular}
\end{scriptsize}
\vspace{-5ex}
\end{center}
\end{table}

As we can see from our experiment, both methods give almost the same result in terms of the objective values $f(x^k)$ and approximate solutions $x^k$.
Given the same stopping criteria and solution quality, Algorithm~\ref{alg:Newton_alg} outperforms BCNM in all datasets in terms of average computational time which is specified by $t_{\textrm{rat}} = \frac{\text{time}_{\mathrm{BCNM}}}{\text{time}_{\mathrm{Alg.~\ref{alg:Newton_alg}}}}$ . 
In particular, for many asymmetric and/or ill-conditioned datasets (e.g., \texttt{H2d5}, or \texttt{bwm}), Algorithm~\ref{alg:Newton_alg} is approximately from $8$ to $17$ times faster than BCNM.

\beforesubsec
\subsection{\bf The case $\nu \in (2, 3)$: Distance-weighted discrimination regression.}\label{subsec:num_ex_dwd}
\aftersubsec
In this example, we test the performance of Algorithm~\ref{alg:Newton_alg} on the distance-weighted discrimination (DWD) problem introduced in \cite{marron2007distance}.
In order to directly use Algorithm~\ref{alg:Newton_alg}, we slightly modify the setting in \cite{marron2007distance} to obtain the following form:
\begin{equation}\label{eq:dwd_prob}
{\!\!\!\!\!}f^{\star} := {\!\!\!\!\!}\min_{\xb = [\wb, \xi, \mu]^{\top}\in\R^p}\left\{f(\xb) := \frac{1}{n}\sum_{i=1}^n \frac{1}{(\ab_i^{\top}\wb + \mu y_i + \xi_i)^q} + \cb^{\top}\xi + \frac{1}{2}\left(\gamma_1 \Vert \wb\Vert_2^2 + \gamma_2\mu^2 + \gamma_3\Vert\xi\Vert_2^2\right) \right\},{\!\!\!}
\end{equation}
where $q > 0$, $\ab_i,  y_i$ ($i=1,\cdots, n)$ and $\cb$ are given, and $\gamma_s > 0$ ($s  = 1,2,3$) are  three regularization parameters for $\wb$, $\mu$ and $\xi$, respectively.
Here, the variable $\xb$ consists of the support vector $\wb$, the intercept $\mu$, and the slack  variable $\xi$ as used in \cite{marron2007distance}.
Here, we penalize these variables by using  least squares terms instead of the $\ell_1$-penalty term as in \cite{marron2007distance}.
Note that  the setting \eqref{eq:dwd_prob} is not just limited to the DWD application above, but can also be used to formulate other practical models such as time optimal path planning problems in robotics \cite{Verscheure2009} if we choose an appropriate parameter $q$.

Since $\varphi(t) := \frac{1}{t^q}$ is $(M_{\varphi}, \nu)$-generalized self-concordant with $M_{\varphi} := \frac{q+2}{\sqrt[(q+2)]{q(q+1)}}n^{\frac{1}{q+2}}$ and  $\nu := \tfrac{2(q+3)}{q+2} \in (2, 3)$, using Proposition~\ref{pro:sum_rule}, we can show that $f$ is $(M_f,  \tfrac{2(q+3)}{q+2})$-generalized self-concordant with $M_f := \frac{q+2}{\sqrt[(q+2)]{q(q+1)}}n^{\frac{1}{q+2}}\max\set{\norm{(\ab_i^{\top}, y_i, \eb_i^{\top})^{\top}}_2^{q/(q+2)} \mid 1\leq i\leq n}$ (here, $\eb_i$ is the $i$-th unit vector).
Problem \eqref{eq:dwd_prob} can be transformed into a second-order cone program \cite{Grant2006}, and can be solved by interior-point methods.
For instance, if we choose $q = 1$, then, by introducing intermediate variables $s_i$ and $r_i$, we can transform \eqref{eq:dwd_prob} into a second-order cone program using the fact that $\frac{1}{r_i} \leq s_i$ is equivalent to $\sqrt{(r_i-s_i)^2 + 2^2} \leq (r_i + s_i)$.

We implement Algorithm~\ref{alg:Newton_alg} to solve \eqref{eq:dwd_prob} and compare it with the interior-point method implemented in  commercial software: Mosek. 
We experienced that Mosek is much faster than other interior-point solvers such as SDPT3 \cite{Toh2010} or SDPA \cite{Yamashita2003} in this test.
For instance, Mosek is from $52$ to $125$ times faster than SDPT3 in this example. Hence, we only present the results of Mosek.

We also incorporate Algorithm~\ref{alg:Newton_alg} with a backtracking linesearch using our step-size $\tau_k$ (LS with $\tau_k$) as a lower bound.
Note that since $f$ does not have a Lipschitz gradient map, we cannot apply gradient-type  methods to solve~\eqref{eq:dwd_prob} due to the lack of a theoretical guarantee.

Since we cannot run Mosek on big data sets, we rather test our algorithms and this interior-point solvers on $6$ small and medium size problems using data from \cite{CC01a} (\url{https://www.csie.ntu.edu.tw/~cjlin/libsvm/}). 
We choose the regularization parameters as $\gamma_1 = \gamma_2 = 10^{-5}$ and $\gamma_3 = 10^{-7}$.
Note that if the data set has the size of $(n, p)$, then number of variables in \eqref{eq:dwd_prob} becomes $p+n+1$.
Hence, we use a built-in Matlab conjugate gradient solver to compute the Newton direction $\ntdir^k$.
The initial point $\xb^0$ is chosen as $\wb^0 := \boldsymbol{0}$, $\mu^0 := 0$ and $\xi^0 := \boldsymbol{1}$. 
In our algorithms, we use $\Vert\nabla{f}(\xb^k)\Vert_2 \leq 10^{-8}\max\set{1, \Vert\nabla{f}(\xb^0)\Vert_2}$ as a stopping criterion.

Note that, by the choice of $\gamma_i$ for $i=1,2,3$ as $\gamma_{\min} := \min\set{\gamma_1,\gamma_2, \gamma_3} = 10^{-7} > 0$, the objective function of \eqref{eq:dwd_prob} is strongly convex. 
By Proposition \ref{pro:scvx_lips_gsc}(a), we can cast this function into an ($\hat{M}_f, \hat{\nu})$-generalized self-concordant with $\hat{\nu} = 3$ and $\hat{M}_f :=  \gamma_{\min}^{\frac{-q}{2(q+2)}}M_f$, where $M_f$ is given above.
We also implement Algorithm \ref{alg:Newton_alg} using $\hat{\nu} = 3$ to solve \eqref{eq:dwd_prob}.

\begin{table}[ht!]
\vspace{-4ex}
\newcommand{\cell}[1]{{\!\!\!}#1{\!\!}}
\newcommand{\cellf}[1]{{\!\!\!}#1{\!\!}}
\newcommand{\cellbf}[1]{{\!\!}{#1}{\!\!\!}}
\rowcolors{2}{white}{black!15!white}
\begin{center}
\caption{The performance and results of the four methods for solving the DWD problem~\eqref{eq:dwd_prob}.}\label{tbl:dwd_ncp1}
\begin{tabular}{ lrr | rrr | rrr | rrr | rr  }
\toprule
\multicolumn{3}{c|}{Problem} & \multicolumn{3}{c|}{ \cell{Algorithm~\ref{alg:Newton_alg} } } & \multicolumn{3}{c|}{\cell{Algorithm~\ref{alg:Newton_alg}~(LS with $\tau_k$)}} & \multicolumn{3}{c|}{\cell{Algorithm~\ref{alg:Newton_alg} ($\nu=3$)}} & \multicolumn{2}{c}{\cell{Mosek}} \\  \midrule
\multicolumn{1}{c}{\cell{Name}} & \multicolumn{1}{c}{\cell{$n$}} & \multicolumn{1}{c|}{\cell{$p$}} & \cell{iter} & \cell{\!time[s]\!}  & \cell{$\frac{\Vert\nabla{f}(\xb^k)\Vert_2}{\Vert\nabla{f}(\xb^0)\Vert_2}$} & \cell{iter} & \cell{\!time[s]\!}  & \cell{$\frac{\Vert\nabla{f}(\xb^k)\Vert_2}{\Vert\nabla{f}(\xb^0)\Vert_2}$} & \cell{\!iter\!} & \cell{\!time[s]\!}  & \cell{$\frac{\Vert\nabla{f}(\xb^k)\Vert_2}{\Vert\nabla{f}(\xb^0)\Vert_2}$} & \cell{\!time[s]\!}  & \cell{$\frac{\Vert\nabla{f}(\xb^k)\Vert_2}{\Vert\nabla{f}(\xb^0)\Vert_2}$} \\ \midrule
\multicolumn{14}{c}{$q = 1$} \\ \midrule
\cell{a1a} &\cell{1605} & \cell{119} & \cell{170} & \cell{1.35} & \cell{9.038e-12} & \cell{13} & \cellbf{0.12} & \cell{4.196e-13} & \cell{574} & \cell{5.77} & \cell{7.031e-14} & \cell{0.49} & \cell{1.806e-08} \\
\cell{a2a} &\cell{2265} & \cell{119} & \cell{192} & \cell{2.71} & \cell{1.661e-13} & \cell{12} & \cellbf{0.15} & \cell{8.549e-09} & \cell{633} & \cell{7.67} & \cell{8.903e-09} & \cell{0.50} & \cell{2.858e-08} \\
\cell{a4a} &\cell{4781} & \cell{122} & \cell{247} & \cell{5.60} & \cell{1.180e-13} & \cell{12} & \cellbf{0.27} & \cell{5.380e-10} & \cell{790} & \cell{21.06} & \cell{3.171e-13} & \cell{0.94} & \cell{1.740e-08} \\
\cell{leu} &\cell{38} & \cell{7129} & \cell{54} & \cell{2.71} & \cell{2.214e-10} & \cell{15} &      \cellbf{0.58} & \cell{3.995e-13} & \cell{193} & \cell{10.64} & \cell{5.275e-12} & \cell{0.72} & \cell{2.828e-07} \\
\cell{w1a} &\cell{2270} & \cell{300} & \cell{169} & \cell{2.88} & \cell{9.752e-09} & \cell{13} & \cellbf{0.17} & \cell{4.968e-09} &\cell{676} & \cell{10.44} & \cell{8.678e-09} & \cell{0.50} & \cell{1.561e-08} \\
\cell{w2a} &\cell{3184} & \cell{300} & \cell{193} & \cell{3.32} & \cell{4.532e-13} & \cell{13} & \cellbf{0.27} & \cell{1.428e-09} &\cell{751} & \cell{15.02} & \cell{7.662e-14} & \cell{0.61} & \cell{1.793e-08} \\
\midrule
\multicolumn{14}{c}{$q = 2$} \\ \midrule
\cell{a1a} &\cell{1605} & \cell{119} & \cell{166} & \cell{2.28} & \cell{6.345e-12} & \cell{14} & \cellbf{0.15} & \cell{5.185e-13} & \cell{1372} & \cell{13.62} & \cell{3.299e-09} & \cell{0.48} & \cell{1.617e-09} \\
\cell{a2a} &\cell{2265} & \cell{119} & \cell{186} & \cell{2.63} & \cell{3.028e-12} & \cell{13} & \cellbf{0.22} & \cell{5.015e-09} & \cell{1484} & \cell{16.65} & \cell{5.325e-09} & \cell{0.56} & \cell{3.070e-09} \\
\cell{a4a} &\cell{4781} & \cell{122} & \cell{235} & \cell{5.03} & \cell{8.676e-13} & \cell{13} & \cellbf{0.31} & \cell{4.347e-10} & \cell{1764} & \cell{53.92} & \cell{2.662e-09} & \cell{1.25} & \cell{4.039e-09} \\
\cell{leu} &\cell{38} & \cell{7129} & \cell{57} & \cell{3.08} & \cell{1.631e-10} & \cell{16} & \cellbf{0.63} & \cell{2.754e-12} & \cell{574} & \cell{39.20} & \cell{2.076e-12} & \cell{0.73} & \cell{6.436e-08} \\
\cell{w1a} &\cell{2270} & \cell{300} & \cell{146} & \cell{2.15} & \cell{1.311e-12} & \cell{14} & \cellbf{0.22} & \cell{4.057e-09} & \cell{1533} & \cell{27.26} & \cell{1.110e-09} & \cell{0.59} & \cell{1.295e-09} \\
\cell{w2a} &\cell{3184} & \cell{300} & \cell{165} & \cell{3.43} & \cell{3.397e-09} & \cell{14} & \cellbf{0.29} & \cell{1.187e-09} & \cell{1661} & \cell{30.63} & \cell{8.004e-09} & \cell{0.71} & \cell{1.653e-09} \\
\bottomrule 
\end{tabular}
\end{center}
\vspace{-5ex}
\end{table}

The results and performance of the four algorithms are reported in Table~\ref{tbl:dwd_ncp1} for two cases: $q=1$ and $q=2$.
We can see that Algorithm~\ref{alg:Newton_alg} with $\nu = 2$ outperforms the case $\hat{\nu} = 3$ in terms of iterations.
The case $\nu = 2$ is approximately from $3$ to $13$ times faster than the case $\hat{\nu} = 3$.
This is not surprising since $\hat{M}_f$ depends on $\gamma_{\min}$, and it is large since $\gamma_{\min}$ is small. 
Hence, the stepsize $\tau_k^{(3)}$ computed by using $\hat{M}_f$ is smaller than $\tau_k^{(2)}$ computed from $M_f$ as we have seen in the first example.
Mosek works really well in this example and it is slightly better than Algorithm~\ref{alg:Newton_alg} with $\nu = 2$.
If we combine Algorithm~\ref{alg:Newton_alg} with a backtracking linesearch, then this variant outperforms Mosek.
All the algorithms achieve a very high accuracy in terms of the relative norm of the gradient $\frac{\Vert\nabla{f}(\xb^k)\Vert_2}{\Vert\nabla{f}(\xb^0)\Vert_2}$ which is up to $10^{-8}$.
We emphasize that our methods are highly parallelizable and their performance can be improved by exploiting this structure as studied in \cite{zhang2015disco} for the logistic case.

\beforesubsec
\subsection{\bf The case $\nu = 3$: Portfolio optimization with logarithmic utility functions.}\label{subsec:num_ex_portfolio}
\aftersubsec
In this example, we aim at verifying Algorithm~\ref{alg:prox_NT_alg} for solving the composite generalized self-concordant minimization problem~\eqref{eq:composite_cvx} with $\nu = 3$.
We illustrate this algorithm on the following portfolio optimization problem with logarithmic utility functions \cite{ryu2014stochastic} (scaled by a factor of $\frac{1}{n}$):
\begin{equation}\label{eq:portfolio_exam} 
f^{\star} = \min_{\xb\in\R^p}\set{ f(\xb) := -\sum_{i=1}^n\log(\wb_i^{\top}\xb) \mid \xb\geq 0, ~~\boldsymbol{1}^{\top}\xb  = 1 },
\end{equation}
where $\wb_i\in\R^p_+$ for $i=1,\cdots, n$ are given vectors presenting the returns at the $i$-th period of the assets considered in the portfolio data.
More precisely, as indicated in \cite{borodin2004can}, $\wb_i$ measures the return as the ratio $\wb_{ij} = v_{i,j}/v_{i-1,j}$ between the closing prices $v_{i,j}$ and $v_{i-1,j}$ of the stocks on the current day $i$ and on the previous day $i-1$, respectively; $\boldsymbol{1}\in\R^p$ is a vector of all ones. 
The aim is to find an optimal strategy to assign the proportion of the assets in order to maximize the expected return among all portfolios.

Note that problem \eqref{eq:portfolio_exam} can be cast into an online optimization model \cite{hazan2006efficient}.
The authors in \cite{hazan2006efficient} proposed an online Newton method to solve this problem. 
In this case, the regret of such an online algorithm showing the difference between the objective function of the online counterpart and the objective function of \eqref{eq:portfolio_exam} converges to zero at a rate of $\frac{1}{\sqrt{n}}$ as $n\to\infty$.
If $n$ is relatively small (e.g., $n=1000$), then the online Newton method does not provide a good approximation to \eqref{eq:portfolio_exam}.

Let $\Delta :=\set{\xb\in\R^p \mid  \xb\geq 0, ~~\boldsymbol{1}^{\top}\xb  = 1}$ be the standard simplex, and $g(\xb) := \delta_{\Delta}(\xb)$ be the indicator function of $\Delta$. 
Then, we can formulate \eqref{eq:portfolio_exam} into \eqref{eq:composite_cvx}. 
The function $f$ defined in \eqref{eq:portfolio_exam} is $(M_f,\nu)$-generalized self-concordant with $\nu = 3$ and $M_f=2$.

We implement Algorithm~\ref{alg:prox_NT_alg} using an accelerated projected gradient method \cite{Beck2009,Nesterov2004} to compute the proximal Newton direction.
We also implement the Frank-Wolfe algorithm and its linesearch variant in \cite{Frank1956,Jaggi2013}, and a projected gradient method using Barzilai and Borwein's step-size to solve \eqref{eq:portfolio_exam}.
We name these algorithms by \texttt{FW}, \texttt{FW-LS}, and \texttt{PG-BB}, respectively. 

We emphasize that both \texttt{PG-BB} and \texttt{FW}-LS do not have a theoretical guarantee when solving \eqref{eq:portfolio_exam}. \texttt{FW} has a theoretical guarantee as recently proved in \cite{odor2016frank}, but the complexity bound is rather pessimistic.
We terminate all the algorithms using $\Vert\xb^{k+1}-\xb^k\Vert_2 \leq \varepsilon\max\set{1,\Vert\xb^k\Vert_2}$, where $\varepsilon = 10^{-8}$ in  Algorithm~\ref{alg:prox_NT_alg}, $\varepsilon = 10^{-6}$ in \texttt{PG-BB}, and $\varepsilon = 10^{-4}$ in \texttt{FW} and \texttt{FW-LS}. 
We choose different accuracies for these methods due to the limitation of first-order methods for attaining high accuracy solutions in the last three algorithms.

We test these algorithms on two categories of dataset: synthetic and real stock data.
For the synthetic data, we generate matrix $\Wb$ with given price ratios as described above in Matlab. 
More precisely, we generate $\Wb := \mathrm{ones}(n, p) + \Nc(0, 0.1)$ which allows the closing prices to vary about $10\%$ between two consecutive periods. 
We test with three instances, where $(n, p) = (1000,800)$, $(1000, 1000)$, and $(1000,1200)$, respectively. 
We name these three datasets by \textrm{PortfSyn1}, \textrm{PortfSyn2}, and \textrm{PortfSyn3}, respectively.
For the real data, we download a US stock dataset using an excel tool~\url{http://www.excelclout.com/historical-stock-prices-in-excel/}. 
This tool gives us the closing prices of the US stock market in a given period of time. 
We generate three datasets with different sizes using different numbers of stocks from 2005 to 2016 as described in \cite{borodin2004can}.
We pre-processed the data by removing stocks that are empty or lacking information in the time period we specified.
We name these three datasets by \textrm{Stock1}, \textrm{Stocks2}, and \textrm{Stocks3}, respectively.

The results and the performance of the four algorithms are given in Table~\ref{tbl:portf_cp1}. 
Here, \texttt{iter} gives the number of iterations, \texttt{time} is the computational time in second, \texttt{error} measures the relative difference between the approximate solution $\xb^k$ given by the algorithms and the interior-point solution provided by CVX \cite{Grant2006} with the high precision configuration (up to $1.8\times 10^{-12})$: $\norm{\xb^k -\xb^{\ast}_{\mathrm{cvx}}}/\max\set{1,\norm{\xb^{\ast}_{\mathrm{cvx}}}}$.

\begin{table}[ht!]
\vspace{-3ex}
\newcommand{\cell}[1]{{\!\!}#1{\!\!}}
\newcommand{\cellf}[1]{{\!\!\!}#1{\!\!}}
\newcommand{\cellbf}[1]{{\!\!}{#1}{\!\!\!}}
\rowcolors{2}{white}{black!15!white}
\begin{center}
\caption{The performance and results of the four algorithms for solving the portfolio optimization problem~\eqref{eq:portfolio_exam}.}\label{tbl:portf_cp1}
\vspace{-1ex}
\begin{tabular}{ lrr | rrr | rrr | rrr | rrr  }
\toprule
\multicolumn{3}{c|}{Problem} & \multicolumn{3}{|c|}{ \cell{Algorithm~\ref{alg:prox_NT_alg} } } & \multicolumn{3}{|c|}{\cell{\texttt{PG-BB}}} & \multicolumn{3}{c|}{\cell{\texttt{FW}}} & \multicolumn{3}{c}{\cell{\texttt{FW-LS}}} \\  \midrule
\multicolumn{1}{c}{\cell{Name}} & \multicolumn{1}{c}{\cell{$n$}} & \multicolumn{1}{c|}{\cell{$p$}} & \cell{iter} & \cell{\!time[s]\!}  & \cell{error} & \cell{iter} & \cell{\!time[s]\!}  & \cell{error} & \cell{iter} & \cell{\!time[s]\!}  & \cell{error} & \cell{iter} & \cell{\!time[s]\!}  & \cell{error} \\ \midrule
\multicolumn{13}{c}{Synthetic Data} \\ \midrule
\cell{PortfSyn1} & \cell{1000} & \cell{800} & \cellbf{6} & \cell{5.68} & \cell{2.4e-04} & \cell{645} & \cellbf{3.98} & \cell{2.3e-04} & \cell{15530} & \cell{96.47} & \cell{2.3e-04} & \cell{6509} & \cell{47.88} & \cell{2.3e-04} \\ 
\cell{PortfSyn2} & \cell{1000} & \cell{1000} & \cellbf{6} & \cell{6.96} & \cell{6.8e-05} & \cell{1207} & \cell{11.54} & \cell{7.5e-05} & \cell{17201} & \cell{166.89} & \cell{1.7e-04} & \cell{6664} & \cell{70.15} & \cell{1.4e-04} \\ 
\cell{PortfSyn3} & \cell{1000} & \cell{1200} & \cellbf{7} & \cell{12.91} & \cell{3.2e-04} & \cell{959} & \cell{9.55} & \cell{3.0e-04} & \cell{16391} & \cell{159.28} & \cell{3.3e-04} & \cell{5750} & \cell{64.36} & \cell{3.2e-04} \\ 
\midrule
\multicolumn{13}{c}{Real Data} \\ \midrule
\cell{Stocks1} & \cell{473} & \cell{500} & \cellbf{8} & \cell{1.22} & \cell{7.1e-06} & \cell{736} & \cell{1.22} & \cell{1.9e-06} & \cell{16274} & \cell{24.93} & \cell{7.0e-05} & \cell{2721} & \cell{5.28} & \cell{4.1e-04} \\ 
\cell{Stocks2} & \cell{625} & \cell{723} & \cellbf{8} & \cell{3.71} & \cell{2.7e-05} & \cell{1544} & \cell{4.37} & \cell{8.0e-06} & \cell{11956} & \cell{34.35} & \cell{3.1e-04} & \cell{2347} & \cell{9.33} & \cell{5.2e-04} \\ 
\cell{Stocks3} & \cell{625} & \cell{889} & \cellbf{10} & \cell{6.83} & \cell{5.6e-05} & \cell{1074} & \cell{6.54} & \cell{5.4e-06} & \cell{13027} & \cell{52.89} & \cell{1.7e-04} & \cell{2096} & \cell{8.46} & \cell{7.4e-04} \\ 
\bottomrule 
\end{tabular}
\end{center}
\vspace{-3ex}
\end{table}

From Table ~\ref{tbl:portf_cp1} we can see that Algorithm \ref{alg:prox_NT_alg} has a comparable performance to the first-order methods: \texttt{FW-LS} and \texttt{PG-BB}. 
While our method has a rigorous convergence guarantee, these first-order methods remains lacking a theoretical guarantee. 
Note that Algorithm \ref{alg:prox_NT_alg} and \texttt{PG-BB} are faster than the \texttt{FW} method and its linesearch variant although the optimal solution $\xopt$ of this problem is very sparse.
We also note that \texttt{PG-BB} gives a smaller error to the CVX solution. 
This CVX solution is not the ground-truth $\xopt$ but gives a high approximation to $\xopt$.
In fact, the CVX solution is dense. 
Hence, it is not clear if \texttt{PG-BB} produces a better solution than other methods.

\vspace{-0.5ex}
\beforesubsec
\subsection{\bf Proximal Quasi-Newton method for sparse multinomial logistic regression.}\label{subsec:num_ex_mullogistic}
\aftersubsec
\vspace{-0.25ex}
We apply our proximal Newton and proximal quasi-Newton methods to solve the following sparse multinomial logistic problem studied in various papers including \cite{Krishnapuram2005}:
\begin{equation}\label{eq:mn_logistic_fx}
{\!\!\!\!}F^{\star} \!:=\! \min_{\xb} \Big\{ F(\xb) \!:=\! \Big[\frac{1}{n}\sum_{j=1}^n\Big(\log\big( \sum_{i=1}^{m}e^{\iprods{\wb^{(j)}, \xb^{(i)}}}\! \big) - \sum_{i=1}^{m}\yb_i^{(j)}\iprods{\wb^{(j)}, \xb^{(i)}}\Big)\Big]_{f(\xb)} {\!\!\!\!} + \Big[\gamma\Vert\vec{\xb}\Vert_1\Big]_{g(\xb)} {\!}\Big\},{\!\!\!}
\vspace{-0.5ex}
\end{equation}
where $\xb$ can be considered as a matrix variable of size $m\times p$ formed from $\xb^{(1)}, \cdots, \xb^{(m)}$, $\vec{\cdot}$ is the vectorization operator, and $\gamma > 0$ is a regularization parameter. 
Both $\yb_i^{(j)} \in \set{0, 1}$ and $\wb^{(j)}$ are given as input data for $i=1,\cdots, m$ and $j=1,\cdots, n$. 

The function $f$ defined in \eqref{eq:mn_logistic_fx} has a closed form Hessian matrix. 
However, forming the full Hessian matrix $\nabla^2{f}(\xb)$ requires an intensive computation in large-scale problems when $n \gg 1$. 
Hence, we apply our proximal-quasi-Newton methods in this case. 
As shown in \cite[Lemma 4]{TranDinh2014d}, the function $f$ is $(M_f, \nu)$-generalized self-concordant with $\nu = 2$ and $M_f := \frac{\sqrt{6}}{n}\max\set{\Vert\wb^{(j)}\Vert_2 \mid 1 \leq j \leq n}$.

We implement our proximal quasi-Newton methods  to solve \eqref{eq:mn_logistic_fx} and compare them with the accelerated first-order methods implemented in a well-established software package called TFOCS \cite{Becker2011a}.
We use three different variants of TFOCS: TFOCS with N07 (using Nesterov's 2007 method with two proximal operations per iteration), TFOCS with N83 (using Nesterov's 1983 method with one proximal operation per iteration), and TFOCS with AT (using Auslender and Teboulle's accelerated method). 

We test on a collection of $26$ multi-class datasets downloaded from \url{https://www.csie.ntu.edu.tw/~cjlin/libsvm/}. 
We set the parameter $\gamma$ in \eqref{eq:mn_logistic_fx} at $\gamma := \frac{0.5}{\sqrt{N}}$ after performing a fine tuning. 
We terminate all the algorithms if $\Vert\xb^{k+1} - \xb^k\Vert \leq 10^{-6}\max\set{1, \Vert\xb^k\Vert}$.

We first plot the convergence behavior in terms of iterations of three proximal Newton-type algorithms we proposed in this paper in Figure~\ref{fig:mnlogistic_profiles} (left) for the \texttt{dna} problem with $3$ classes, $2000$ data points, and $180$ features.
\begin{figure}[ht!]
\vspace{-4ex}
\centerline{\includegraphics[width=1\textwidth]{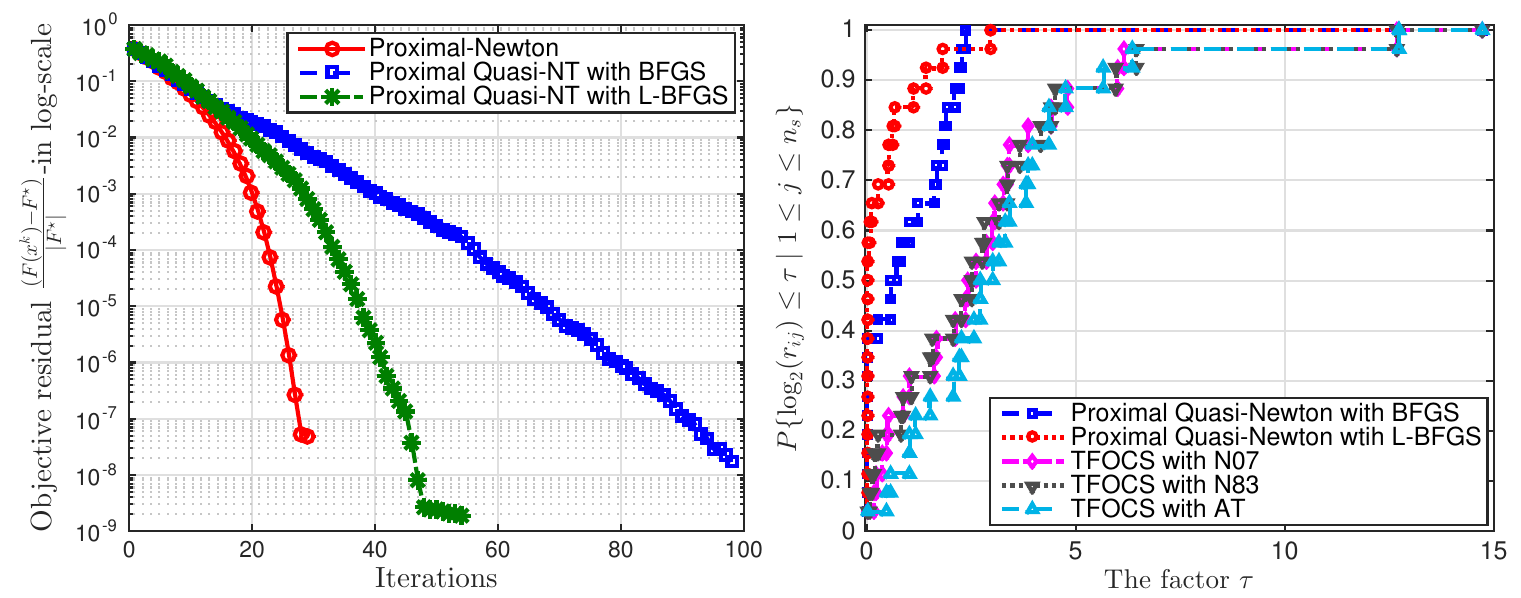}}
\vspace{-1ex}
\caption{Left: Convergence behavior of three methods, ~Right: Performance profile in time [second] of $5$ methods.}\label{fig:mnlogistic_profiles}
\vspace{-4ex}
\end{figure} 

As we can see from this figure, the proximal Newton method takes fewer iterations than the other two methods.
However, each iteration of this method is more expensive than the proximal-quasi-Newton methods due to the evaluation of the Hessian matrix. 
In our experiment, the quasi-Newton method with L-BFGS outperforms the one with BFGS.

Next, we build a performance profile in time [second] to compare five different algorithms: two proximal quasi-Newton methods proposed in this paper (BFGS and L-BFGS), and three variants of the accelerated first-order methods implemented in TFOCS.

The performance profile was studied in \cite{Dolan2002} which can be considered as a standard way to compare different optimization algorithms.
A performance profile is built based on a set $\mathcal{S}$ of $n_s$ algorithms (solvers) and a collection $\mathcal{P}$ of $n_p$ problems. We  build a profile based on computational time. 
We denote by $T_{ij} := \textit{computational time required to solve problem $i$ by solver $j$}$.
We compare the performance of solver $j$ on problem $i$ with the best performance of any algorithm on this problem; that is we compute the performance ratio
$r_{ij} := \frac{T_{ij}}{\min\{T_{ik} \mid k \in \mathcal{S}\}}$.
Now, let $\tilde{\rho}_j(\tilde{\tau}) := \frac{1}{n_p}\mathrm{size}\left\{i\in\mathcal{P} \mid r_{ij}\leq \tilde{\tau}\right\}$ for
$\tilde{\tau} \in\R_{+}$. The function $\tilde{\rho}_j:\R\to [0,1]$ is the probability for solver $j$ that a performance ratio is within a
factor $\tilde{\tau}$ of the best possible ratio. We use the term ``performance profile'' for the distribution function $\tilde{\rho}_j$ of a performance
metric.
In the following numerical examples, we plotted the performance profiles in $\log_2$-scale, i.e. $\rho_j(\tau) :=
\frac{1}{n_p}\mathrm{size}\left\{i\in\mathcal{P} \mid \log_2(r_{i,j})\leq \tau := \log_2\tilde{\tau}\right\}$.

Figure~\ref{fig:mnlogistic_profiles} (right) shows the performance profile of five algorithms on a collection of $26$ problems indicated above.
The proximal quasi-Newton method with L-BFGS achieves $13/26$ ($50\%$) with the best performance, while the BFGS obtains $10/26$ ($38\%$) with the best performance.
In terms of computational time, both proximal quasi-Newton methods outperform the optimal proximal gradient methods in this experiment.
It is also clear that our proximal quasi-Newton-type methods achieve a higher accuracy solution in this experiment compared to the accelerated proximal gradient-type methods implemented in TFOCS.

\beforesec
\section{Conclusion}\label{sec:concl}
\aftersec
We have generalized the self-concordance notion in \cite{Nesterov1994} to a more general class of smooth and convex functions.
Such a function class covers  several well-known examples, including logistic, exponential, reciprocal, and standard self-concordant functions, just to name a few.
We have developed a unified theory  with several basic properties to reveal the smoothness structure of this functional class.
We have provided several key bounds on local norms,  Hessian mapping, gradient mapping, and function value of this functional class.
Then, we have illustrated our theory by applying it to solve a class of smooth convex minimization problems and its composite setting.
We believe that our theory provides an appropriate approach to exploit the curvature of these problems and allows us to compute an explicit step-size in Newton-type methods that have a global convergence guarantee even for non-Lipschitz gradient/Hessian functions.
While our theory is still valid for the case $\nu > 3$, we have not found yet a representative application in a high-dimensional space. 
Therefore, we limit our consideration to Newton and proximal Newton methods for $\nu \in [2, 3]$, but our key bounds in Subsection~\ref{subsec:key_bounds} remain valid for different ranges of $\nu$ with $\nu > 0$.

Our future research is to focus on several aspects. 
Firstly, we can exploit this theory to develop more practical inexact and quasi-Newton-type methods that can easily capture practical applications in large-scale settings.
Secondly, we will combine our approach and stochastic, randomized, and coordinate descent methods to develop new variants of algorithms that can scale better in high-dimensional space.
Thirdly, by exploiting both generalized self-concordant, Lipschitz gradient, and strong convexity, one can also develop first-order methods to solve convex optimization problems.
Finally, we plan to generalize our theory to primal-dual settings and monotone operators to apply to other classes of convex problems such as convex-concave saddle points, constrained convex optimization, and monotone equations and inclusions.

\vspace{1ex}
\noindent\textbf{Acknowledgments:}
This work is partially supported by the NSF-grant No. DMS-1619884, USA.

\appendix
\beforesec
\normalsize
\section{Appendix: The proof of technical results}
\aftersec
This appendix provides the full proofs of technical results presented in this paper. 
We  prove some technical results used in the paper, and missing proofs in the main text. 
We also provide a full convergence analysis of the Newton-type methods presented in the main text.

\beforesubsec
\subsection{\bf The proof of Proposition \ref{pro:conjugate}: Fenchel's conjugate}\label{apdx:pro:conjugate}
\aftersubsec
Let us consider the set $\Xc := \set{x\in\R^p \mid f(u) - \iprods{x, u}~\text{is bounded from below on}~\dom{f}}$. 
We first show that $\dom{f^{\ast}} = \Xc$.

By the definition of $\dom{f^{\ast}}$, we have $\dom{f^{\ast}} = \set{ x\in\R^p \mid f^{\ast}(x) < +\infty}$. 
Take any $x\in\dom{f^{\ast}}$, one has $f^{\ast}(x) = \max_{u\in\dom{f}}\set{ \iprods{x, u} - f(u)} <+\infty$. 
Hence, $f(u) - \iprods{x,u}\geq -f^{\ast}(x) > -\infty$ for all $u\in\dom{f}$ which implies $x\in\Xc$.

Conversely, assume that $x\in\Xc$. By the definition of $\Xc$, $f(u)-\iprods{x,u}$ is bounded from below for all $u\in\dom{f}$. 
That is, there exists $M \in [0, +\infty)$, such that $f(u) - \iprods{x, u} \geq -M$ for all $u\in\dom{f}$.
By the definition of the conjugate, $f^{\ast}(x) = \max_{u\in\dom{f}}\set{ \iprods{x, u} - f(u)} \leq M < +\infty$. 
Hence, $x\in\dom{f^{\ast}}$.

For any $x\in\dom{f^{\ast}}$, the optimality condition of $\max_{u}\set{\iprods{x, u} - f(u)}$ is $x = \nabla{f}(u)$.
Let us denote by $x(u) = \nabla{f}(u)$.
Then, we have $f^{\ast}(x(u)) = \iprods{x(u), u} - f(u)$.
Taking derivative of $f^{\ast}$ with respect to $x$ on both sides, and using $x(u)=\nabla f(u)$, we have
\begin{equation*}
\nabla_x f^{\ast}(x(u)) = u + u'_xx(u) - u'_x\nabla f(u) = u.
\end{equation*}
We further take the second-order derivative of the above equation with respect to $u$ to get
\begin{equation*}
\nabla^2f^{\ast}(x(u))x_u'(u) = \Id.
\end{equation*}
Using the two relations above and the fact that $x_u'(u) = \nabla^2{f}(u)$, we can derive
\begin{align}
\iprods{\nabla f^{\ast}(x(u)),x_u'(u)v} &= \iprods{u,x_u'(u)v} = \iprods{\nabla^2f(u)v, u} \label{eq:conjugate_re1}\\
\iprods{\nabla^2{f^{\ast}}(x(u))x_u'(u)v, x_u'(u)w} &= \iprods{v, x_u'(u)w} = \iprods{\nabla^2f(u)v, w}, \label{eq:conjugate_re2}
\end{align}
where $u\in\dom{f}$, and $v, w\in \R^p$. 
Using \eqref{eq:conjugate_re1} and \eqref{eq:conjugate_re2}, we can compute the third-order derivative of $f^{\ast}$ with respect to $x(u)$ as
\begin{equation}\label{eq:conjuaget_rel1}
{\!\!\!}\begin{array}{rl}
\langle \nabla^3 f^{*}(x(u))[x_u'(u)w]x_u'(u)v, & x_u'(u)v\rangle = \iprods{ \left(\iprods{\nabla^2{f}^{*}(x(u))x_u'(u)v, x_u'(u)v}\right)'_{u}, w}   \vspace{1ex}\\
&- 2\iprods{\nabla^2 f^{*}(x(u))x_u'(u)v, (x_u'(u)v)'_uw} \vspace{1ex}\\
& \overset{\tiny\eqref{eq:conjugate_re1}}{=} \iprods{(\iprods{x_u'(u)v,v})'_u,w} -2\iprods{\nabla^2f^{*}(x(u))x_u'(u)v, (x_u'(u)v)'_uw} \vspace{1ex}\\
&\overset{\tiny\eqref{eq:conjugate_re2}}{=}  \iprods{\nabla^3 f(u)[w]v,v} -2\iprods{(x_u'(u)v)_u'w,v} \vspace{1ex}\\
& = -\iprods{\nabla^3 f(u)[w]v,v}.
\end{array}{\!\!\!}
\end{equation}
Denote $\xi := x_u'(u)w$ and $\eta := x_u'(u)v$.
Note that since $x_u'(u) = \nabla^2{f}(u)$, we have $\xi = \nabla^2{f}(u)w$, $\eta = \nabla^2{f}(u)v$, and $w = \nabla^2{f}(u)^{-1}\xi$.
Using these relations and $\nabla^2f^{\ast}(x(u))x_u'(u) = \Id$, we can derive
\begin{equation*}
\begin{array}{ll}
\vert \langle \nabla^3{f^{\ast}}(x(u))[\xi]\eta, \eta\rangle \vert &\overset{\tiny\eqref{eq:conjuaget_rel1}}{=} \vert \iprods{\nabla^3{f}(u)[w]v, v}  \overset{\tiny\eqref{eq:gsc_def}}{\leq} M_f\norm{v}_u^2\norm{w}_u^{\nu - 2}\norm{w}_2^{3-\nu} \vspace{1ex}\\
&= M_f\iprods{\nabla^2f(u)v, v}\iprods{\nabla^2{f}(u)w, w}^{\frac{\nu-2}{2}}\norms{w}^{3-\nu}_2 \vspace{1ex}\\
&= M_f\iprods{\eta, \nabla^2f^{\ast}(x(u))x'(u)v} \iprods{\xi, \nabla^2f^{\ast}(x(u))x'(u)w}^{\frac{\nu-2}{2}}\norms{\nabla^2{f}(u)^{-1}\xi}^{3-\nu} \vspace{1ex}\\
&= M_f\iprods{\nabla^2f^{\ast}(x(u))\eta, \eta}\iprods{\nabla^2f^{\ast}(x(u))\xi, \xi}^{\frac{\nu-2}{2}}\iprods{\nabla^2f^{\ast}(x(u))\xi, \nabla^2f^{\ast}(x(u))\xi}^{3-\nu}.
\end{array}
\end{equation*}
For any $H\in\Sc^p_{++}$, we have $\iprods{H\xi, \xi} \leq \norm{H\xi}_2\norm{\xi}_2$.
For any $\nu \geq 3$, this inequality leads to
\begin{equation*}
\iprods{H\xi, \xi}^{\frac{\nu-2}{2}}\norm{H\xi}^{3-\nu} \leq \iprods{H\xi,\xi}^{\frac{4-\nu}{2}}\norm{\xi}_2^{\nu-3}.
\end{equation*}
Using this inequality with $H = \nabla^2f^{\ast}(x(u))$ into the last expression, we obtain
\begin{equation*}
\begin{array}{ll}
\vert \iprods{\nabla^3{f^{\ast}}(x(u))[\xi]\eta, \eta} \vert &\leq M_f\iprods{\nabla^2f^{\ast}(x(u))\eta, \eta}\iprods{\nabla^2f^{\ast}(x(u))\xi, \xi}^{\frac{4 - \nu}{2}}\norm{\xi}_2^{\nu-3}\vspace{1ex}\\
&= M_f\norms{\eta}_{x(u)}^2\norm{\xi}_{x(u)}^{4-\nu}\norms{\xi}_2^{\nu - 3}.
\end{array}
\end{equation*}
By Definition~\ref{de:gsc_def}, we need $\nu - 3 = 3 - \nu_{\ast}$ and $4-\nu = \nu_{\ast} - 2$ which hold if $\nu_{\ast} = 6 - \nu$.
Under the choice of $\nu_{\ast}$, the above inequality shows that $f^{\ast}$ is $(M_{f^{\ast}}, \nu_{\ast})$-generalized self-concordant with $M_{f^{\ast}} = M_f$ and  $\nu_{\ast} = 6 - \nu$. 
However, to guarantee $\nu - 3 \geq 0$ and $6 - \nu > 0$, we require $3 \leq \nu < 6$.

Finally, we prove the case of univariate functions, i.e., $p = 1$.
Indeed, we have
\begin{equation}\label{eq:temp_2}
x(u)=f'(u),~~ (f^{\ast})'(x(u))=u,~~\text{and}~~(f^{\ast})''(x(u))x'(u)=1.
\end{equation}
Here, $f'$ is the derivative of $f$ with respect to $u$.
Taking the derivative of the last equation on both sides with respect to $u$, we obtain
\begin{equation*}
(f^{\ast})'''(x(u))(x'(u))^2+(f^{\ast})''(x(u))x''(u)=0.
\end{equation*}
Solving this equation for $(f^{*})'''(x(u))$ and then using \eqref{eq:temp_2} and $x''(u) = f'''(u)$, we get
\begin{equation*}
\begin{array}{rl}
\abs{(f^{\ast})'''(x(u))} =& \abs{\frac{(f^{\ast})''(x(u))x''(u)}{(x'(u))^2}} = \abs{((f^{\ast})''(x(u)))^3f'''(u)}\vspace{1ex}\\
\leq &  M_f\abs{((f^{\ast})''(x(u)))^3(f''(u))^{\frac{\nu}{2}}} =  M_f((f^{\ast})''(x(u)))^{\frac{6-\nu}{2}}.
\end{array}
\end{equation*}
This inequality shows that $f^{\ast}$ is generalized self-concordant with $\nu_{\ast} = 6 - \nu$ for any $\nu \in (0, 6)$.
\Eproof

\beforesubsec
\subsection{\bf The proof of Corollary~\ref{co:hessian_bound2}: Bound on the mean of Hessian operator}\label{apdx:co:hessian_bound2}
\aftersubsec
Let $\yb_{\tau} := \xb + \tau(\yb - \xb)$. Then $d_{\nu}(\xb, \yb_{\tau}) = \tau d_{\nu}(\xb, \yb)$.
By \eqref{eq:hessian_bound1}, we have $\nabla^2{f}(\xb + \tau(\yb - \xb)) \preceq \left(1 - \tau d_{\nu}(\xb,\yb)\right)^{\frac{-2}{\nu-2}}\nabla^2{f}(\xb)$ and $\nabla^2{f}(\xb + \tau(\yb - \xb)) \succeq  \left(1 - \tau d_{\nu}(\xb,\yb)\right)^{\frac{2}{\nu-2}}\nabla^2{f}(\xb)$ .
Hence, we have 
\begin{align*}
\underline{I}_{\nu}(\xb,\yb)\nabla^2{f}(\xb) \preceq \int_0^1\nabla^2{f}(\xb + \tau(\yb - \xb))d\tau \preceq \overline{I}_{\nu}(\xb, \yb)\nabla^2{f}(\xb),
\end{align*}
where $\underline{I}_{\nu}(\xb, \yb) := \int_0^1\left(1 - \tau d_{\nu}(\xb,\yb)\right)^{\frac{2}{\nu-2}}d\tau$ and $\overline{I}_{\nu}(\xb, \yb) := \int_0^1\left(1 - \tau d_{\nu}(\xb,\yb)\right)^{\frac{-2}{\nu-2}}d\tau$ are the two integrals in the above inequality.
Computing these integrals explicitly, we can show that
\begin{itemize}
\item If $\nu = 4$, then  $\underline{I}_{\nu}(\xb,\yb) = \frac{1 - (1 - d_4(\xb,\yb))^2}{2d_4(\xb,\yb)}$ and $\overline{I}_{\nu}(\xb, \yb) = \frac{-\ln(1 - d_4(\xb,\yb))}{d_4(\xb,\yb)}$.

\vspace{1ex}
\item If $\nu \neq 4$, then we can easily compute $\underline{I}_{\nu}(\xb, \yb) = \frac{(\nu-2)}{\nu d_{\nu}(\xb,\yb)}\left(1 - \left(1 - d_{\nu}(\xb,\yb)\right)^{\frac{\nu}{\nu-2}}\right)$, and  $\overline{I}_{\nu}(\xb, \yb) =  \frac{(\nu-2)}{(\nu-4)d_{\nu}(\xb,\yb)}\left(1 - \left(1 - d_{\nu}(\xb,\yb)\right)^{\frac{\nu-4}{\nu-2}}\right)$.
\end{itemize}
Hence, we obtain \eqref{eq:hessian_bound2}.

Finally, we prove for the case $\nu = 2$.
Indeed, by \eqref{eq:hessian_bound1b}, we have $e^{-d_2(\xb,\yb_{\tau})}\nabla^2f(\xb) \preceq \nabla^2f(\yb_{\tau}) \preceq e^{d_2(\xb,\yb_{\tau})}\nabla^2f(\xb)$.
Since $d_2(\xb, \yb_{\tau}) = \tau d_2(\xb, \yb)$, the last estimate leads to
\begin{equation*}
\left(\int_0^1e^{-d_2(\xb,\yb)\tau}d\tau\right)\nabla^2f(\xb) \preceq \int_0^1\nabla^2f(\yb_{\tau})d\tau \preceq \left(\int_0^1e^{d_2(\xb,\yb)\tau}d\tau\right)\nabla^2f(\xb),
\end{equation*}
which is exactly \eqref{eq:hessian_bound2}.
\Eproof

\beforesubsec
\subsection{\bf Techical lemmas}\label{apdx:subsec:tech_lemma}
\aftersubsec
The following lemmas will be used in our analysis. 
Lemma \ref{le:elem_inequalities} is elementary, but we provide its proof for completeness.

\begin{lemma}\label{le:elem_inequalities}
\begin{itemize}
\item[$\mathrm{(a)}$] For a fixed $r \geq 1$ and $\bar{t} \in (0, 1)$, consider a function $\psi_r(t) := \frac{1 - (1-t)^r - rt(1-t)^r}{rt^2(1-t)^r}$  on $t\in (0, 1)$.
Then, $\psi$ is positive and increasing on $(0, \bar{t}]$ and
\begin{equation*}
\lim_{t\to 0^{+}}\psi_r(t) = \tfrac{r+1}{2},~~\lim_{t\to 1^{-}}\psi_r(t) = +\infty,~~\text{and} ~~~\sup_{0 \leq t \leq \bar{t}}\abs{\psi_r(t)} \leq \bar{C}_r(\bar{t}) < +\infty,
\end{equation*}
where $\bar{C}_r(\bar{t}) := \frac{1 - (1-\bar{t})^r - r\bar{t}(1-\bar{t})^r}{r\bar{t}^2(1-\bar{t})^r} \in (0, +\infty)$.

\vspace{1ex}
\item[$\mathrm{(b)}$] For $t > 0$, we also have $\frac{e^{t} - 1 - t}{t} \leq \left(\frac{3}{2} + \frac{t}{3}\right)te^t$.

\end{itemize}
\end{lemma}

\begin{proof}
The statement $\mathrm{(b)}$ is rather elementary, we only prove $\mathrm{(a)}$.
Since $r \geq 1$, $\lim_{t\to 0^{+}}(1 - (1-t)^r - rt(1-t)^r) = \lim_{t\to 0^{+}}rt^2(1-t)^r = 0$ and $rt^2(1-t)^r > 0$ for $t\in (0, 1)$, applying L'H$\hat{\mathrm{o}}$spital's rule, we have
\begin{equation*}
\lim_{t\to0^+}\psi_r(t) = \frac{\lim_{t\to0^+}r(r+1)t(1-t)^{r-1}}{\lim_{t\to0^+}rt(2-(2+r)t)(1-t)^{r-1}}=\frac{\lim_{t\to0^+}(r+1)}{\lim_{t\to0^+}(2-(2+r)t)}=\frac{r+1}{2}.
\end{equation*}
The limit $\lim_{t\to 1^{-}}\psi_r(t) = +\infty$ is obvious.

Next, it is easily to compute $\psi'_r(t) = \frac{(1-t)^{r+1}(rt+2)+(r+2)t-2}{rt^3(1-t)^{r+1}}$.
Let $m_r(t) := (1-t)^{r+1}(rt+2)+(r+2)t-2$ be the numerator of $\psi'_r(t)$.

We have $m_r'(t) = r+2 - (1-t)^r(r^2t+2rt+r+2)$, and $m_r''(t) = r(r+1)(r+2)t(1-t)^{r-1}$.
Clearly, since $r \geq 1$, $m_r''(t) \geq 0$ for $t \in [0, 1]$.
This implies that $m_r'$ is nondecreasing on $[0, 1]$. 
Hence, $m_r'(t) \geq m_r'(0)  = 0$ for all $t \in [0, 1]$.
Consequently, $m_r$ is nondecreasing on $[0, 1]$. 
Therefore, $m_r(t) \geq m_r(0) = 0$ for all $t\in [0, 1]$.
Using the formula of $\psi'_r$, we can see that $\psi'_r(t) \geq 0$ for all $t \in (0, 1)$.
This implies that $\psi_r$ is nondecreasing on $(0, 1)$. 
Moreover, $\lim_{t\to0^+}\psi_r(t) = \frac{r+1}{2} > 0$. 
Hence, $\psi_r(t) > 0$ for all $t\in (0, 1)$.
This implies that $\psi_r$ is bounded on $(0, \bar{t}] \subset (0, 1)$ by $\psi_r(\bar{t})$.
\Eproof
\end{proof}

Similar to Corollary~\ref{co:hessian_bound2}, we can prove the following lemma on the bound of the Hessian difference. 

\begin{lemma}\label{le:H_norm}
Given $\xb, \yb\in\dom{f}$, the matrix $\Hb(\xb,\yb)$ defined by
\begin{equation}\label{eq:H_matrix}
\Hb(\xb,\yb) := \nabla^2f(\xb)^{-1/2}\left[\int_0^1\big(\nabla^2{f}(\xb + \tau(\yb-\xb)) - \nabla^2f(\xb)\big)d\tau\right]\nabla^2f(\xb)^{-1/2},
\end{equation}
satisfies 
\begin{equation}\label{eq:H_matrix_norm}
\Vert \Hb(\xb,\yb) \Vert \leq R_{\nu}\left(d_{\nu}(\xb, \yb)\right)d_{\nu}(\xb, \yb),
\end{equation}
where $R_{\nu}(t)$ is defined as follows for $t \in [0, 1)$:
\begin{equation}\label{eq:R_alpha}
R_{\nu}(t) := \begin{cases}
\left(\frac{3}{2} + \frac{t}{3}\right)e^t &\text{if $\nu = 2$}\vspace{1ex}\\
\frac{1 - (1-t)^{\frac{4-\nu}{\nu-2}} - \left(\frac{4-\nu}{\nu-2}\right)t(1-t)^{\frac{4-\nu}{\nu-2}}}{\left(\frac{4-\nu}{\nu-2}\right)t^2(1-t)^{\frac{4-\nu}{\nu-2}}} &\text{if $2 < \nu \leq 3$}.
\end{cases}
\end{equation}
Moreover, for a fixed $\bar{t} \in (0, 1)$, we have $\displaystyle\sup_{0 \leq t \leq \bar{t}}\abs{R_{\nu}(t)} \leq \bar{M}_{\nu}(\bar{t})$, where 
\begin{equation*}
\bar{M}_{\nu}(\bar{t}) := \max\set{\frac{1 - (1 \!-\! \bar{t})^{\frac{4 \!-\! \nu}{\nu \!-\! 2}} - \left(\frac{4\!-\!\nu}{\nu \!-\! 2}\right)\bar{t}(1 \!-\! \bar{t})^{\frac{4 \!-\! \nu}{\nu \!-\! 2}}}{\left(\frac{4 \!-\! \nu}{\nu \!-\! 2}\right)\bar{t}^2(1 \!-\! \bar{t})^{\frac{4 \!-\! \nu}{\nu-2}}}, \left(\frac{3}{2} + \frac{\bar{t}}{2}\right)e^{\bar{t}} } \in (0, +\infty).
\end{equation*}
\end{lemma}

\begin{proof}
By Corollary \ref{co:hessian_bound2}, if we define $\Gb(\xb,\yb) := \int_0^1 \left[\nabla^2{f}(\xb + \tau(\yb-\xb)) - \nabla^2{f}(\xb)\right]d\tau$, then
\begin{equation}\label{eq:G_matrix}
\left[\underline{\kappa}_{\nu}(d_{\nu}(\xb,\yb)) - 1\right]\nabla^2{f}(\xb) \preceq \Gb(\xb,\yb) \preceq \left[\overline{\kappa}_{\nu}(d_{\nu}(\xb,\yb)) - 1\right]\nabla^2{f}(\xb).
\end{equation}
Since $\Hb(\xb,\yb)  = \nabla^2f(\xb)^{-1/2}\Gb(\xb,\yb)\nabla^2f(\xb)^{-1/2}$, the last inequality implies
\begin{equation*} 
\Vert \Hb(\xb,\yb)  \Vert \leq \max\big\{1 - \underline{\kappa}_{\nu}(d_{\nu}(\xb,\yb)), \overline{\kappa}_{\nu}(d_{\nu}(\xb,\yb)) - 1\big\}.
\end{equation*}
Let $C_{\max}(t) := \max\big\{1 - \underline{\kappa}_{\nu}(t), \overline{\kappa}_{\nu}(t) - 1 \big\}$ be for $t \in [0, 1)$.
We consider three cases.

\vspace{1ex}
\indent{(a)}~For $\nu = 2$, since $e^{-t} + e^{t} \geq 2$, we have $\frac{1-e^{-t}}{t}+  \frac{e^{t}-1}{t} \geq 2$ which implies $C_{\max}(t) = \overline{\kappa}_{\nu}(t) - 1 = \frac{e^{t}-1-t}{t}$.
Hence, by Lemma \ref{le:elem_inequalities}, we have $C_{\max}(t) \leq \left(\frac{3}{2} + \frac{t}{3}\right)te^t$ which leads to $R_{\nu}(t) := \left(\frac{3}{2} + \frac{t}{3}\right)e^t$.

\vspace{1ex}
\indent{(b)}~For $\nu \in (2, 3]$, we have 
\begin{equation*}
\begin{array}{ll}
C_{\max}(t) &= \max\set{1 \!-\! \frac{(\nu \!-\! 2)}{\nu t}\left[1  \!-\! (1  \!-\!  t)^{\frac{\nu}{\nu \!-\! 2}}\right], \frac{(\nu \!-\! 2)}{(4 \!-\! \nu) t}\Big[\frac{1}{(1  \!-\!  t)^{\frac{4 \!-\! \nu}{\nu-2}}} \!-\! 1\Big] - 1} = \frac{(\nu - 2)}{(4 - \nu) t}\Big[\frac{1}{(1  -  t)^{\frac{4-\nu}{\nu-2}}} - 1\Big] - 1.
\end{array}
\end{equation*}
Indeed, we show that $\frac{(\nu - 2)}{(4 -\nu) t}\Big[\frac{1}{(1  -  t)^{\frac{4-\nu}{\nu-2}}} - 1\Big] + \frac{(\nu - 2)}{\nu t}\left[1 - (1 - t)^{\frac{\nu}{\nu -  2}}\right] \geq 2$.
Let $u := \frac{4-\nu}{\nu-2} > 0$ and $v := \frac{\nu}{\nu-2} > 0$. The last inequality is equivalent to $\frac{1}{u}\left[\frac{1}{(1 - t)^u}-1\right] + \frac{1}{v}\left[1 - (1-t)^v\right] \geq 2t$, 
which can be reformulated as  $\frac{1}{v} - \frac{1}{u} + \frac{1}{u(1-t)^u} - \frac{(1-t)^v}{v} - 2t \geq 0$. 
Consider $s(t) := \frac{1}{v} - \frac{1}{u} + \frac{1}{u(1-t)^u} - \frac{(1-t)^v}{v} - 2t$. 
It is clear that $s'(t) = \frac{1}{(1-t)^{u+1}} + (1-t)^{v-1} - 2 = (1-t)^{-\frac{2}{\nu-2}} + (1-t)^{\frac{2}{\nu-2}} - 2 \geq 0$ for all $t\in [0, 1)$.
We obtain $s(t) \geq s(0) = 0$. 
Hence, $C_{\max}(t) = \frac{(\nu - 2)}{(4 - \nu) t}\Big[\frac{1}{(1  -  t)^{\frac{4-\nu}{\nu-2}}} - 1\Big] - 1$.

Let us define $r := \frac{4-\nu}{\nu -2} = \frac{2}{\nu-2} - 1$. Then, it is clear that $\nu = 2 + \frac{2}{1+r}$, and $\nu \in (2, 3]$ is equivalent to $r \geq 1$. 
Now, using Lemma \ref{le:elem_inequalities} with $r = \frac{2}{\nu-2} - 1 \geq 1$, we obtain $R_{\nu}(t) := \frac{1 - (1-t)^{\frac{4-\nu}{\nu-2}} - \left(\frac{4-\nu}{\nu-2}\right)t(1-t)^{\frac{4-\nu}{\nu-2}}}{\left(\frac{4-\nu}{\nu-2}\right)t^2(1-t)^{\frac{4-\nu}{\nu-2}}}$.
Put (a) and (b) together, we obtain \eqref{eq:H_matrix_norm} with $R_{\nu}$ defined by \eqref{eq:R_alpha}. 
The boundedness of $R_{\nu}$ follows from Lemma \ref{le:elem_inequalities}.
\Eproof
\end{proof}

\beforesubsec
\subsection{\bf The proof of Theorem~\ref{th:existence_and_unique}: Solution existence and uniqueness}\label{apdx:th:existence_and_unique}
\aftersubsec
Consider a sublevel set $\mathcal{L}_F(\xb):=\set{\yb\in\dom{F} \mid F(\yb)\leq F(\xb)}$ of $F$ in \eqref{eq:composite_cvx}. 
For any $\yb\in\mathcal{L}_F(\xb)$ and $\vb\in\partial g(\xb)$, by \eqref{eq:f_bound1} and the convexity of $g$, we have
\begin{equation*}
F(\xb)\geq F(\yb)\geq F(\xb)+\iprod{\nabla f(\xb)+\vb,\yb - \xb}+\omega_{\nu}\left(-d_{\nu}(\xb, \yb)\right)\norm{\yb-\xb}_{\xb}^2.
\end{equation*}
By the Cauchy-Schwarz inequality, we have
\begin{equation}\label{eq_bddd}
\omega_{\nu}\left(-d_{\nu}(\xb, \yb)\right)\norm{\yb-\xb}_{\xb}\leq \norm{\nabla f(\xb)+\vb}_{\xb}^{*}.
\end{equation}
Now, using the assumption $\nabla^2{f}(\xb)\succ 0$ for some $\xb\in\dom{F}$, we have $\sigma_{\min}(x) := \lambda_{\min}(\nabla^2{f}(\xb)) > 0$, the smallest eigenvalue of $\nabla^2{f}(x)$.
\begin{enumerate}
\item[$\mathrm{(a)}$] If $\nu = 2$, then $d_2(\xb,\yb)=M_f\norm{\yb-\xb}_2\leq \frac{M_f}{\sqrt{\sigma_{\min}(x)}}\norm{\yb-\xb}_{\xb}$. This estimate together with \eqref{eq_bddd} imply
\begin{equation}\label{eq:exist_nu2}
\omega_2\left(-d_2(\xb, \yb)\right)d_2(\xb,\yb)\leq \frac{M_f}{\sqrt{\sigma_{\min}(x)}}\norm{\nabla f(\xb)+\vb}_{\xb}^{*} = \frac{M_f}{\sqrt{\sigma_{\min}(x)}}\lambda(x).
\end{equation}
We consider the function $s_2(t) := \omega_2(-t)t = 1 - \frac{1-e^{-t}}{t}$.
Clearly, $s_2'(t) = \frac{e^t - t - 1}{t^2e^t} > 0$ for all $t \in \R_{+}$.
Hence, $s_2(t)$ is increasing on $\R_{+}$. 
However, $s_2(t) < 1$ and $\lim\limits_{t\to+\infty}s_2(t) = 1$.
Therefore, if $\frac{M_f}{\sqrt{\sigma_{\min}(x)}}\lambda(x) < 1$, then the equation $s_2(t) - \frac{M_f}{\sqrt{\sigma_{\min}(x)}}\lambda(x) = 0$ has a unique solution $t^{\ast} \in (0, +\infty)$.
In this case, for $0 \leq d_2(\xb, \yb) \leq t^{\ast}$, \eqref{eq:exist_nu2} holds. 
This condition leads to $M_f\norm{y-x}_2 \leq t^{\ast} <+\infty$ which implies that the sublevel set $\mathcal{L}_F(\xb)$ is bounded.
Consequently, solution $x^{\star}$ of \eqref{eq:composite_cvx} exists.

\vspace{1ex}
\item[$\mathrm{(b)}$] If $2 < \nu < 3$, then
\begin{equation*}
d_{\nu}(\xb,\yb)\leq \left(\frac{\nu}{2}-1\right)\frac{M_f}{\sigma_{\min}(x)^{\frac{3-\nu}{2}}}\norm{\yb-\xb}_{\xb}.
\end{equation*}
This inequality together with \eqref{eq_bddd} imply
\begin{equation*}
\omega_{\nu}\left(-d_{\nu}(\xb, \yb)\right)d_{\nu}(\xb,\yb)\leq \left(\frac{\nu}{2}-1\right)\frac{M_f}{\sigma_{\min}(x)^{\frac{3-\nu}{2}}}\norm{\nabla f(\xb)+\vb}_{\xb}^{*} = \left(\frac{\nu}{2}-1\right)\frac{M_f}{\sigma_{\min}(x)^{\frac{3-\nu}{2}}}\lambda(x).
\end{equation*}
We consider $s_{\nu}(t) := \omega_{\nu}(-t)t$.
After a few elementary calculations, we can easily check that $s_{\nu}$ is increasing on $\R_{+}$ and $s_{\nu}(t) < \frac{\nu-2}{4-\nu}$ for all $t > 0$, and $\lim\limits_{t\to+\infty}s_{\nu}(t) = \frac{\nu-2}{4-\nu}$.
Hence, if $\left(\frac{\nu}{2}-1\right)\frac{M_f}{\sigma_{\min}(x)^{\frac{3-\nu}{2}}}\lambda(x) < \frac{\nu-2}{4-\nu}$, then, similar to Case (a), we can show that solution $x^{\star}$ of \eqref{eq:composite_cvx} exists.
This condition implies that $\lambda(x) < \frac{2\sigma_{\min}(x)^{\frac{3-\nu}{2}}}{(4-\nu)M_f}$.

\vspace{1ex}
\item[$\mathrm{(c)}$] If $\nu = 3$, then $d_3(\xb,\yb) = \frac{M_f}{2}\norm{\yb-\xb}_{\xb}$. Combining this estimate and \eqref{eq_bddd} we get
\begin{equation*}
\omega_3\left(-d_3(\xb, \yb)\right)d_3(\xb,\yb)\leq \frac{M_f}{2}\norm{\nabla f(\xb)+\vb}_{\xb}^{*}.
\end{equation*}
With the same proof as in \cite[Theorem 4.1.11]{Nesterov2004}, if $\frac{M_f}{2}\norm{\nabla f(\xb)+\vb}_{\xb}^{*} < 1$ which is equivalent to $\lambda(x) < \frac{2}{M_f}$, then solution $x^{\star}$ of \eqref{eq:composite_cvx} exists.
\end{enumerate}
Note that the condition on $\lambda(x)$ in three cases (a), (b), and (c) can be unified.
The uniqueness of the solution $x^{\star}$ in these three cases follows from the strict convexity of $F$.
\Eproof

\beforesubsec
\subsection{\bf The proof of Theorem~\ref{th:damped_step_NT}: Convergence of the damped-step Newton method}\label{apdx:th:damped_step_NT}
\aftersubsec
The proof of this theorem is divided into two parts: computing the step-size, and proving the local quadratic convergence.

\beforepar
\paragraph{\textbf{Computing the step-size $\tau_k$:}}
From Proposition \ref{pro:fx_bound1}, for any $\xb^k,\xb^{k+1}\in\dom{f}$, if $d_{\nu}(\xb^k,\xb^{k+1}) < 1$, then we have 
\begin{equation*} 
f(\xb^{k+1}) \leq f(\xb^k) + \iprods{\nabla{f}(\xb^k), \xb^{k+1}-\xb^k} + \omega_{\nu}\left(d_{\nu}(\xb^k, \xb^{k+1})\right)\norm{\xb^{k+1} - \xb^k}_{\xb^k}^2.
\end{equation*}
Now, using \eqref{eq:NT_scheme}, we have $\iprods{\nabla{f}(\xb^k), \xb^{k+1}-\xb^k} = -\tau_k\left(\Vert\nabla{f}(\xb^k)\Vert_{\xb^k}^{\ast}\right)^2 = -\tau_k\lambda_k^2$.
On the other hand, we have 
\begin{equation*} 
\begin{array}{ll}
&\Vert\xb^{k+1} - \xb^k\Vert_{\xb^k}^2 \overset{\tiny\eqref{eq:NT_scheme}}{=}  \tau_k^2\iprods{\nabla^2f(\xb^k)^{-1}\nabla{f}(\xb^k),\nabla{f}(\xb^k)} \overset{\tiny\eqref{eq:NT_decrement}}{=} \tau_k^2\lambda_k^2, \vspace{1ex}\\
&\Vert\xb^{k+1} - \xb^k\Vert_2^2 \overset{\tiny\eqref{eq:NT_scheme}}{=}  \tau_k^2\iprods{\nabla^2f(\xb^k)^{-1}\nabla{f}(\xb^k), \nabla^2f(\xb^k)^{-1}\nabla{f}(\xb^k)}  \overset{\tiny\eqref{eq:NT_decrement}}{=}  \frac{\tau_k^2\beta_k^2}{M_f^2}.\\
\end{array}
\end{equation*}
Using the definition of $d_{\nu}(\cdot)$ in \eqref{eq:dxy_def}, the two last equalities,  and \eqref{eq:d_k}, we can easily show that $d_{\nu}(\xb^k, \xb^{k+1}) = \tau_kd_k$.
Substituting these relations into the first estimate, we obtain 
\begin{equation*} 
f(\xb^{k+1}) \leq f(\xb^k)  - \left(\tau_k\lambda_k^2 - \omega_{\nu}\left( \tau_kd_k\right)\tau_k^2\lambda_k^2\right).
\end{equation*}
We consider the following cases:

\indent{(a)}~If $\nu = 2$, then, by \eqref{eq:omega_def}, we have $\eta_k(\tau) := \lambda_k^2\tau - \left(\frac{\lambda_k}{d_k}\right)^2\left(e^{\tau d_k} - \tau d_k - 1\right)$ with $d_k = \beta_k$. This function attains the maximum at $\tau_k  := \frac{\ln(1 + d_k)}{d_k}  = \frac{\ln(1 + \beta_k)}{\beta_k} \in (0, 1)$ with 
\begin{equation*}
\eta_k(\tau_k) = \left(\frac{\lambda_k}{d_k}\right)^2\Big[ (1 \!+\! d_k)\ln(1 \!+\! d_k) \!-\! d_k\Big] = \left(\frac{\lambda_k}{\beta_k}\right)^2\Big[ (1 \!+\! \beta_k)\ln(1 \!+\! \beta_k) \!-\! \beta_k\Big].
\end{equation*}
It is easy to check from the right-most term of the last expression that $\Delta_k := \eta_k(\tau_k) > 0$ for $\tau_k  > 0$.

\indent{(b)}~If $\nu = 3$, then, by \eqref{eq:omega_def}, we have $\eta_k(\tau) := \lambda_k^2\tau  +  \left(\frac{\lambda_k}{d_k}\right)^2\left[\tau d_k + \ln(1 - \tau d_k)\right]$ with $d_k = 0.5M_f\lambda_k$. 
We can show that $\eta_k(\tau)$ achieves the maximum at $\tau_k = \frac{1}{1 + d_k} = \frac{1}{1 + 0.5M_f\lambda_k}\in (0,1)$ with
\begin{equation*}
\eta_k(\tau_k) = \frac{\lambda_k^2}{1 + 0.5M_f\lambda_k}+\left(\frac{2}{M_f}\right)^2\left[\frac{0.5M_f\lambda_k}{1 + 0.5M_f\lambda_k}+\ln\left(1-\frac{0.5M_f\lambda_k}{1 + 0.5M_f\lambda_k}\right)\right].
\end{equation*}
We can also easily check that the last term $\Delta_k := \eta_k(\tau_k)$ of this expression is positive for $\lambda_k > 0$.

\indent{(c)}~If $2 < \nu < 3$, then we have $d_k=M_f^{\nu-2}\left(\frac{\nu}{2} - 1\right) \lambda_k^{\nu-2}\beta_k^{3-\nu}$. 
By \eqref{eq:omega_def}, we have
\begin{eqnarray*}
\eta_k(\tau) &=& \left(\lambda_k^2+\frac{\lambda_k^2}{d_k}\frac{\nu-2}{4-\nu}\right)\tau-\left(\frac{\lambda_k}{d_k}\right)^2\frac{(\nu-2)^2}{2(4-\nu)(3-\nu)}\left((1 - \tau d_k)^{\frac{2(3-\nu)}{2-\nu}} - 1\right).
\end{eqnarray*}
Our aim is to find $\tau^{\ast} \in (0, 1]$ by solving $\max_{\tau \in [0, 1]}\eta_k(\tau)$. This problem always has a global solution.
First, we compute the first- and the second-order derivatives of $\eta_k$ as follows:
\begin{equation*}
\eta_k'(\tau) = \lambda_k^2\left[1 - \frac{1}{d_k}\frac{\nu-2}{\nu - 4}\left(1-(1-\tau d_k)^{\frac{\nu-4}{\nu-2}}\right)\right]\textrm{ and }\eta_k''(\tau)=-\lambda_k^2(1-\tau d_k)^{\frac{-2}{\nu-2}}.
\end{equation*}
Let us set $\eta_k'(\tau_k) = 0$. Then, we get
\begin{equation*} 
\tau_k = \frac{1}{d_k}\left[1-\left(1+\frac{4-\nu}{\nu-2}d_k\right)^{-\frac{\nu-2}{4-\nu}}\right] \in (0,1)~~~~\textrm{(by the Bernoulli inequality)},
\end{equation*}
with
\begin{equation*}
\eta_k(\tau_k)=\frac{\lambda_k^2}{d_k}\left[1-\frac{4-\nu}{2(3-\nu)}\left(1+\frac{4-\nu}{\nu-2}d_k\right)^{2-\nu}\right]+\left(\frac{\lambda_k}{d_k}\right)^2 \frac{\nu-2}{2(3-\nu)}\left[1-\left(1+\frac{4-\nu}{\nu-2}d_k\right)^{2-\nu}\right].
\end{equation*}
In addition, we can check that $\eta_k''(\tau_k) < 0$. Hence, the value of $\tau_k$ above achieves the maximum of $\eta_k(\cdot)$. 
Then, we have $\Delta_k := \eta_k(\tau_k) > \eta_k(0)=0$.

\beforepar
\paragraph{\textbf{The proof of local quadratic convergence:}}
Let $\xopt_f$ be the optimal solution of \eqref{eq:gsc_min}. 
We have 
\begin{equation*}
\begin{array}{ll}
\Vert\xb^{k+1} - \xopt_f\Vert_{\xb^k} &= \Vert \xb^k - \tau_k\nabla^2{f}(\xb^k)^{-1}\nabla{f}(\xb^k) - \xopt_f\Vert_{\xb^k} \vspace{1ex}\\
&= (1-\tau_k)\Vert\xb^k - \xopt_f\Vert_{\xb^k} + \tau_k\Vert \xb^k - \xopt_f - \nabla^2{f}(\xb^k)^{-1}\nabla{f}(\xb^k)\Vert_{\xb^k}.
\end{array}
\end{equation*}
Hence, we can write
\begin{equation}\label{eq:damped_step_local1}
\Vert\xb^{k\!+\!1} {\!\!}- \xopt_f\Vert_{\xb^k} \!=\! (1 \!-\! \tau_k)\Vert\xb^k - \xopt_f\Vert_{\xb^k} \!+\! \tau_k\Vert \nabla^2{f}(\xb^k)^{-1}{\!\!}\left[ \nabla{f}(\xopt_f) - \nabla{f}(\xb^k) \!-\! \nabla^2{f}(\xb^k)(\xopt_f \!-\! \xb^k)\right] \Vert_{\xb^k}.
\end{equation}
Let us define $T_k := \Big\Vert \nabla^2{f}(\xb^k)^{-1}{\!\!}\left[ \nabla{f}(\xopt_f) - \nabla{f}(\xb^k) \!-\! \nabla^2{f}(\xb^k)(\xopt_f \!-\! \xb^k)\right]\Big\Vert_{\xb^k}$ and consider three cases as follows:

$\mathrm{(a)}$~ For $\nu = 2$, using Corollary~\ref{co:hessian_bound2}, we have $\left(\frac{1-e^{-\bar{\beta}_k}}{\bar{\beta}_k}\right)\nabla^2{f}(\xb^k) \preceq \int_0^1\nabla^2{f}(\xb^k + t(\xopt_f -\xb^k))dt \preceq \left(\frac{e^{\bar{\beta}_k}-1}{\bar{\beta}_k}\right)\nabla^2{f}(\xb^k)$, where $\bar{\beta}_k := M_f\Vert\xb^k - \xopt_f\Vert_2$. 
Using the above inequality, we can show that
\begin{equation*}
T_k \leq \max\set{  1 - \frac{1-e^{- \bar{\beta}_k}}{\bar{\beta}_k}, \frac{e^{\bar{\beta}_k}-1}{\bar{\beta}_k}-1}\Vert\xb^k - \xopt_f\Vert_{\xb^k} = \left(\frac{e^{\bar{\beta}_k} - 1 - \bar{\beta}_k}{\bar{\beta}_k^2}\right)\bar{\beta}_k\Vert\xb^k - \xopt_f\Vert_{\xb^k}.
\end{equation*}
Let $\underline{\sigma}_k := \lambda_{\min}(\nabla^2{f}(\xb^k))$. 
We first derive 
\begin{equation*}
\begin{array}{ll}
\Vert\nabla^2{f}(\xb^k)^{-1}\nabla{f}(\xb^k)\Vert_2 &= \Vert\nabla^2{f}(\xb^k)^{-1}(\nabla{f}(\xb^k) - \nabla{f}(\xopt_f))\Vert_2 \vspace{1ex}\\
&= \Vert \int_0^1\nabla^2{f}(\xb^k)^{-1}\nabla^2{f}(\xb^k + t(\xopt_f - \xb^k))(\xb^k - \xopt_f) dt\Vert_2  \vspace{1ex}\\
&= \Vert \nabla^2{f}(\xb^k)^{-1/2}\Kb(\xb^k,\xopt_f)\nabla^2{f}(\xb^k)^{1/2}(\xb^k - \xopt_f)\Vert_2   \vspace{1ex}\\
&\leq \frac{1}{\sqrt{\underline{\sigma}_k}}\Vert  \Kb(\xb^k,\xopt_f)\Vert \Vert\xb^k - \xopt_f\Vert_{\xb^k}.
\end{array}
\end{equation*}
where $\Kb(\xb^k,\xopt_f) :=\int_0^1 \nabla^2{f}(\xb^k)^{-1/2}\nabla^2{f}(\xb^k + t(\xopt_f - \xb^k) \nabla^2{f}(\xb^k)^{-1/2}dt$.
Using Corollary~\ref{co:hessian_bound2} and noting that $\bar{\beta}_k := M_f\Vert\xb^k - \xopt_f\Vert_2$, we can estimate  $\Vert  \Kb(\xb^k,\xopt_f)\Vert \leq \frac{e^{\bar{\beta}_k} - 1}{\bar{\beta}_k}$.
Using the two last estimates, and the definition of $\beta_k$, we can derive
\begin{equation*}
\begin{array}{ll}
\beta_k &= M_f\Vert\nabla^2{f}(\xb^k)^{-1}\nabla{f}(\xb^k)\Vert_2 \leq \frac{M_fe^{\bar{\beta}_k - 1}}{\bar{\beta}_k\sqrt{\underline{\sigma}_k}}\Vert\xb^k - \xopt_f\Vert_{\xb^k} \leq M_fe\frac{\Vert\xb^k - \xopt_f\Vert_{\xb^k}}{\sqrt{\underline{\sigma}_k}},
\end{array}
\end{equation*}
provided that $\bar{\beta}_k \leq 1$.
Since, the step-size $\tau_k = \frac{1}{\beta_k}\ln(1+\beta_k)$, we have $1 - \tau_k \leq \frac{\beta_k}{2} \leq \frac{M_fe\Vert\xb^k - \xopt_f\Vert_{\xb^k}}{2\sqrt{\underline{\sigma}_k}}$.
On the other hand, $\frac{e^{\bar{\beta}_k}-1 - \bar{\beta}_k}{\bar{\beta}_k^2} \leq \frac{e}{2}$ for all $0\leq \bar{\beta}_k \leq 1$.
Substituting $T_k$ into \eqref{eq:damped_step_local1} and using these relations, we have
\begin{equation*}
\Vert\xb^{k\!+\!1} - \xopt_f\Vert_{\xb^k} \leq \tfrac{e}{2}\bar{\beta}_k\Vert\xb^k - \xopt_f\Vert_{\xb^k} + \tfrac{M_fe}{2}\tfrac{\Vert\xb^k - \xopt_f\Vert_{\xb^k}^2}{\sqrt{\underline{\sigma}_k}},
\end{equation*}
provided that $\bar{\beta}_k \leq 1$. 
On the other hand, by Proposition~\ref{pro:hessian_bounds}, we have $\Vert\xb^{k\!+\!1} - \xopt_f\Vert_{\xb^{k+1}} \leq e^{\frac{\bar{\beta}_{k+1} + \bar{\beta}_k}{2}}\Vert\xb^{k\!+\!1} - \xopt_f\Vert_{\xb^k}$ and $\underline{\sigma}_{k+1}^{-1} \leq e^{\bar{\beta}_k + \bar{\beta}_{k+1}}\underline{\sigma}_k^{-1}$. In addition, $\bar{\beta}_k \leq \frac{M_f}{\sqrt{\underline{\sigma}_k}}\Vert\xb^{k} - \xopt_f\Vert_{\xb^k}$
Combining the above inequalities, we finally get
\begin{equation*}
\frac{\Vert\xb^{k\!+\!1} - \xopt_f\Vert_{\xb^{k+1}}}{\sqrt{\underline{\sigma}_{k+1}}} \leq M_fe^{1+\bar{\beta}_{k+1} + \bar{\beta}_k}\left( \frac{\Vert\xb^{k} - \xopt_f\Vert_{\xb^k}}{\sqrt{\underline{\sigma}_k}} \right)^2.
\end{equation*}
Under the fact that $\beta_k\leq 1$, and $\beta_{k+1} \leq 1$, this estimate shows that $\set{ \frac{\Vert\xb^{k} - \xopt_f\Vert_{\xb^k}}{\sqrt{\underline{\sigma}_k}}} $ quadratically converges to zero.
Since $\Vert\xb^k - \xopt_f\Vert_2 \leq \frac{\Vert\xb^{k} - \xopt_f\Vert_{\xb^k}}{\sqrt{\underline{\sigma}_k}}$, we can also conclude that $\set{\Vert\xb^k - \xopt_f\Vert_2}$ quadratically converges to zero.

$\mathrm{(b)}$~ For $\nu = 3$, we can follow \cite{Nesterov2004}. However, for completeness, we give a short proof here.
Using Corollary~\ref{co:hessian_bound2}, we have $\left(1 - r_k + \frac{r_k^2}{3}\right)\nabla^2{f}(\xb^k) \preceq \int_0^1\nabla^2{f}(\xb^k + t(\xopt_f -\xb^k))dt \preceq \frac{1}{1-r_k}\nabla^2{f}(\xb^k)$, where $r_k := 0.5M_f\Vert\xb^k - \xopt_f\Vert_{\xb^k} < 1$. 
Using the above inequality, we can show that
\begin{equation*}
T_k \leq \max\set{  r_k - \frac{r_k^2}{3}, \frac{r_k}{1 - r_k}}\Vert\xb^k - \xopt_f\Vert_{\xb^k} = \frac{0.5M_f\Vert\xb^k - \xopt_f\Vert_{\xb^k}^2}{1 - 0.5M_f\Vert\xb^k - \xopt_f\Vert_{\xb^k}}.
\end{equation*}
Substituting $T_k$ into \eqref{eq:damped_step_local1} and using $\tau_k = \frac{1}{1 + 0.5M_f\lambda_k}$, we have
\begin{equation*}
\Vert\xb^{k\!+\!1} - \xopt_f\Vert_{\xb^k} \leq \frac{0.5M_f\lambda_k}{1+0.5M_f\lambda_k}\norms{\xb^k - \xopt_f}_{\xb^k} + \frac{1}{1 + 0.5M_f\lambda_k}\left(\frac{0.5M_f\norms{\xb^k - \xopt_f}_{\xb^k}^2}{1 - 0.5M_f\norms{\xb^k - \xopt_f}_{\xb^k}}\right).
\end{equation*}
Next, we need to upper bound $\lambda_k$. Since $\nabla{f}(\xopt_f) = 0$. 
Using Corollary~\ref{co:hessian_bound2}, we can bound $\lambda_k$ as
\begin{equation*}
\begin{array}{ll}
\lambda_k &= \Vert\nabla{f}(\xb^k)\Vert_{\xb^k}^{\ast} = \Vert \nabla^2{f}(\xb^k)^{-1/2}(\nabla{f}(\xb^k) - \nabla{f}(\xopt_f))\Vert_2 \vspace{1ex}\\
&= \Vert \int_0^1\nabla^2{f}(\xb^k)^{-1/2}\nabla^2f(\xb^k + t(\xopt_f - \xb^k))(\xopt_f - \xb^k)dt\Vert_2 \vspace{1ex}\\
&\leq \Vert \xb^k - \xopt_f\Vert_{\xb^k}\Vert \int_0^1\nabla^2{f}(\xb^k)^{-1/2}\nabla^2f(\xb^k + t(\xopt_f - \xb^k))\nabla^2{f}(\xb^k)^{-1/2}dt\Vert_2 \vspace{1ex}\\
&\overset{\tiny\text{Corollary~\ref{co:hessian_bound2}}}{\leq} \frac{\Vert \xb^k - \xopt_f\Vert_{\xb^k} }{1 - 0.5M_f \Vert \xb^k - \xopt_f\Vert_{\xb^k} } \leq 2\norms{\xb^k - \xopt_f}_{\xb^k},
\end{array}
\end{equation*}
provided that $M_f\norms{\xb^k - \xopt_f}_{\xb^k} < 1$.
Overestimating the above inequality using this bound, we get
\begin{equation*}
\begin{array}{ll}
\Vert\xb^{k\!+\!1} - \xopt_f\Vert_{\xb^k} & \leq 0.5M_f\lambda_k\norms{\xb^k-\xb_f^{\star}}_{\xb^k} + \frac{0.5M_f\norms{\xb^k-\xb_f^{\star}}_{\xb^k}^2}{1-0.5M_f\norms{\xb^k-\xb_f^{\star}}_{\xb^k}}\vspace{1ex}\\
& \leq M_f\norms{\xb^k-\xb_f^{\star}}_{\xb^k}^2+M_f\norms{\xb^k-\xb_f^{\star}}_{\xb^k}^2=2M_f\norms{\xb^k-\xb_f^{\star}}_{\xb^k}^2,
\end{array}
\end{equation*}
provided that $M_f\norms{\xb^k-\xb_f^{\star}}_{\xb^k} < 1$.
On the other hand, we can also estimate $\Vert\xb^{k\!+\!1} - \xopt_f\Vert_{\xb^{k+1}} \leq \frac{\Vert\xb^{k\!+\!1} - \xopt_f\Vert_{\xb^{k}}}{1 - 0.5M_f\left(\Vert\xb^{k\!+\!1} - \xopt_f\Vert_{\xb^{k}} + \Vert \xb^k - \xopt_f\Vert_{\xb^k}\right)}$.
Combining the last two inequalities, we get
\begin{equation*}
\Vert\xb^{k\!+\!1} - \xopt_f\Vert_{\xb^{k+1}} \leq \frac{2M_f\Vert\xb^k - \xopt_f\Vert_{\xb^k}^2}{ 1 - 2M_f\Vert\xb^k - \xopt_f\Vert_{\xb^k}^2 - 0.5M_f\Vert\xb^k - \xopt_f\Vert_{\xb^k}}
\end{equation*}
The right-hand side function $\psi(t) = \frac{2M_f}{1 - 2M_ft^2 - 0.5M_ft} \leq 4M_f$ on $t \in \left[0, \frac{1}{2M_f} \right]$.
Hence, if $\Vert \xb^k - \xopt_f\Vert_{\xb^k} \leq  \frac{1}{2M_f}$, then $\Vert\xb^{k\!+\!1} - \xopt_f\Vert_{\xb^{k+1}} \leq 4M_f\Vert \xb^k - \xopt_f\Vert_{\xb^k}^2$. 
This shows that if $\xb^0\in\dom{f}$ is chosen such that $\Vert\xb^0 - \xopt_f\Vert_{\xb^0} \leq  \frac{1}{4M_f}$, then $\set{\Vert \xb^k - \xopt_f\Vert_{\xb^k}}$  quadratically converges to zero.

$\mathrm{(c)}$~ For $\nu \in (2, 3)$, with the same argument as in the proof of Theorem~\ref{th:full_step_NT_scheme_converg}, we can show that
\begin{equation*}
\Vert\xb^{k\!+\!1} - \xopt_f\Vert_{\xb^k} \leq R_{\nu}(d^k_{\nu})d_{\nu}^k\Vert\xb^k - \xopt_f\Vert_{\xb^k},
\end{equation*}
where $R_{\nu}$ is defined by \eqref{eq:R_alpha} and $d_{\nu}^k := M_f^{\nu-2}\left(\frac{\nu}{2} - 1\right)\Vert\xb^k-\xopt_f\Vert_2^{3-\nu}\Vert\xb^k - \xopt_f\Vert_{\xb^k}^{\nu-2}$.
Using again the argument as in the proof of Theorem~\ref{th:full_step_NT_scheme_converg}, we have
\begin{equation*}
\frac{\Vert\xb^{k\!+\!1} - \xopt_f\Vert_{\xb^{k+1}}}{\underline{\sigma}_{k+1}^{\frac{3-\nu}{2}}} \leq C_{\nu}(d^k_{\nu},\Vert\xb^k - \xopt_f\Vert_{\xb^k})\left(\frac{\Vert\xb^k - \xopt_f\Vert_{\xb^k}}{ \underline{\sigma}_k^{\frac{3-\nu}{2}} }\right)^2.
\end{equation*}
Here, $C_{\nu}(\cdot,\cdot)$ is a given function deriving from $R_{\nu}$.
Under the condition that $d^k_{\nu}$ and $\Vert\xb^k - \xopt_f\Vert_{\xb^k}$ are sufficiently small, we can show that $C_{\nu}(d^k_{\nu},\Vert\xb^k - \xopt_f\Vert_{\xb^k}) \leq \bar{C}_{\nu}$.
Hence, the last inequality shows that $\Big\{ \frac{\Vert\xb^{k} - \xopt_f\Vert_{\xb^k}}{\underline{\sigma}_k^{\frac{3-\nu}{2}} } \Big\}$ quadratically converges to zero. 
Since $\underline{\sigma}_k^{\frac{3-\nu}{2}}\Vert\xb^k -\xopt_f\Vert_{\Hb_k} \leq \Vert\xb^k - \xopt_f\Vert_{\xb^k}$, where $\Hb_k := \nabla^2{f}(\xb^k)^{\frac{\nu-2}{2}}$, we have $\Vert\xb^k -\xopt_f\Vert_{\Hb_k} \leq \frac{\Vert\xb^{k} - \xopt_f\Vert_{\xb^k}}{\underline{\sigma}_k^{\frac{3-\nu}{2}} }$. 
Hence, we can conclude that $\set{\Vert\xb^k -\xopt_f\Vert_{\Hb_k}}$ also locally converges to zero at a quadratic rate.
\Eproof

\beforesubsec
\subsection{\bf The proof of Theorem~\ref{th:full_step_NT_scheme_converg}: The convergence of the full-step Newton method}\label{apdx:th:full_step_NT_scheme_converg}
\aftersubsec
We divide this proof into two parts: the quadratic convergence of $\Big\{\frac{\lambda_k}{\underline{\sigma}_k^{\frac{3-\nu}{2}}}\Big\}$, and the quadratic convergence of $\big\{\Vert\xb^k - \xopt_f\Vert_{\Hb_k}\big\}$.

\beforepar
\paragraph{\textbf{The quadratic convergence of $\Big\{\frac{\lambda_k}{\underline{\sigma}_k^{\frac{3-\nu}{2}}}\Big\}$}:}
Since the full-step Newton scheme updates $\xb^{k+1} := \xb^k - \nabla^2f(\xb^k)^{-1}\nabla{f}(\xb^k)$, if we denote by $\ntdir^k = \xb^{k+1} -\xb^k = - \nabla^2f(\xb^k)^{-1}\nabla{f}(\xb^k)$, then the last expression leads to $\nabla{f}(\xb^k) + \nabla^2f(\xb^k)\ntdir^k = 0$.
In addition, $\Vert\ntdir^k\Vert_{\xb^k} = \Vert\nabla{f}(\xb^k)\Vert_{\xb^k}^{\ast} = \lambda_k$.
Using the definition of $d_{\nu}(\cdot,\cdot)$ in \eqref{eq:dxy_def}, we denote $d^k_{\nu} := d_{\nu}(\xb^k, \xb^{k+1})$.

First, by $\nabla{f}(\xb^k) + \nabla^2f(\xb^k)\ntdir^k = 0$ and the mean-value theorem, we can show that
\begin{equation*}\label{eq:th32_est1}
 \nabla{f}(\xb^{k+1}) =  \nabla{f}(\xb^{k+1}) - \nabla{f}(\xb^k) - \nabla^2f(\xb^k)\ntdir^k = \int_0^1\left[\nabla^2{f}(\xb^k + t\ntdir^k) - \nabla^2{f}(\xb^k)\right]\ntdir^kdt.
\end{equation*}
Let us define $\Gb_k := \int_0^1\left[\nabla^2{f}(\xb^k + t\ntdir^k) - \nabla^2{f}(\xb^k)\right]dt$ and $\Hb_k := \nabla^2{f}(\xb^k)^{-1/2}\Gb_k\nabla^2{f}(\xb^k)^{-1/2}$.
Then, the above estimate implies $ \nabla{f}(\xb^{k+1}) = \Gb_k\ntdir^k$. 
Hence, we can show that
\begin{align*} 
\left[\Vert\nabla{f}(\xb^{k+1})\Vert_{\xb^k}^{\ast}\right]^2 &= \iprods{\nabla^2{f}(\xb^k)^{-1}\Gb_k\ntdir^k, \Gb_k\ntdir^k}  = \iprods{\Hb_k\nabla^2{f}(\xb^k)^{1/2}\ntdir^k, \Hb_k\nabla^2{f}(\xb^k)^{1/2}\ntdir^k}\nonumber\\
&\leq \Vert \Hb_k\Vert^2\Vert \ntdir^k \Vert_{\xb^k}^2 = \Vert \Hb_k\Vert^2\lambda_k^2.
\end{align*}
By Lemma \ref{le:H_norm}, we can estimate 
\begin{align*} 
 \Vert \Hb_k\Vert &\leq R_{\nu}( d_{\nu}^k )d_{\nu}^k,
\end{align*}
where $R_{\nu}$ is defined by \eqref{eq:R_alpha}.
Combining the two last inequalities and using Proposition \ref{pro:hessian_bounds}, we consider the following cases:

(a)~If $\nu = 2$, then we have $\lambda_{k+1}^2 \leq e^{d_2^k}\left[\norm{\nabla{f}(\xb^{k+1})}_{\xb^k}^{\ast}\right]^2$ which implies $\lambda_{k+1} \leq e^{\frac{d_2^k}{2}}R_2(d_2^k)d_2^k\lambda_k$.
Note that $\lambda_k \geq \frac{\sqrt{\underline{\sigma}_k}d_2^k}{M_f}$ and $\frac{1}{\underline{\sigma}_{k+1}}\leq \frac{e^{d_2^k}}{\underline{\sigma}_k}$. 
Based on the above inequality, we have
\begin{equation*}
\frac{\lambda_{k+1}}{\sqrt{\underline{\sigma}_{k+1}}}\leq M_f R_2(d_2^k)e^{d_2^k}\left(\frac{\lambda_k}{\sqrt{\underline{\sigma}_k}}\right)^2.
\end{equation*}
By a numerical calculation, we can easily check that if $d_2^k < d_2^{\star}\approx \needcheck{0.12964}$, then
\begin{equation*}
\frac{\lambda_{k+1}}{\sqrt{\underline{\sigma}_{k+1}}}\leq 2M_f\left(\frac{\lambda_k}{\sqrt{\underline{\sigma}_k}}\right)^2.
\end{equation*}
Consequently, if $\frac{\lambda_0}{\sqrt{\underline{\sigma}_0}} < \frac{1}{M_f}\min\set{d_2^{\star},0.5} = \frac{d_2^{\star}}{M_f}$, then we can prove
\begin{equation*}
d_2^{k+1} \leq d_2^{k}\textrm{ and }\frac{\lambda_{k+1}}{\sqrt{\underline{\sigma}_{k+1}}} \leq \frac{\lambda_k}{\sqrt{\underline{\sigma}_k}},
\end{equation*}
by induction.  Under the condition $\frac{\lambda_0}{\sqrt{\underline{\sigma}_0}} <  \frac{d_2^{\star}}{M_f}$, the above inequality shows that the ratio $\set{\frac{\lambda_k}{\sqrt{\underline{\sigma}_k}}}$ converges to zero at a quadratic rate.

Now, if $\nu > 2$, then we consider different cases. 
Note that
\begin{equation*}
\lambda_{k+1}^2 \leq (1-d_{\nu}^k)^{\frac{-2}{\nu-2}}\left[\norm{\nabla{f}(\xb^{k+1})}_{\xb^k}^{\ast}\right]^2,
\end{equation*}
which follows that
\begin{equation}\label{noncomp_3plus}
\lambda_{k+1}\leq (1-d_{\nu}^k)^{\frac{-1}{\nu-2}}R_{\nu}(d_{\nu}^k)d_{\nu}^k\lambda_k.
\end{equation}
Note that $d_{\nu}^k=\left(\frac{\nu}{2}-1\right)M_f\norm{\db^k}_2^{3-\nu}\lambda_k^{\nu-2}$ and $\underline{\sigma}_{k+1}^{-1}\leq (1-d_{\nu}^k)^{\frac{-2}{\nu-2}}\underline{\sigma}_k^{-1}$.
Based on these relations and \eqref{noncomp_3plus} we can argue as follows:
 
$\mathrm{(b)}$~If $2 < \nu < 3$, then $\lambda_k \geq \norm{\db^k}_2\sqrt{\underline{\sigma}_k}$ which follows that $d_{\nu}^k\leq \left(\frac{\nu}{2}-1\right)M_f\underline{\sigma}_k^{-\frac{3-\nu}{2}}\lambda_k$. Hence, 
\begin{equation*}
\frac{\lambda_{k+1}}{\underline{\sigma}_{k+1}^{\frac{3-\nu}{2}}}\leq (1-d_{\nu}^k)^{-\frac{4-\nu}{\nu-2}}R_{\nu}(d_{\nu}^k)\left(\frac{\nu}{2}-1\right)M_f\left(\frac{\lambda_k}{\underline{\sigma}_k^{\frac{3-\nu}{2}}}\right)^2.
\end{equation*}
If $d_{\nu}^k < d_{\nu}^{\star}$, where $d_{\nu}^{\star}$ is the unique solution to the equation
\begin{equation*} 
\left(\frac{\nu}{2}-1\right)\frac{R_{\nu}(d_{\nu}^k)}{(1-d_{\nu}^k)^{\frac{4-\nu}{\nu-2}}}= 2,
\end{equation*}
then $\underline{\sigma}_{k+1}^{-\frac{3-\nu}{2}}\lambda_{k+1}\leq 2M_f\left(\underline{\sigma}_k^{-\frac{3-\nu}{2}}\lambda_k \right)^2$.
Note that it is straightforward to check that this equation always admits a positive solution. 
Hence, if we choose $\xb^0\in\dom{f}$ such that $\underline{\sigma}_0^{-\frac{3-\nu}{2}}\lambda_0 < \frac{1}{M_f}\min\set{\frac{2d_{\nu}^{\star}}{\nu-2},\frac{1}{2}}$, then we can prove the following two inequalities together by induction:
\begin{equation*}
d_{\nu}^k \leq d_{\nu}^{k+1}\textrm{ and }\underline{\sigma}_{k+1}^{-\frac{3-\nu}{2}}\lambda_{k+1} \leq \underline{\sigma}_k^{-\frac{3-\nu}{2}}\lambda_k.
\end{equation*}
In addition, the above inequality also shows that $\set{\underline{\sigma}_k^{-\frac{3-\nu}{2}}\lambda_k}$ quadratically converges to zero.

 $\mathrm{(c)}$~If $\nu = 3$, then $d_3^k= \frac{M_f}{2}\lambda_k$, and
\begin{equation*}
\lambda_{k+1}\leq (1-d_3^k)^{-1}R_3(d_3^k)d_3^k\lambda_k=M_f\frac{R_3(d_3^k)}{2(1-d_3^k)}\lambda_k^2.
\end{equation*}
Directly checking the right-hand side of the above estimate, one can show that if $d_3^k < d_3^{\star}=0.5$, then $\lambda_{k+1}\leq 2M_f\lambda_k^2$. 
Hence, if $\lambda_0 < \frac{1}{M_f}\min\set{2d_3^{\star},0.5} = \frac{1}{2M_f}$, then we can prove the following two inequalities together by induction:
\begin{equation*}
d_3^{k+1} \leq d_3^k\textrm{ and }\lambda_{k+1} \leq \lambda_k.
\end{equation*}
Moreover, the first inequality above also shows that $\set{\lambda_k}$ converges to zero at a quadratic rate.

\beforepar
\paragraph{\textbf{The quadratic convergence of $\big\{\Vert\xb^k - \xopt_f\Vert_{\Hb_k}\big\}$:}}
First, using Proposition~\ref{pro:gradient_bound1} with $\xb := \xb^k$ and $\yb = \xopt_f$, and noting that $\nabla{f}(\xopt_f) = 0$, we have 
\begin{equation*} 
\bar{\omega}_{\nu}(-d_{\nu}(\xb^k, \xopt_f))\Vert\xb^k - \xopt_f\Vert_{\xb^k}^2 \leq \iprods{\nabla{f}(\xb^k), \xb^k - \xopt_f} \leq \Vert\nabla{f}(\xb^k)\Vert_{\xb^k}^{\ast}\Vert \xb^k - \xopt_f\Vert_{\xb^k},
\end{equation*}
where the last inequality follows from the Cauchy-Schwarz inequality.
Hence, we obtain
\begin{equation}\label{eq:grad_x}
\bar{\omega}_{\nu}(-d_{\nu}(\xb^k, \xopt_f))\Vert\xb^k - \xopt_f\Vert_{\xb^k} \leq \Vert\nabla{f}(\xb^k)\Vert_{\xb^k}^{\ast} = \lambda_k.
\end{equation}
We consider three cases:

(1)~When $\nu = 2$, we have $\bar{\omega}_{\nu}(\tau) = \frac{e^\tau-1}{\tau}$. 
Hence, $\bar{\omega}_{\nu}(-d_{\nu}(\xb^k, \xopt_f)) = \frac{1 - e^{-d_{\nu}(\xb^k, \xopt_f)}}{d_{\nu}(\xb^k, \xopt_f)} \geq 1 - \frac{d_{\nu}(\xb^k, \xopt_f)}{2} \geq \frac{1}{2}$ whenever $d_{\nu}(\xb^k, \xopt_f) \leq 1$.
Using this inequality in \eqref{eq:grad_x}, we have $\Vert\xb^k - \xopt_f\Vert_{\xb^k} \leq 2\Vert\nabla{f}(\xb^k)\Vert_{\xb^k}^{\ast} = 2\lambda_k$ provided that $d_{\nu}(\xb^k, \xopt_f) \leq 1$.
One the other hand, by the definition of $\underline{\sigma}_k$, we have $\sqrt{\underline{\sigma}_k}\Vert \xb^k - \xopt_f\Vert_2 \leq \Vert\xb^k - \xopt_f\Vert_{\xb^k}$. 
Combining the two last inequalities, we obtain $\Vert \xb^k - \xopt_f\Vert_2 \leq \frac{2\lambda_k}{\sqrt{\underline{\sigma}_k}}$ provided that $d_{\nu}(\xb^k, \xopt_f) \leq 1$.
Since $\set{\frac{\lambda_k}{\sqrt{\underline{\sigma}_k}}}$ locally converges to zero at a quadratic rate, the last relation also shows that $\big\{\Vert \xb^k - \xopt_f\Vert_2\big\}$ also locally converges to zero at a quadratic rate.

(2)~For $\nu = 3$, we have $\bar{\omega}_{\nu}(-d_{\nu}(\xb^k, \xopt_f)) = \frac{1}{1 + d_{\nu}(\xb^k, \xopt_f)}$ and $d_{\nu}(\xb^k, \xopt_f) = \frac{M_f}{2}\Vert \xb^k - \xopt_f\Vert_{\xb^k}$.
Hence, from \eqref{eq:grad_x}, we obtain $\frac{\Vert\xb^k - \xopt_f\Vert_{\xb^k} }{1 + 0.5M_f\Vert\xb^k - \xopt_f\Vert_{\xb^k} } \leq \lambda_k$. This implies $\Vert\xb^k - \xopt_f\Vert_{\xb^k}  \leq \frac{\lambda_k}{1 - 0.5M_f\lambda_k}$ as long as $0.5M_f\lambda_k < 1$.
Clearly, since $\lambda_k$ locally converges to zero at a quadratic rate, $\Vert\xb^k - \xopt_f\Vert_{\xb^k}$ also locally converges to zero at a quadratic rate.

(3)~For $2 < \nu < 3$, we have $\bar{\omega}_{\nu}(-d_{\nu}(\xb^k, \xopt_f)) = \left(\frac{\nu-2}{\nu-4}\right)\frac{\left(1 + d_{\nu}(\xb^k, \xopt_f) \right)^{\frac{\nu-4}{\nu-2}} - 1}{d_{\nu}(\xb^k, \xopt_f)} \geq 1 - \frac{1}{\nu-2}d_{\nu}(\xb^k, \xopt_f) \geq \frac{1}{2}$ provided that $d_{\nu}(\xb^k, \xopt_f) < \frac{\nu}{2}-1$.
Similar to the case $\nu = 2$, we have $\underline{\sigma}_k^{\frac{3-\nu}{2}}\Vert\xb^k -\xopt_f\Vert_{\Hb_k} \leq \Vert\xb^k - \xopt_f\Vert_{\xb^k} \leq 2\lambda_k$, where $\Hb_k := \nabla^2{f}(\xb^k)^{\frac{\nu-2}{2}}$.
Hence, $\Vert\xb^k -\xopt_f\Vert_{\Hb_k} \leq \frac{2\lambda_k}{\underline{\sigma}_k^{\frac{3-\nu}{2}}}$. 
Since $\big\{\frac{\lambda_k}{\underline{\sigma}_k^{\frac{3-\nu}{2}}}\big\}$ locally converges to zero at a quadratic rate, $\big\{\Vert\xb^k -\xopt_f\Vert_{\Hb_k} \big\}$ also locally converges to zero at a quadratic rate.
\Eproof

\beforesubsec
\subsection{\bf The proof of Theorem~\ref{th:comp_decr}: Convergence of the damped-step PN method}\label{apdx:th:comp_decr}
\aftersubsec
Given $\Hb\in\Sc^p_{++}$ and a proper, closed, and convex function $g : \R^p\to\Rext$, we define
\begin{equation*}
\Pa_{\Hb}^g(\ub):=(\Hb+\partial g)^{-1}(\ub) = \argmin_{\xb}\set{g(\xb) + \tfrac{1}{2}\iprod{\Hb\xb,\xb}-\iprod{\ub,\xb}}.
\end{equation*}
If $\Hb = \nabla^2{f}(\xb)$ is the Hessian mapping of a strictly convex function $f$, then we can also write $\Pa_{\nabla^2 f(\xb)}(\ub)$ shortly as $\Pa_{\xb}(\ub)$ for our notational convenience.
The following lemma will be used in the sequel whose proof can be found in~\cite{Tran-Dinh2013a}.

\begin{lemma}\label{le:nonexpansiveness}
Let $g : \R^p\to\Rext$ be a proper, closed, and convex function, and $\Hb\in\Sc^p_{++}$. 
Then, the mapping $\Pa_{\Hb}^g$ defined above is non-expansive with respect to the weighted norm defined by $\Hb$, i.e., for any $\ub,\vb\in\mathbb{R}^p$, we have
\begin{equation}
\norm{\Pa^g_{\Hb}(\ub)-\Pa^g_{\Hb}(\vb)}_{\Hb} \leq \norm{\ub-\vb}^{*}_{\Hb}.
\end{equation}
\end{lemma}
Let us define 
\begin{equation}\label{eq:Sx_e_x}
S_{\xb}(\ub):=\nabla^2 f(\xb)\ub-\nabla f(\ub)~~~\text{and}~~~e_{\xb}(\ub,\vb):=[\nabla^2 f(\xb)-\nabla^2 f(\ub)](\vb-\ub),
\end{equation}
for any vectors $\xb,\ub \in\dom{f}$ and $\vb\in\mathbb{R}^p$. 
We now prove Theorem~\ref{th:comp_decr} in the main text.

\begin{proof}[The proof of Theorem~\ref{th:comp_decr}]
\vspace{-2ex}
\paragraph{\textbf{Computing the step-size $\tau_k$:}}
Since $\zb^k$ satisfies the optimality condition~\eqref{eq:opt_cp_sub}, we have
\begin{equation*}
-\nabla f(\xb^k) - \nabla^2 f(\xb^k)\pntdir^k \in \partial{g}(\zb^k).
\end{equation*}
Using  Proposition \ref{pro:fx_bound1} we obtain
\begin{equation*} 
f(\xb^{k+1}) \leq f(\xb^k) + \tau_k\iprod{\nabla f(\xb^k),\pntdir^k} + \omega_{\nu}(\tau_kd_k)\tau_k^2\lambda_k^2.
\end{equation*}
Since $\xb^{k+1}=(1-\tau_k)\xb^k+\tau_k\zb^k$, using this relation and the convexity of $g$, we have
\begin{equation*} 
g(\xb^{k+1})\leq g(\xb^k)-\tau_k\iprod{\nabla f(\xb^k)+\nabla^2 f(\xb^k)\pntdir^k, \pntdir^k}.
\end{equation*}
Summing up the last two inequalities, we obtain the following estimate
\begin{equation*}
F(\xb^{k+1}) \leq F(\xb^k) - \eta_k(\tau_k).
\end{equation*}
With the same argument as in the proof of Theorem \ref{th:damped_step_NT}, we  obtain the conclusion of Theorem~\ref{th:comp_decr}.

\beforepar
\paragraph{\textbf{The proof of local quadratic convergence:}}
We consider the distance between $\xb^{k+1}$ and $\xb^{\star}$ measured by $\Vert \xb^{k+1}-\xb^{\star}\Vert_{\xb^{\star}}$. 
By the definition of $\xb^{k+1}$, we have
\begin{equation}\label{eq:xb_next_bound2}
\Vert\xb^{k+1} - \xb^{\star}\Vert_{\xb^{\star}}\leq (1-\tau_k)\Vert \xb^k-\xb^{\star}\Vert_{\xb^{\star}}+\tau_k\Vert\zb^k-\xb^{\star}\Vert_{\xb^{\star}}.
\end{equation}
Using the new notations in \eqref{eq:Sx_e_x}, it follows from the optimality condition \eqref{eq:composite_cvx_opt_cond} and \eqref{eq:opt_cp_sub} that $\zb^k = \Pa^g_{\xb^{\star}}(S_{\xb^{\star}}(\xb^k)+e_{\xb^{\star}}(\xb^k,\zb^k))$ and $\xb^{\star}=\Pa^g_{\xb^{\star}}(S_{\xb^{\star}}(\xb^{\star}))$. 
By Lemma \ref{le:nonexpansiveness} and the triangle inequality, we can show that
\begin{equation}\label{eq:bd_diffz_nm}
\Vert\zb^k-\xb^{\star}\Vert_{\xb^{\star}}\leq \Vert S_{\xb^{\star}}(\xb^k) - S_{\xb^{\star}}(\xb^{\star})\Vert^{*}_{\xb^{\star}} + \Vert e_{\xb^{\star}}(\xb^k,\zb^k)\Vert^{*}_{\xb^{\star}}.
\end{equation}
By following the same argument as in \cite{Tran-Dinh2013a}, if we apply Lemma~\ref{le:H_norm}, then we can derive
\begin{equation}\label{eq:bd_first}
\Vert S_{\xb^{\star}}(\xb^k) - S_{\xb^{\star}}(\xb^{\star}) \Vert ^{*}_{\xb^{\star}} \leq R_{\nu}(d_{\nu}(\xb^{\star},\xb^k))d_{\nu}(\xb^{\star},\xb^k)\Vert \xb^k-\xb^{\star}\Vert_{\xb^{\star}},
\end{equation}
where $R_{\nu}(\cdot)$ is defined by \eqref{eq:R_alpha}.

Next, using the same argument as the proof of  \eqref{eq:iny_e} in Theorem~\ref{th:comp_full_pNT_scheme} below, we can bound the second term $\Vert  e_{\xb^{\star}}(\xb^k,\zb^k) \Vert ^{*}_{\xb^{\star}}$ of \eqref{eq:bd_diffz_nm} as
\begin{equation*}
\Vert  e_{\xb^{\star}}(\xb^k,\zb^k) \Vert _{\xb^{\star}}^{*} \leq \begin{cases}
\big[(1-d_{\nu}(\xb^{\star},\xb^k))^{\frac{-2}{\nu-2}}-1 \big] \Vert  \zb^k - \xb^k \Vert _{\xb^{\star}}, ~~~&\text{if $\nu > 2$} \vspace{1.5ex}\\
\big(e^{d_{\nu}(\xb^{\star},\xb^k)} - 1 \big) \Vert \zb^k - \xb^k \Vert _{\xb^{\star}} &\text{if $\nu = 2$}.
\end{cases}
\end{equation*}
Combining this inequality, \eqref{eq:bd_diffz_nm} \eqref{eq:bd_first}, and the triangle inequality
\begin{equation*}
\Vert \zb^k - \xb^k \Vert_{\xb^{\star}} \leq \Vert \zb^k - \xb^{\star} \Vert_{\xb^{\star}} + \Vert \xb^k - \xb^{\star} \Vert_{\xb^{\star}},
\end{equation*}
we obtain
\begin{equation}\label{eq:bound_on_zk}
\Vert \zb^k - \xb^k \Vert_{\xb^{\star}} \leq \hat{R}_{\nu}( d_{\nu}(\xb^{\star},\xb^k)) \Vert  \xb^k-\xb^{\star} \Vert_{\xb^{\star}}
\end{equation}
and
\begin{equation}\label{eq:bound_on_zk_xstar}
\Vert \zb^k - \xb^{\star} \Vert_{\xb^{\star}} \leq \tilde{R}_{\nu}( d_{\nu}(\xb^{\star},\xb^k))d_{\nu}(\xb^{\star},\xb^k) \Vert  \xb^k-\xb^{\star} \Vert_{\xb^{\star}},
\end{equation}
where $\hat{R}_{\nu}$ and $\tilde{R}_{\nu}$ are defined as
\begin{equation*}
\hat{R}_{\nu}( t ) := \begin{cases}
\frac{tR_{\nu}(t)+1}{2-(1-t)^{\frac{-2}{\nu-2}}}, ~~~&\text{if $\nu > 2$} \vspace{1.5ex}\\
\frac{tR_{\nu}(t)+1}{2-e^{t}} &\text{if $\nu = 2$}
\end{cases}\quad\textrm{and}\quad \tilde{R}_{\nu}( t ) := \begin{cases}
\frac{tR_{\nu}(t)+(1-t)^{\frac{-2}{\nu-2}}-1}{t\left(2-(1-t)^{\frac{-2}{\nu-2}}\right)}, ~~~&\text{if $\nu > 2$} \vspace{1.5ex}\\
\frac{tR_{\nu}(t)+1}{t(2-e^{t})} &\text{if $\nu = 2$}
\end{cases} 
\end{equation*}
respectively.

By using Lemma \ref{le:H_norm} and after some simple calculations, one can show that there exists a constant $c_{\nu}\in (0,+\infty)$ such that if $t\in [0,\bar{d}_{\nu}]$, then both $\hat{R}_{\nu}(t)$ and $\tilde{R}_{\nu}(t) \in [0,c_{\nu}]$ (when $t\to 0+$, we consider the limit), where $\bar{d}_2 := 0.6$ and $\bar{d}_{\nu}:= 1-0.6^{\frac{\nu-2}{2}}$ for $\nu > 2$ respectively. Using this bound, \eqref{eq:xb_next_bound2} \eqref{eq:bound_on_zk_xstar} and the fact that $\tau_k \leq 1$, we can bound
\begin{equation}\label{eq:bound_xb_next2}
\Vert \xb^{k+1}-\xb^{\star} \Vert_{\xb^{\star}} \leq \left[ (1 - \tau_k) + c_{\nu} d_{\nu}(\xb^{\star},\xb^k)\right]\Vert \xb^k-\xb^{\star}\Vert_{\xb^{\star}}.
\end{equation}
Let $\underline{\sigma}^{\star} := \sigma_{\min}(\nabla^2 f(\xb^{\star}))$ be the smallest eigenvalue of $\nabla^2 f(\xb^{\star})$.
We consider the following cases:

(a)~If $\nu=2$, for $0 \leq d_{\nu}(\xopt, \xb^k) \leq \bar{d}_{\nu}$, we can bound $1-\tau_k$ as
\begin{equation*}
\begin{array}{ll}
1-\tau_k & =  1-\frac{\ln(1+\beta_k)}{\beta_k} \leq \frac{\beta_k}{2} =  \frac{M_f}{2}\Vert \zb^k - \xb^k\Vert_2 \leq \frac{M_f}{2}\frac{ \Vert \zb^k - \xb^k\Vert_{\xb^{\star}}}{\sqrt{\underline{\sigma}^{\star}}}  \overset{\tiny\eqref{eq:bound_on_zk}}{\leq}  \frac{c_{\nu} M_f}{2\sqrt{\underline{\sigma}^{\star}}} \Vert \xb^k-\xb^{\star}\Vert_{\xb^{\star}}.
\end{array}
\end{equation*}
On the other hand, we have $d_{\nu}(\xb^{\star},\xb^k)=M_f\Vert\xb^k - \xb^{\star} \Vert_2 \leq \frac{M_f}{\sqrt{\underline{\sigma}^{\star}}}\Vert \xb^k-\xb^{\star}\Vert_{\xb^{\star}}$. 
Using these estimates into \eqref{eq:bound_xb_next2}, we get
\begin{equation*}
\Vert \xb^{k+1} - \xb^{\star} \Vert_{\xb^{\star}} \leq \left( \frac{c_{\nu}M_f}{2\sqrt{\underline{\sigma}^{\star}}}\norm{\xb^k-\xb^{\star}}_{\xb^{\star}} + \frac{c_{\nu}M_f}{\sqrt{\underline{\sigma}^{\star}}}\Vert \xb^k-\xb^{\star}\Vert_{\xb^{\star}}\right)\Vert\xb^k-\xopt\Vert_{\xopt} =  \frac{3c_{\nu}M_f}{2\sqrt{\underline{\sigma}^{\star}}}\Vert\xb^k-\xopt\Vert_{\xopt}^2.
\end{equation*}
Let $c^{\star}_{\nu} := \frac{3c_{\nu}M_f}{2\sqrt{\underline{\sigma}^{\star}}}$. 
The last estimate shows that if $\Vert\xb^0 - \xopt\Vert_{\xopt} \leq \min\set{ \frac{ \bar{d}_{\nu}\sqrt{\underline{\sigma}^{\star}}}{M_f}, \frac{1}{c^{\star}_{\nu}}}$, then $\set{\Vert\xb^k - \xopt\Vert_{\xopt}}$ quadratically converges to zero.

(b)~If $2 < \nu \leq 3$, then we first show that
\begin{equation}\label{23dv}
d_{\nu}(\xb^{\star},\xb^k)=\left(\tfrac{\nu}{2}-1\right)M_f\Vert \xb^k - \xopt\Vert_2^{3-\nu}\Vert \xb^k - \xb^{\star}\Vert_{\xb^{\star}}^{\nu-2} \leq \left(\tfrac{\nu}{2}-1\right)\frac{M_f}{\left(\underline{\sigma}^{\star}\right)^{\frac{3-\nu}{2}}}\Vert \xb^k-\xb^{\star}\Vert_{\xb^{\star}}.
\end{equation}
Hence, if $\Vert\xb^k-\xb^{\star}\Vert_{\xb^{\star}}\leq m_{\nu}\bar{d}_{\nu}$, where $m_{\nu} := \tfrac{2}{\nu-2}\frac{\left(\underline{\sigma}^{\star}\right)^{\frac{3-\nu}{2}}}{M_f}$, then $d_{\nu}(\xb^{\star},\xb^k)\leq \bar{d}_{\nu}$.
Next, using the definition of $d_k$ in \eqref{eq:d_k}, we can bound it as
\begin{equation}\label{eq:2-3_tau2}
\begin{array}{ll}
d_k & =  M_f\left(\frac{\nu}{2}-1\right)\Vert \zb^k-\xb^k \Vert_{\xb^k}^{\nu-2}\Vert\zb^k - \xb^k\Vert_2^{3-\nu} \overset{\eqref{eq:hessian_bound1}}{\leq} M_f\left(\frac{\nu}{2}-1\right)\left[\frac{ \Vert \zb^k - \xb^k\Vert_{\xb^{\star}}}{(1-d_{\nu}(\xb^{\star},\xb^k))^{\frac{1}{\nu-2}}}\right]^{\nu-2}\frac{\Vert \zb^k - \xb^k\Vert_{\xb^{\star}}^{3-\nu}}{(\underline{\sigma}^{\star})^{\frac{3-\nu}{2}}} \vspace{1ex}\\ \nonumber
& \leq  \frac{M_f}{(1-\bar{d}_{\nu})(\underline{\sigma}^{\star})^{\frac{3-\nu}{2}}}\left(\frac{\nu}{2}-1\right)\Vert \zb^k - \xb^k\Vert_{\xb^{\star}}  \overset{\tiny\eqref{eq:bound_on_zk}}{\leq}  \frac{(\nu-2)M_f}{2(1-\bar{d}_{\nu})(\underline{\sigma}^{\star})^{\frac{3-\nu}{2}}}c_{\nu}\Vert \xb^k - \xb^{\star} \Vert_{\xb^{\star}}.
\end{array}
\end{equation}
Using this estimate, we can bound $1-\tau_k$ as follows:
\begin{equation}\label{eq:2-3_tau1}
\begin{array}{ll}
1-\tau_k & = 1-\frac{1}{d_k}+\frac{1}{d_k}\left(1-\frac{\frac{4-\nu}{\nu-2}d_k}{1+\frac{4-\nu}{\nu-2}d_k}\right)^{\frac{\nu-2}{4-\nu}} \overset{\tiny\text{Bernoulli's inequality}}{\leq}  1 - \frac{1}{d_k}+\frac{1}{d_k}\left(1-\frac{\nu-2}{4-\nu}\frac{\frac{4-\nu}{\nu-2}d_k}{1+\frac{4-\nu}{\nu-2}d_k}\right)  \vspace{1ex}\\ \nonumber
& = \frac{\frac{4-\nu}{\nu-2}d_k}{1+\frac{4-\nu}{\nu-2}d_k} \leq \frac{4-\nu}{\nu-2}d_k \leq \frac{(4 -\nu)M_f}{2(1-\bar{d}_{\nu})(\underline{\sigma}^{\star})^{\frac{3-\nu}{2}}}c_{\nu}\Vert \xb^k - \xb^{\star} \Vert_{\xb^{\star}} = n_{\nu}\Vert\xb^k - \xopt\Vert_{\xopt},
\end{array}
\end{equation}
where $n_{\nu} := \frac{(4 -\nu)M_f}{2(1-\bar{d}_{\nu})(\underline{\sigma}^{\star})^{\frac{3-\nu}{2}}}c_{\nu} > 0$.
Substituting this estimate and \eqref{23dv} into \eqref{eq:bound_xb_next2}, we get 
\begin{equation*}
\Vert\xb^{k+1} - \xopt\Vert_{\xopt} \leq \left(n_{\nu} + \frac{c_{\nu}}{m_{\nu}}\right)\Vert\xb^k - \xopt\Vert_{\xopt}^2 = c^{\ast}_{\nu}\Vert\xb^k - \xopt\Vert_{\xopt}^2.
\end{equation*}
Hence, if $\Vert\xb^0 - \xopt\Vert_{\xopt} \leq \min\set{ m_{\nu}\bar{d}_{\nu}, \frac{1}{c^{\star}_{\nu}}}$, then the last estimate shows that the sequence $\set{\Vert\xb^k - \xopt\Vert_{\xopt}}$ quadratically converges to zero.

In summary, there exists a neighborhood $\Nc(\xopt)$ of $\xopt$, such that if $\xb^0\in\Nc(\xopt)\cap\dom{F}$, then the whole sequence $\set{\Vert\xb^k-\xb^{\star}\Vert_{\xb^{\star}}}$ quadratically converges  to zero. 
\Eproof
\end{proof}

\beforesubsec
\subsection{\bf The proof of Theorem~\ref{th:comp_full_pNT_scheme}: Locally quadratic convergence of the PN method}\label{apdx:th:comp_full_pNT_scheme}
\aftersubsec
Since $\zb^k$ is the optimal solution to \eqref{eq:dir_ds} which satisfies \eqref{eq:opt_cp_sub}, we have $\nabla^2 f(\xb^k)\xb^k-\nabla f(\xb^k)\in (\nabla^2 f(\xb^k) + \partial g)(\zb^k)$. 
Using this optimality condition, we get
\begin{equation*}
\begin{array}{lll}
\xb^{k+1} &=\zb^k &= \Pa^g_{\xb^k}(S_{\xb^k}(\xb^k)+e_{\xb^k}(\xb^k,\zb^k))~~~~~\textrm{ and }\vspace{1.5ex}\\
\xb^{k+2} &=\zb^{k+1} &= \Pa^g_{\xb^k}(S_{\xb^k}(\xb^{k+1})+e_{\xb^k}(\xb^{k+1},\zb^{k+1})).
\end{array}
\end{equation*}
Let us define $\tilde{\lambda}_{k+1}:=\Vert\pntdir^{k+1}\Vert_{\xb^k}$. 
Then, by Lemma \eqref{le:nonexpansiveness} and the triangular inequality, we have
\begin{equation}\label{eq:iny_se}
\begin{array}{lll}
\tilde{\lambda}_{k+1} & \leq & \norm{S_{\xb^k}(\xb^{k+1})-S_{\xb^k}(\xb^k)}_{\xb^k}^{*}+\norm{e_{\xb^k}(\xb^{k+1},\zb^{k+1})-e_{\xb^k}(\xb^k,\zb^k)}_{\xb^k}^{*} \vspace{1.25ex}\\
& = &  \norm{S_{\xb^k}(\xb^{k+1})-S_{\xb^k}(\xb^k)}_{\xb^k}^{*}+\norm{e_{\xb^k}(\xb^{k+1},\zb^{k+1})}_{\xb^k}^{*}.
\end{array}
\end{equation}
Let us first bound the term $\norm{S_{\xb^k}(\xb^{k+1})-S_{\xb^k}(\xb^k)}_{\xb^k}^{*}$ as follows:
\begin{equation}\label{eq:iny_bdS}
\norm{S_{\xb^k}(\xb^{k+1})-S_{\xb^k}(\xb^k)}_{\xb^k}^{*}\leq R_{\nu}(d_{\nu}^k)d_{\nu}^k\lambda_k,
\end{equation}
where $R_{\nu}(t)$ is defined as \eqref{eq:R_alpha}.
Indeed, from the mean-value  theorem, we have
\begin{equation*}
\norm{S_{\xb^k}(\xb^{k+1})-S_{\xb^k}(\xb^k)}_{\xb^k}^{*} = \norm{\int_0^1 [\nabla^2 f(\xb^k+t\pntdir^k)-\nabla^2 f(\xb^k)]\pntdir^k\ud t}_{\xb^k} \leq \norm{\Hb(\xb^k,\xb^{k+1})}\lambda_k,
\end{equation*}
where $\Hb$ is defined as \eqref{eq:H_matrix}. 
Combining the above inequality and \eqref{eq:R_alpha} in Lemma \ref{le:H_norm}, we get \eqref{eq:iny_bdS}.

Next we bound the term $\norm{e_{\xb^k}(\xb^{k+1},\zb^{k+1})}_{\xb^k}^{*}$ as follows:
\begin{equation}\label{eq:iny_e}
\norm{e_{\xb^k}(\xb^{k+1}, \zb^{k+1})}_{\xb^k} \leq \begin{cases}
\big[(1-d_{\nu}^k)^{\frac{-2}{\nu-2}}-1 \big]\tilde{\lambda}_{k+1}, ~~~&\text{if $\nu > 2$} \vspace{1ex}\\
(e^{d_{\nu}^k}-1)\tilde{\lambda}_{k+1} &\text{if $\nu = 2$}.
\end{cases}
\end{equation}
Note that
\begin{equation*}
\norm{e_{\xb^k}(\xb^{k+1},\zb^{k+1})}_{\xb^k}^{*} = \norm{[\nabla^2 f(\xb^k)-\nabla^2 f(\xb^{k+1})](\zb^{k+1}-\xb^{k+1})}_{\xb^k}^{*} \leq \Vert\widetilde{\Hb}(\xb^k,\xb^{k+1})\Vert\tilde{\lambda}_{k+1},
\end{equation*}
where
\begin{eqnarray*}
\widetilde{\Hb}(\xb,\yb) & := & \nabla^2 f(\xb)^{-1/2}\left(\nabla^2 f(\xb)-\nabla^2 f(\yb)\right) \nabla^2 f(\xb)^{-1/2} \vspace{1ex}\\
& = & \Id  - \nabla^2{f}(\xb)^{-1/2}\nabla^2{f}(\yb) \nabla^2{f}(\xb)^{-1/2}.
\end{eqnarray*}
By Proposition \ref{pro:hessian_bounds}, we have
\begin{equation*} 
\Vert\widetilde{\Hb}(\xb,\yb)\Vert \leq \begin{cases}
~~~\max\set{1-(1-d_{\nu}(\xb,\yb))^{\frac{2}{\nu-2}},(1-d_{\nu}(\xb,\yb))^{\frac{-2}{\nu-2}}-1}, ~~~&\text{if $\nu > 2$} \vspace{1ex}\\
~~~\max\set{1-e^{-d_{\nu}(\xb,\yb)},e^{d_{\nu}(\xb,\yb)}-1} &\text{if $\nu = 2$}.
\end{cases}
\end{equation*}
This inequality can be simplified as
\begin{equation}\label{eq:iny_hessdif}
\Vert\widetilde{\Hb}(\xb,\yb)\Vert \leq \begin{cases}
~~~(1-d_{\nu}(\xb,\yb))^{\frac{-2}{\nu-2}}-1, ~~~&\text{if $\nu > 2$} \vspace{1ex}\\
~~~e^{d_{\nu}(\xb,\yb)}-1 &\text{if $\nu = 2$}.
\end{cases}
\end{equation}
Hence, the inequality \eqref{eq:iny_e} holds.

Now, we combine \eqref{eq:iny_se}, \eqref{eq:iny_bdS}, and \eqref{eq:iny_e}, if $\nu = 2$, and assuming that $d_2^k < \ln 2$, then we get
\begin{equation*}
\tilde{\lambda}_{k+1}\leq \frac{R_2(d_2^k)d_2^k}{2-e^{d_2^k}}\lambda_k.
\end{equation*}
By Proposition \ref{pro:hessian_bounds}, we have $\lambda_{k+1}^2\leq e^{d_{\nu}^k}\tilde{\lambda}_{k+1}^2$. 
Combining this estimate and the last inequality, we get
\begin{equation}\label{iny_lamb}
\lambda_{k+1}\leq \frac{R_2(d_2^k)d_2^ke^{\frac{d_2^k}{2}}}{2-e^{d_2^k}}\lambda_k.
\end{equation}
Note that $\lambda_k \geq \frac{\sqrt{\underline{\sigma}_k}d_2^k}{M_f}$ and $\underline{\sigma}_{k+1}^{-1}\leq e^{d_2^k}\underline{\sigma}_k^{-1}$. 
It follows from \eqref{iny_lamb} that
\begin{equation*}
\frac{\lambda_{k+1}}{\sqrt{\underline{\sigma}_{k+1}}}\leq M_f\frac{R_2(d_2^k)e^{d_2^k}}{2-e^{d_2^k}}\left(\frac{\lambda_k}{\sqrt{\underline{\sigma}_k}}\right)^2.
\end{equation*}
By a numerical calculation, we can check that  if $d_2^k \leq d_2^{\star}\approx \needcheck{0.35482}$, then
\begin{equation*}
\frac{\lambda_{k+1}}{\sqrt{\underline{\sigma}_{k+1}}}\leq 2M_f\left(\frac{\lambda_k}{\sqrt{\underline{\sigma}_k}}\right)^2.
\end{equation*}
Hence, if we choose $\xb^0\in\dom{F}$ such that $\frac{\lambda_0}{\sqrt{\underline{\sigma}_0}} \leq  \frac{1}{M_f}\min\set{d_2^{\star},0.5} = \frac{d_2^{\star}}{M_f}$, then we can prove the following two inequalities together by induction:
\begin{equation*}
d_2^{k+1} \leq d_2^{k}~~~~\textrm{and}~~~~\frac{\lambda_{k+1}}{\sqrt{\underline{\sigma}_{k+1}}} \leq \frac{\lambda_k}{\sqrt{\underline{\sigma}_k}}.
\end{equation*}
These inequalities show the nonincreasing monotonicity of $\set{d_2^k}$ and $\set{\lambda_k}$. 
The above inequality also shows the local quadratic convergence of the sequence $\set{\frac{\lambda_k}{\sqrt{\underline{\sigma}_k}}}$.

Now, if $\nu > 2$ and assume that $d_{\nu}^k < 1- \left({\frac{1}{2}}\right)^{\frac{\nu-2}{2}}$, then
\begin{equation*}
\tilde{\lambda}_{k+1}\leq \frac{R_{\nu}(d_{\nu}^k)d_{\nu}^k}{2-(1-d_{\nu}^k)^{\frac{-2}{\nu-2}}}\lambda_k.
\end{equation*}
By Proposition \ref{pro:hessian_bounds}, we have $\lambda_{k+1}^2 \leq (1-d_{\nu}^k)^{\frac{-2}{\nu-2}}\tilde{\lambda}_{k+1}^2$. Hence, combining these  inequalities, we get
\begin{equation}\label{eq:iny_lamcom}
\lambda_{k+1}\leq \frac{R_{\nu}(d_{\nu}^k)d_{\nu}^k(1-d_{\nu}^k)^{\frac{-1}{\nu-2}}}{2-(1-d_{\nu}^k)^{\frac{-2}{\nu-2}}}\lambda_k.
\end{equation}
Note that $d_{\nu}^k=\left(\frac{\nu}{2}-1\right)M_f\norm{\pb^k}_2^{3-\nu}\lambda_k^{\nu-2}$, $\underline{\sigma}_{k+1}^{-1}\leq (1-d_{\nu}^k)^{\frac{-2}{\nu-2}}\underline{\sigma}_k^{-1}$ and $\sigma_{k+1}^{-1}\leq (1-d_{\nu}^k)^{\frac{-2}{\nu-2}}\sigma_k^{-1}$. 
Using these relations and \eqref{eq:iny_lamcom}, we consider two cases:

\indent{(a)}~If $\nu = 3$, then $d_3^k = \frac{M_f}{2}\lambda_k$, and
\begin{equation*}
\lambda_{k+1}\leq \frac{R_3(d_3^k)(1-d_3^k)^{-1}}{2-(1-d_3^k)^{-2}}d_3^k\lambda_k=M_f\frac{R_3(d_3^k)(1-d_3^k)^{-1}}{2\left(2-(1-d_3^k)^{-2}\right)}\lambda_k^2.
\end{equation*}
By a simple numerical calculation, we can show that if $d_3^k \leq d_3^{\star}\approx 0.20943$, then $\lambda_{k+1}\leq 2M_f\lambda_k^2$. 
Hence, if $\lambda_0 < \frac{1}{M_f}\min\set{2d_3^{\star},0.5} = \frac{2}{M_f}d_3^{\star}$, then we can prove the following two inequalities together by induction
\begin{equation*}
d_3^{k+1} \leq d_3^k\textrm{ and }\lambda_{k+1} \leq \lambda_k.
\end{equation*} 
These inequalities show the non-increasing monotonicity of $\set{d_2^k}$ and $\set{\lambda_k}$. 
The above inequality also shows the quadratic convergence of the sequence $\set{\lambda_k}$.

\indent{(b)}~If $2 < \nu < 3$, then $\lambda_k \geq \Vert\pb^k\Vert_2\sqrt{\underline{\sigma}_k}$ which implies that $d_{\nu}^k\leq \left(\frac{\nu}{2}-1\right)M_f\underline{\sigma}_k^{-\frac{3-\nu}{2}}\lambda_k$. 
Hence, we have
\begin{equation*}
\frac{\lambda_{k+1}}{\underline{\sigma}_{k+1}^{\frac{3-\nu}{2}}}\leq \frac{R_{\nu}(d_{\nu}^k)(1-d_{\nu}^k)^{-\frac{4-\nu}{\nu-2}}}{2-(1-d_{\nu}^k)^{\frac{-2}{\nu-2}}}\left(\frac{\nu}{2}-1\right)M_f\left(\frac{\lambda_k}{\underline{\sigma}_k^{\frac{3-\nu}{2}}}\right)^2.
\end{equation*}
If $d_{\nu}^k < d_{\nu}^{\star}$, then $\underline{\sigma}_{k+1}^{-\frac{3-\nu}{2}}\lambda_{k+1}\leq 2M_f\left(\underline{\sigma}_k^{-\frac{3-\nu}{2}}\lambda_k \right)^2$, where $d_{\nu}^{\star}$ is the unique solution to the equation
\begin{equation*} 
\frac{R_{\nu}(d_{\nu}^k)(1-d_{\nu}^k)^{-\frac{4-\nu}{\nu-2}}}{2-(1-d_{\nu}^k)^{\frac{-2}{\nu-2}}}\left(\frac{\nu}{2}-1\right) = 2.
\end{equation*}
Note that it is straightforward to check that this equation always admits a positive solution. 
Therefore, if $\underline{\sigma}_0^{-\frac{3-\nu}{2}}\lambda_0 \leq  \frac{1}{M_f}\min\set{\frac{2d_{\nu}^{\star}}{\nu-2},\frac{1}{2}}$, then we can prove the following two inequalities together by induction:
\begin{equation*}
d_{\nu}^k \leq d_{\nu}^{k+1}\textrm{ and }\underline{\sigma}_{k+1}^{-\frac{3-\nu}{2}}\lambda_{k+1} \leq \underline{\sigma}_k^{-\frac{3-\nu}{2}}\lambda_k.
\end{equation*}
These inequalities show the non-increasing monotonicity of $\set{d_2^k}$ and $\set{\lambda_k}$. 
The above inequality also shows the quadratic convergence of the sequence $\Big\{\frac{\lambda_k}{\underline{\sigma}_k^{\frac{3-\nu}{2}}}\Big\}$.

Finally, to prove the local quadratic convergence of $\set{\xb^k}$ to $\xopt$, we use the same argument as in the proof of Theorem~\ref{th:full_step_NT_scheme_converg} and Theorem~\ref{th:comp_decr}, where we omit the details here.
\Eproof

\beforesubsec
\subsection{\bf The proof of Theorem~\ref{th:quasi_newton_alg}: Convergence of the quasi-Newton method}\label{apdx:th:quasi_newton_alg}
\aftersubsec
The full-step quasi-Newton method for solving \eqref{eq:gsc_min} can be written as $\xb^{k+1} = \xb^k - \Bb_k\nabla{f}(\xb^k)$. 
This is equivalent to $\Hb_k(\xb^{k+1} - \xb^k) + \nabla{f}(\xb^k) = 0$.
Using this relation and $\nabla{f}(\xopt_f) = 0$, we can write 
\begin{equation}\label{eq:qsnt_est1}
\xb^{k\!+\!1} {\!\!}-{\!} \xopt_f  \!=\!  \nabla^2{f}(\xopt_f)^{-1}\left[\nabla^2{f}(\xopt_f)(\xb^k \!-\! \xopt_f) \!+\!  \left(\nabla^2{f}(\xopt_f) \!-\! \Hb_k\right)(\xb^{k\!+\!1} \!-\! \xb^k) \!-\! \nabla{f}(\xb^k) \!+\! \nabla{f}(\xopt_f)\right].
\end{equation}
We first consider $T_k := \Vert \nabla^2{f}(\xopt_f)^{-1}\left[\nabla{f}(\xb^k) - \nabla{f}(\xopt_f) - \nabla^2{f}(\xopt_f)(\xb^k - \xopt_f) \right]\Vert_{\xopt_f}$.
Similar to the proof of Theorem~\ref{th:full_step_NT_scheme_converg}, we can show that 
\begin{equation}\label{eq:qsnt_est2}
T_k =  \Big\Vert\int_0^1\nabla^2{f}(\xopt_f)^{-1}{\!}\left[\nabla^2{f}(\xopt_f \!+\! t(\xb^k \!-\! \xopt_f)) \!-\! \nabla^2{f}(\xopt_f)\right](\xb^k \!-\! \xopt_f)\Big\Vert_{\xopt_f} \!\leq\! R_{\nu}( d_{\nu}^k )d_{\nu}^k\Vert\xb^k \!-\! \xopt_f\Vert_{\xopt_f}
\end{equation}
where $R_{\nu}$ is defined by \eqref{eq:R_alpha} and $d_{\nu}^k := M_f^{\nu-2}\left(\frac{\nu}{2} - 1\right)\Vert\xb^k-\xopt_f\Vert_2^{3-\nu}\Vert\xb^k - \xopt_f\Vert_{\xopt_f}^{\nu-2}$.
Moreover, we note that 
\begin{equation*}
S_k := \Vert \nabla^2{f}(\xopt_f)^{-1}\left(\Hb_k - \nabla^2{f}(\xopt_f)\right)(\xb^{k\!+\!1} \!-\! \xb^k)\Vert_{\xopt_f} = \Vert \left(\Hb_k - \nabla^2{f}(\xopt)\right)(\xb^{k\!+\!1} \!-\! \xb^k)\Vert^{\ast}_{\xopt_f} 
\end{equation*}
Combining this estimate, \eqref{eq:qsnt_est1}, and \eqref{eq:qsnt_est2}, we can derive
\begin{equation}\label{eq:qsnt_est3}
\Vert\xb^{k\!+\!1} - \xopt_f\Vert_{\xopt_f} \leq R_{\nu}( d_{\nu}^k )d_{\nu}^k\Vert\xb^k-\xopt_f\Vert_{\xopt_f} + \Vert \left(\Hb_k - \nabla^2{f}(\xopt_f)\right)(\xb^{k\!+\!1} \!-\! \xb^k)\Vert^{\ast}_{\xopt_f}.
\end{equation}
First, we prove statement (a). 
Indeed, from the Dennis-Mor\'{e} condition \eqref{eq:DM_cond}, we have 
\begin{equation*}
\Vert \left(\Hb_k - \nabla^2{f}(\xopt_f)\right)(\xb^{k\!+\!1} \!-\! \xb^k)\Vert^{\ast}_{\xopt_f} \leq \gamma_k\Vert\xb^{k+1} -\xb_k\Vert_{\xopt_f} \leq \gamma_k\left(\Vert\xb^{k+1} - \xopt_f\Vert_{\xopt_f} + \Vert \xb^k - \xopt_f\Vert_{\xopt_f}\right),
\end{equation*}
where $\lim_{k\to\infty}\gamma_k = 0$.
Substituting this estimate into \eqref{eq:qsnt_est3}, and noting that $\Vert\xb^k - \xopt_f\Vert_2 \leq \frac{1}{\underline{\sigma}^{\star}}\Vert\xb^k - \xopt_f\Vert_{\xopt_f}$, where $\underline{\sigma}^{\star} := \lambda_{\min}(\nabla^2{f}(\xopt_f)) > 0$, we can show that
\begin{equation}\label{eq:qsnt_est4}
\Vert\xb^{k\!+\!1} - \xopt_f\Vert_{\xopt_f} \leq \frac{1}{1-\gamma_k}\left(R_{\nu}^{\star}\Vert\xb^k - \xopt_f\Vert_{\xopt_f}^2 + \gamma_k\Vert\xb^k - \xopt_f\Vert_{\xopt_f}\right),
\end{equation}
provided that $\Vert\xb^k - \xopt_f\Vert_{\xopt_f} \leq \bar{r}$ and $R_{\nu}^{\star} := \max\set{ R_{\nu}(d_{\nu}^k) \mid \Vert\xb^k - \xopt_f\Vert_{\xopt_f} \leq \bar{r}} < +\infty$.
Here, $\bar{r} > 0$ is a given value such that $R_{\nu}^{\star}$ is finite.
The estimate \eqref{eq:qsnt_est4} shows that if $\bar{r}$ is sufficiently small, $\set{\Vert\xb^k - \xopt_f\Vert_{\xopt_f}}$ superlinearly converges to zero.
Finally, the statement (b) is proved similarly by combining statement (a) and \cite[Theorem 11]{Tran-Dinh2013a}.
\Eproof

\bibliographystyle{plain}

\begin{thebibliography}{10}

\bibitem{Bach2009}
F.~Bach.
\newblock Self-concordant analysis for logistic regression.
\newblock {\em Electron. J. Statist.}, 4:384--414, 2010.

\bibitem{Bach2013a}
F.~Bach.
\newblock Adaptivity of averaged stochastic gradient descent to local strong
  convexity for logistic regression.
\newblock {\em J. Mach. Learn. Res.}, 15(1):595--627, 2014.

\bibitem{Bauschke2011}
H.~H. Bauschke and P.~Combettes.
\newblock {\em Convex analysis and monotone operators theory in {H}ilbert
  spaces}.
\newblock Springer-Verlag, 2nd edition, 2017.

\bibitem{Beck2009}
A.~Beck and M.~Teboulle.
\newblock {A} fast iterative shrinkage-thresholding agorithm for linear inverse
  problems.
\newblock {\em SIAM J. Imaging Sci.}, 2(1):183--202, 2009.

\bibitem{Becker2011a}
S.~Becker, E.~J. Cand\`{e}s, and M.~Grant.
\newblock Templates for convex cone problems with applications to sparse signal
  recovery.
\newblock {\em Math. Program. Compt.}, 3(3):165--218, 2011.

\bibitem{Becker2012a}
S.~Becker and M.J. Fadili.
\newblock A quasi-{N}ewton proximal splitting method.
\newblock In {\em Proceedings of Neutral Information Processing Systems
  Foundation (NIPS)}, 2012.

\bibitem{bollapragada2016exact}
Raghu Bollapragada, Richard Byrd, and Jorge Nocedal.
\newblock Exact and inexact subsampled {N}ewton methods for optimization.
\newblock {\em Arxiv preprint:1609.08502}, 2016.

\bibitem{Bonnans1994a}
J.F. Bonnans.
\newblock {L}ocal {A}nalysis of {N}ewton-{T}ype {M}ethods for {V}ariational
  {I}nequalities and {N}onlinear {P}rogramming.
\newblock {\em Appl. Math. Optim}, 29:161--186, 1994.

\bibitem{borodin2004can}
Allan Borodin, Ran El-Yaniv, and Vincent Gogan.
\newblock Can we learn to beat the best stock.
\newblock {\em J. Artif. Intell. Res.(JAIR)}, 21:579--594, 2004.

\bibitem{Boyd2004}
S.~Boyd and L.~Vandenberghe.
\newblock {\em {C}onvex {O}ptimization}.
\newblock University {P}ress, Cambridge, 2004.

\bibitem{byrd2016stochastic}
Richard~H Byrd, SL~Hansen, Jorge Nocedal, and Yoram Singer.
\newblock A stochastic quasi-{N}ewton method for large-scale optimization.
\newblock {\em SIAM J. Optim.}, 26(2):1008--1031, 2016.

\bibitem{CC01a}
Chih-Chung Chang and Chih-Jen Lin.
\newblock {LIBSVM}: A library for support vector machines.
\newblock {\em ACM Transactions on Intelligent Systems and Technology},
  2:27:1--27:27, 2011.
\newblock Software available at \url{http://www.csie.ntu.edu.tw/~cjlin/libsvm}.

\bibitem{chen2000balancing}
Tzu-Yi Chen and James~W Demmel.
\newblock Balancing sparse matrices for computing eigenvalues.
\newblock {\em Linear algebra and its applications}, 309(1-3):261--287, 2000.

\bibitem{cohen2017matrix}
M.~Cohen, A.~Madry, D.~Tsipras, and A.~Vladu.
\newblock Matrix scaling and balancing via box constrained {N}ewton's method
  and interior-point methods.
\newblock {\em The 58th Annual IEEE Symposium on Foundations of Computer
  Science}, pages 902--913, 2017.

\bibitem{Dennis1974}
J.E. Dennis and J.~J. Mor{\'{e}}.
\newblock A characterisation of superlinear convergence and its application to
  quasi--{N}ewton methods.
\newblock {\em Math. Comput.}, 28:549--560, 1974.

\bibitem{Deuflhard2006}
P.~Deuflhard.
\newblock {\em {N}ewton {M}ethods for {N}onlinear {P}roblems -- {A}ffine
  {I}nvariance and {A}daptative {A}lgorithms}, volume~35 of {\em Springer
  Series in Computational Mathematics}.
\newblock Springer, 2nd edition, 2006.

\bibitem{Dolan2002}
E.D. Dolan and J.J. Mor\'e.
\newblock Benchmarking optimization software with performance profiles.
\newblock {\em Math. Program.}, 91:201--213, 2002.

\bibitem{erdogdu2015convergence}
Murat~A Erdogdu and Andrea Montanari.
\newblock Convergence rates of sub-sampled {N}ewton methods.
\newblock In {\em Advances in Neural Information Processing Systems}, pages
  3052--3060, 2015.

\bibitem{fercoq2016restarting}
O.~Fercoq and Z.~Qu.
\newblock Restarting accelerated gradient methods with a rough strong convexity
  estimate.
\newblock {\em Preprint: arXiv:1609.07358}, pages 1--23, 2016.

\bibitem{Frank1956}
M.~Frank and P.~Wolfe.
\newblock {A}n algorithm for quadratic programming.
\newblock {\em Naval Research Logistics Quarterly}, 3:95--110, 1956.

\bibitem{friedlander2016efficient}
M.~Friedlander and G.~Goh.
\newblock Efficient evaluation of scaled proximal operators.
\newblock {\em Electron. Trans. Numer. Anal}, 46:1--22, 2017.

\bibitem{gao2016quasi}
W.~Gao and D.~Goldfarb.
\newblock {Q}uasi-{N}ewton {M}ethods: {S}uperlinear {C}onvergence {W}ithout
  {L}inesearch for {S}elf-{C}oncordant {F}unctions.
\newblock {\em Preprint: arXiv:1612.06965}, pages 1--16, 2016.

\bibitem{Giselsson2014}
P.~Giselsson and S.~Boyd.
\newblock {M}onotonicity and {R}estart in {F}ast {G}radient {M}ethods.
\newblock In {\em IEEE Conference on Decision and Control}, pages 5058--5063,
  Los Angeles, USA, December 2014. CDC.

\bibitem{Goel2006}
V.~Goel and I.E. Grossmann.
\newblock {A} {C}lass of stochastic programs with decision dependent
  uncertainty.
\newblock {\em Math. Program.}, 108:355--394, 2006.

\bibitem{Grant2006}
M.~Grant, S.~Boyd, and Y.~Ye.
\newblock Disciplined convex programming.
\newblock In L.~Liberti and N.~Maculan, editors, {\em Global Optimization: From
  Theory to Implementation}, Nonconvex Optimization and its Applications, pages
  155--210. Springer, 2006.

\bibitem{halko2009finding}
N.~Halko, P.-G. Martinsson, and J.~A. Tropp.
\newblock Finding structure with randomness: {S}tochastic algorithms for
  constructing approximate matrix decompositions.
\newblock 2009.

\bibitem{hazan2006efficient}
Elad Hazan and Sanjeev Arora.
\newblock {\em Efficient algorithms for online convex optimization and their
  applications}.
\newblock Princeton University, 2006.

\bibitem{He2016}
N.~He, Z.~Harchaoui, Y.~Wang, and L.~Song.
\newblock Fast and simple optimization for {P}oisson likelihood models.
\newblock {\em arXiv preprint arXiv:1608.01264}, 2016.

\bibitem{Hosmer2005}
D.~W. Hosmer and S.~Lemeshow.
\newblock {\em {A}pplied {L}ogistic {R}egression}.
\newblock John Wiley and Sons, 2005.

\bibitem{Jaggi2013}
M.~Jaggi.
\newblock {R}evisiting {F}rank-{W}olfe: {P}rojection-{F}ree {S}parse {C}onvex
  {O}ptimization.
\newblock {\em JMLR W\&CP}, 28(1):427--435, 2013.

\bibitem{Krishnapuram2005}
B.~Krishnapuram, M.~Figueiredo, L.~Carin, and H.~Hartemink.
\newblock {S}parse {M}ultinomial {L}ogistic {R}egression: {F}ast {A}lgorithms
  and {G}eneralization {B}ounds.
\newblock {\em IEEE Trans. Pattern Analysis and Machine Intelligence (PAMI)},
  27:957--968, 2005.

\bibitem{Kyrillidis2014}
A.~Kyrillidis, R.~Karimi, Q.~Tran-Dinh, and V.~Cevher.
\newblock Scalable sparse covariance estimation via self-concordance.
\newblock In {\em Proc. of the 28th AAAI Conference on Artificial
  Intelligence}, pages 1946--1952, 2014.

\bibitem{lafferty2002boosting}
Guy Lebanon~John Lafferty.
\newblock Boosting and maximum likelihood for exponential models.
\newblock {\em Advances in Neural Information Processing Systems (NIPS)},
  14:447, 2002.

\bibitem{Lee2014}
J.D. Lee, Y.~Sun, and M.A. Saunders.
\newblock Proximal {N}ewton-type methods for convex optimization.
\newblock {\em SIAM J. Optim.}, 24(3):1420--1443, 2014.

\bibitem{Lu2016a}
Zh. Lu.
\newblock Randomized block proximal damped {N}ewton method for composite
  self-concordant minimization.
\newblock {\em SIAM J. Optim.}, 27(3):1910--1942, 2017.

\bibitem{marron2007distance}
JS~Marron, Michael~J Todd, and Jeongyoun Ahn.
\newblock Distance-weighted discrimination.
\newblock {\em Journal of the American Statistical Association},
  102(480):1267--1271, 2007.

\bibitem{mccullagh1989generalized}
P.~McCullagh and J.A. Nelder.
\newblock {\em Generalized linear models}, volume~37.
\newblock CRC press, 1989.

\bibitem{monteiro2015hybrid}
Renato~DC Monteiro, Mauricio~R Sicre, and Benar~Fux Svaiter.
\newblock A hybrid proximal extragradient self-concordant primal barrier method
  for monotone variational inequalities.
\newblock {\em SIAM Journal on Optimization}, 25(4):1965--1996, 2015.

\bibitem{nelder1972generalized}
John~A Nelder and R~Jacob Baker.
\newblock Generalized linear models.
\newblock {\em Encyclopedia of statistical sciences}, 1972.

\bibitem{Nesterov2004}
Y.~Nesterov.
\newblock {\em {I}ntroductory lectures on convex optimization: {A} basic
  course}, volume~87 of {\em Applied Optimization}.
\newblock Kluwer Academic Publishers, 2004.

\bibitem{Nesterov2005c}
Y.~Nesterov.
\newblock Smooth minimization of non-smooth functions.
\newblock {\em Math. Program.}, 103(1):127--152, 2005.

\bibitem{Nesterov2006b}
Y.~Nesterov.
\newblock Cubic regularization of {N}ewton's method for convex problems with
  constraints.
\newblock CORE Discussion Paper 2006/39, Catholic University of Louvain (UCL) -
  Center for Operations Research and Econometrics (CORE), 2006.

\bibitem{Nesterov2008b}
Y.~Nesterov.
\newblock Accelerating the cubic regularization of {N}ewton’s method on
  convex problems.
\newblock {\em Math. Program.}, 112:159--181, 2008.

\bibitem{Nesterov2007}
Y.~Nesterov.
\newblock Gradient methods for minimizing composite objective function.
\newblock {\em Math. Program.}, 140(1):125--161, 2013.

\bibitem{Nesterov1994}
Y.~Nesterov and A.~Nemirovski.
\newblock {\em Interior-point Polynomial Algorithms in Convex Programming}.
\newblock Society for Industrial Mathematics, 1994.

\bibitem{Nesterov2006}
Y.~Nesterov and B.T. Polyak.
\newblock {C}ubic regularization of {N}ewton's method and its global
  performance.
\newblock {\em Mathmatical Programming}, 112(1):177--205, 2006.

\bibitem{Nocedal2006}
J.~Nocedal and S.J. Wright.
\newblock {\em {N}umerical {O}ptimization}.
\newblock Springer Series in Operations Research and Financial Engineering.
  Springer, 2 edition, 2006.

\bibitem{Odonoghue2012}
B.~O'Donoghue and E.~Candes.
\newblock {Adaptive Restart for Accelerated Gradient Schemes}.
\newblock {\em Found. Comput. Math.}, 15:715--732, 2015.

\bibitem{odor2016frank}
G.~Odor, Y.-H. Li, A.~Yurtsever, Y.-P. Hsieh, Q.~Tran-Dinh, M.~El-Halabi, and
  V.~Cevher.
\newblock {F}rank-{W}olfe works for non-lipschitz continuous gradient
  objectives: {S}calable poisson phase retrieval.
\newblock In {\em 2016 IEEE International Conference on Acoustics, Speech and
  Signal Processing (ICASSP)}, pages 6230--6234. IEEE, 2016.

\bibitem{Ortega2000}
James~M. Ortega and Werner~C. Rheinboldt.
\newblock {\em Iterative solution of nonlinear equations in several variables}.
\newblock Society for Industrial and Applied Mathematics, Philadelphia, PA,
  USA, 2000.

\bibitem{Parikh2013}
N.~Parikh and S.~Boyd.
\newblock Proximal algorithms.
\newblock {\em Foundations and Trends in Optimization}, 1(3):123--231, 2013.

\bibitem{parlett1982methods}
BN~Parlett and TL~Landis.
\newblock Methods for scaling to doubly stochastic form.
\newblock {\em Linear Algebra and its Applications}, 48:53--79, 1982.

\bibitem{peng2009self}
Jiming Peng, Cornelis Roos, and Tam{\'a}s Terlaky.
\newblock {\em Self-regularity: {A} new paradigm for primal-dual interior-point
  algorithms}.
\newblock Princeton University Press, 2009.

\bibitem{pilanci2015newton}
M.~Pilanci and M.J. Wainwright.
\newblock {N}ewton sketch: A linear-time optimization algorithm with
  linear-quadratic convergence.
\newblock {\em Arxiv preprint:1505.02250}, 2015.

\bibitem{polyak2009regularized}
R.~A. Polyak.
\newblock Regularized {N}ewton method for unconstrained convex optimization.
\newblock {\em Math. Program.}, 120(1):125--145, 2009.

\bibitem{Robinson1980}
S.~M. Robinson.
\newblock {S}trongly {R}egular {G}eneralized {E}quations.
\newblock {\em Math. Opers. Res., Vol. 5, No. 1 (Feb., 1980), pp. 43-62},
  5:43--62, 1980.

\bibitem{Roosta-Khorasani2016}
Farbod Roosta-Khorasani and Michael~W Mahoney.
\newblock Sub-sampled {N}ewton methods i: globally convergent algorithms.
\newblock {\em Arxiv preprint:1601.04737}, 2016.

\bibitem{roosta2016sub}
Farbod Roosta-Khorasani and Michael~W Mahoney.
\newblock Sub-sampled {N}ewton methods ii: Local convergence rates.
\newblock {\em Arxiv preprint:1601.04738}, 2016.

\bibitem{ryu2014stochastic}
Ernest~K Ryu and Stephen Boyd.
\newblock Stochastic proximal iteration: a non-asymptotic improvement upon
  stochastic gradient descent.
\newblock {\em Author website, early draft}, 2014.

\bibitem{Toh2010}
K.-Ch. Toh, M.J. Todd, and R.H. T\"{u}t\"{u}nc\"{u}.
\newblock On the implementation and usage of {SDPT}3 -- a {M}atlab software
  package for semidefinite-quadratic-linear programming.
\newblock Tech. Report~4, NUS Singapore, 2010.

\bibitem{Tran-Dinh2013b}
Q.~Tran-Dinh, A.~Kyrillidis, and V.~Cevher.
\newblock A proximal {N}ewton framework for composite minimization: {G}raph
  learning without {C}holesky decompositions and matrix inversions.
\newblock {\em JMLR W\&CP}, 28(2):271--279, 2013.

\bibitem{Tran-Dinh2013a}
Q.~Tran-Dinh, A.~Kyrillidis, and V.~Cevher.
\newblock Composite self-concordant minimization.
\newblock {\em J. Mach. Learn. Res.}, 15:374--416, 2015.

\bibitem{TranDinh2014d}
Q.~Tran-Dinh, Y.-H. Li, and V.~Cevher.
\newblock Composite convex minimization involving self-concordant-like cost
  functions.
\newblock In T.~Pham~Dinh H.~A. Le-Thi and N.~T. Nguyen, editors, {\em
  Modelling, Computation and Optimization in Information Systems and Management
  Sciences}, pages 155--168. Springer-Verlag, 2015.

\bibitem{Tran-Dinh2013}
Q.~Tran-Dinh, I.~Necoara, and M.~Diehl.
\newblock A dual decomposition algorithm for separable nonconvex optimization
  using the penalty function framework.
\newblock In {\em Proceedings of the conference on Decision and Control (CDC)},
  pages 2372--2377, Florence, Italy, December 2013.

\bibitem{TranDinh2012e}
Q.~Tran-Dinh, I.~Necoara, and M.~Diehl.
\newblock {P}ath-following gradient-based decomposition algorithms for
  separable convex optimization.
\newblock {\em J. Global Optim.}, 59(1):59--80, 2014.

\bibitem{TranDinh2012c}
Q.~Tran-Dinh, I.~Necoara, C.~Savorgnan, and M.~Diehl.
\newblock {A}n inexact perturbed path-following method for {L}agrangian
  decomposition in large-scale separable convex optimization.
\newblock {\em SIAM J. Optim.}, 23(1):95--125, 2013.

\bibitem{TranDinh2016c}
Q.~Tran-Dinh, T.~Sun, and S.~Lu.
\newblock Self-concordant inclusions: A unified framework for path-following
  generalized {N}ewton-type algorithms.
\newblock {\em Math. Program.}, pages 1--51, 2018.

\bibitem{vapnik1998statistical}
V.~N. Vapnik and V.~Vapnik.
\newblock {\em Statistical learning theory}, volume~1.
\newblock Wiley New York, 1998.

\bibitem{Verscheure2009}
D.~Verscheure, B.~Demeulenaere, J.~Swevers, J.~De Schutter, and M.~Diehl.
\newblock {T}ime-{O}ptimal {P}ath {T}racking for {R}obots: a {C}onvex
  {O}ptimization {A}pproach.
\newblock {\em IEEE Trans. Autom. Control}, 54:2318--2327, 2009.

\bibitem{Yamashita2003}
M.~Yamashita, K.~Fujisawa, and M.~Kojima.
\newblock Implementation and evaluation of {SDPA} 6.0 ({S}emi{D}efinite
  {P}rogramming {A}lgorithm 6.0).
\newblock {\em Optim. Method Softw.}, 18:491--505, 2003.

\bibitem{yang2016rsg}
Tianbao Yang and Qihang Lin.
\newblock {RSG}: Beating {SGD} without smoothness and/or strong convexity.
\newblock {\em CoRR abs/1512.03107}, 2016.

\bibitem{zhang2015disco}
Y.~Zhang and X.~Lin.
\newblock {DiSCO}: {D}istributed optimization for self-concordant empirical
  loss.
\newblock In {\em Proceedings of The 32th International Conference on Machine
  Learning}, pages 362--370, 2015.

\end{thebibliography}


\end{document}